\documentclass[11pt]{article}
\usepackage[margin = 1in]{geometry}
\usepackage[affil-it]{authblk}
\usepackage{lipsum}
\usepackage{amsfonts}
\usepackage{graphicx}
\graphicspath{ {images/} }
\usepackage{epstopdf}
\usepackage{dsfont}
\usepackage{mathtools}
\usepackage{empheq}
\usepackage{amsfonts}
\usepackage{amsgen,amsthm,amsmath,amstext,amsbsy,amsopn,amssymb,subfigure,stmaryrd,mathabx,mathrsfs}
\usepackage{comment}
\usepackage[font=footnotesize,labelfont=bf]{caption}

\usepackage{blkarray}

\usepackage{url}
\usepackage{longtable}
\usepackage{mathtools}
\usepackage{multirow}
\usepackage{array}

\usepackage{enumitem}
\usepackage{hyperref}
\usepackage{natbib}
\usepackage{booktabs}
\usepackage{algorithm}
\usepackage{algpseudocode}
\usepackage{xcolor}
\usepackage[utf8]{inputenc} 
\usepackage[T1]{fontenc}

\ifpdf
  \DeclareGraphicsExtensions{.eps,.pdf,.png,.jpg}
\else
  \DeclareGraphicsExtensions{.eps}
\fi
\newcommand{\cH}{\mathcal{H}}

\newcommand{\cD}{\mathcal{D}}

\newcommand{\cE}{\mathcal{E}}

\newcommand{\cI}{\mathcal{I}}

\newcommand{\htheta}{\widehat{\theta}}

\newcommand{\bbP}{\mathbb{P}}

\newcommand{\bbR}{\mathbb{R}}

\renewcommand{\exp}{{\rm{exp}}}
\newcommand{\TV}{{\sf TV}}

\renewcommand{\H}{{\rm{H}}}

\newcommand{\argmin}{\mathop{\rm arg\min}}
\newcommand{\argmax}{\mathop{\rm arg\max}}

\newcommand{\indi}{{\mathds{1}}}

\newcommand{\CI}{{\textnormal{CI}}}
\newcommand{\CS}{{\textnormal{CS}}}
\newcommand{\wh}{\widehat}

\newtheorem{Theorem}{Theorem}

\newtheorem{Lemma}{Lemma}
\newtheorem{Remark}{Remark}
\newtheorem{Corollary}{Corollary}

\newtheorem{Proposition}{Proposition}

\DeclareMathAlphabet\mathbfcal{OMS}{cmsy}{b}{n}

\hypersetup{
    colorlinks=true,
    linkcolor=blue,
    filecolor=blue,      
    urlcolor=cyan,
    citecolor=blue
}

\makeatletter
\newcommand*{\rom}[1]{\expandafter\@slowromancap\romannumeral #1@}
\makeatother

\allowdisplaybreaks

\begin{document}
	\title{Adaptive Robust Confidence Intervals}
	
	\author{Yuetian Luo$^1$ ~ and ~ Chao Gao$^2$ }
	\affil{University of Chicago}
	\date{}
	\maketitle

	\footnotetext[1]{Email: \texttt{yuetian@uchicago.edu}.}
	\footnotetext[2]{Email: \texttt{chaogao@uchicago.edu}. The research of CG is supported in part by NSF
Grants ECCS-2216912 and DMS-2310769, NSF Career Award DMS-1847590, and an Alfred Sloan fellowship. }

\begin{abstract}
This paper studies the construction of adaptive confidence intervals under Huber's contamination model when the contamination proportion is unknown. For the robust confidence interval of a Gaussian mean, we show that the optimal length of an adaptive interval must be exponentially wider than that of a non-adaptive one. An optimal construction is achieved through simultaneous uncertainty quantification of quantiles at all levels. The results are further extended beyond the Gaussian location model by addressing a general family of robust hypothesis testing. In contrast to adaptive robust estimation, our findings reveal that the optimal length of an adaptive robust confidence interval critically depends on the distribution's shape. 
\end{abstract}

%{\bf Keywords}: 
%\tableofcontents

\begin{sloppypar}
%%%%%%%%%%%%%%%%%%%%%%%%%%%%%%%%%%%%%%%%%
\section{Introduction} \label{sec: introduction}
%%%%%%%%%%%%%%%%%%%%%%%%%%%%%%%%%%%%%%%%%

Given $X_1, \ldots, X_n \overset{iid}\sim N(\theta,1)$, one can compute the classical confidence interval
\begin{equation}
\left[\bar{X}-\frac{1.96}{\sqrt{n}},\bar{X}+\frac{1.96}{\sqrt{n}}\right], \label{eq:101}
\end{equation} where $\bar{X} = (\sum_{i=1}^n X_i) /n$ and it has coverage probability at least $95\%$. Moreover, the interval (\ref{eq:101}) has the smallest length among all that have the coverage property \citep{neyman1937outline}. This paper will tackle the question of extending the classical solution (\ref{eq:101}) to modern settings where data may have contamination.

We will consider a robust estimation setting proposed by \cite{huber1964robust}, where
\begin{equation}
X_1, \ldots, X_n \overset{iid}\sim P_{\epsilon,\theta,Q}= (1-\epsilon)N(\theta,1)+\epsilon Q. \label{eq:additive-contamination}
\end{equation}
This is known as Huber's contamination model. In this setting, there is an $\epsilon$ fraction of data points that are generated by an arbitrary process $Q$. It is important to note that in Huber's setting, the contamination distribution $Q$ does not have any restriction. It may or may not be close to $N(\theta,1)$. Our goal in this paper is to construct a confidence interval for $\theta$ with coverage property holding uniformly over all $Q$ under the data generating process (\ref{eq:additive-contamination}).

A standard strategy of confidence interval construction is to rearrange a high-probability error bound of an optimal estimator. For example, the classical confidence interval (\ref{eq:101}) can be obtained by the high-probability error bound $\inf_{\theta}P_{\theta}\left(|\bar{X}-\theta|\leq \frac{1.96}{\sqrt{n}}\right)\geq 0.95$. Similarly, when data is generated from Huber's contamination model, the sample median satisfies
\begin{equation}
\inf_{\theta,Q}P_{\epsilon,\theta,Q}\left(|\wh{\theta}_{\rm median}-\theta|\leq \sqrt{2.1\pi}\left(\frac{1.36}{\sqrt{n}}+\epsilon\right)\right)\geq 0.95. \label{eq:median-error-bd}
\end{equation}
See \cite{chen2018robust} or Chapter 1 of \cite{diakonikolas2023algorithmic}. The explicit constants above can be computed when $\frac{1}{\sqrt{n}}+\epsilon$ is sufficiently small. %\footnote{The high-probability bound proved by \cite{chen2018robust} is $C\left(\frac{1}{\sqrt{n}}+\epsilon\right)$ for some constant $C>0$. The explicit constants in (\ref{eq:median-error-bd}) can be calculated by assuming that $\frac{1}{\sqrt{n}}+\epsilon$ is sufficiently small.} 
Moreover, it is also known that $\frac{1}{\sqrt{n}}+\epsilon$ is the minimax rate \citep{chen2018robust}, which means that the additional additive error term $O(\epsilon)$ in Huber's setting is a necessary consequence of contamination. The high-probability error bound (\ref{eq:median-error-bd}) immediately suggests a confidence interval centering at the sample median,
\begin{equation}
\left[\wh{\theta}_{\rm median}-\sqrt{2.1\pi}\left(\frac{1.36}{\sqrt{n}}+\epsilon\right),\wh{\theta}_{\rm median}+\sqrt{2.1\pi}\left(\frac{1.36}{\sqrt{n}}+\epsilon\right)\right]. \label{eq:median-RCI}
\end{equation}
The coverage of (\ref{eq:median-RCI}) is guaranteed by (\ref{eq:median-error-bd}), and the length of (\ref{eq:median-RCI}), which is of order $\frac{1}{\sqrt{n}}+\epsilon$, is rate optimal given the estimation error lower bound \citep{chen2018robust}.

Despite its length optimality, a critical issue of the median confidence interval (\ref{eq:median-RCI}) is its dependence on the contamination proportion $\epsilon$. In other words, since the error of the sample median now depends on $\epsilon$ under Huber's model, the uncertainty quantification requires the knowledge of $\epsilon$. This requirement is certainly unrealistic in practice. In addition, since the contamination distribution $Q$ does not have any restriction, the parameter $\epsilon$ is not identifiable, which rules out any attempt that tries to estimate the error of the median. One naive solution that does not require the knowledge of $\epsilon$ is a conservative interval,
\begin{equation}
\left[\wh{\theta}_{\rm median}-R,\wh{\theta}_{\rm median}+R\right], \label{eq:conservative}
\end{equation}
for some constant $R$ so that (\ref{eq:median-RCI}) is a subset of (\ref{eq:conservative}) for all $\epsilon$. Obviously, the naive strategy (\ref{eq:conservative}) is not appealing, since the length of the interval stays the same when the sample size increases or the contamination proportion decreases.

Confidence interval construction under Huber's contamination model (\ref{eq:additive-contamination}) with unknown $\epsilon$ is a well-recognized open problem in the literature \citep{bateni2020confidence,wang2023huber}. This is in contrast to the estimation problem, where rate optimal estimators can usually be constructed without the knowledge of $\epsilon$; the sample median (\ref{eq:median-error-bd}) is a leading example. In general, adaptive robust estimation can be achieved in most settings by applying Lepski's method \citep{lepskii1991problem,lepskii1992asymptotically,lepskii1993asymptotically,jain2022robust}, but nearly nothing is known about adaptive robust confidence intervals.

Our main result of the paper is a construction of an adaptive confidence interval when the contamination proportion $\epsilon$ is unknown. Moreover, we prove a matching length lower bound, which precisely characterizes the adaptation cost of not knowing $\epsilon$. It turns out that the ignorance of $\epsilon$ results in a severe adaptation cost. In particular, the optimal length of an adaptive confidence interval is of much greater order than the rate $\frac{1}{\sqrt{n}}+\epsilon$ achieved by the median interval (\ref{eq:median-RCI}), though it is still strictly better than the conservative strategy (\ref{eq:conservative}).

A key step leading to the optimal construction is the understanding of a family of robust hypothesis testing. We establish an equivalence between robust testing and the construction of the adaptive robust confidence interval, in the sense that a solution to one problem can be turned into that for the other. This equivalence also allows us to extend the results to general location models beyond Gaussianity. An intriguing phenomenon is thus revealed: while the minimax rate of location estimation usually does not depend on the distribution shape, the optimal adaptive confidence interval length is case by case; the rate is faster for a distribution with a lighter tail.

%%%%%%%%%%%%%%%%%%%%%%%%%%%%%%%%%%%
\subsection{Related Work} \label{sec:literature review}
%%%%%%%%%%%%%%%%%%%%%%%%%%%%%%%%%%%

% \paragraph{Robust Estimation}

{\it Robust Estimation. } Early works on robust estimation of a location parameter date back to \cite{huber1964robust}. Since then, various extensions have also been considered in   \cite{huber1967behavior,huber1968robust,bickel1965some,collins1976robust,crow1967robust,tukey1975mathematics,beran1977robust}. 
More recently, statistically optimal estimators have been thoroughly studied under various settings including distributions with heavy tail and adversarial contamination \citep{catoni2012challenging,lugosi2021robust,minsker2021robust,lee2022optimal,minasyan2023statistically,gupta2024beyond,bhatt2022minimax}. Computationally efficient algorithms in high-dimensional settings have been considered by \cite{diakonikolas2016robust,lai2016agnostic,cheng2019high,hopkins2020robust,prasad2020robust,zhu2022robust,dalalyan2022all,diakonikolas2024near,hopkins2020mean,cherapanamjeri2019fast,depersin2022robust}.
Substantial progress has also been made for covariance matrix estimation \citep{chen2018robust,cheng2019faster,minsker2018sub}, regression \citep{huber1973robust,gao2020robust,loh2017statistical,diakonikolas2019efficient,sun2020adaptive}, sparse estimation \citep{balakrishnan2017computationally}, clustering \citep{hopkins2018mixture,diakonikolas2020robustly,cherapanamjeri2020algorithms} and many other learning tasks. We refer readers to the textbooks \citep{huber2011robust,hampel2011robust,diakonikolas2023algorithmic} for a comprehensive review.

%\paragraph{Uncertainty Quantification}

{\it Uncertainty Quantification. } The most classic and common approach to perform uncertainty quantification is through asymptotic theory. In the context of robust statistics, asymptotic normality was proved for various M-estimators \citep{huber1964robust,huber1973robust,yohai1979asymptotic,portnoy1985asymptotic,el2013robust}. Subsampling methods such as bootstrap \citep{efron1994introduction} can also be used for uncertainty quantification, though the validity is usually established by asymptotic analysis as well \citep{politis1994large,hall2013bootstrap}. However, asymptotic analysis is not applicable to Huber's contamination model without additional assumptions on the contamination process.

In the setting of the contamination model, non-asymptotic constructions of confidence sets have been considered by \cite{bateni2020confidence,wang2023huber} when the knowledge of the contamination proportion is available. To the best of our knowledge, there is no method available in the literature that can deal with an unknown contamination proportion. Certain general frameworks of constructing confidence sets such as universal inference \citep{wasserman2020universal,park2023robust} and a convex hull method \citep{kuchibhotla2023hulc} offer finite sample guarantees with very little assumption on the data-generating process. However, when applied to Huber's contamination model, these methods either require knowledge of the contamination proportion or are designed to cover a different functional other than the parameter of interest. We refer the readers to Appendix A for a more detailed discussion.

%\paragraph{Adaptive Confidence Set}

 {\it Adaptive Confidence Set. } The construction of adaptive confidence sets is known to be a hard topic in nonparametric and high-dimensional inference. The difficulty is usually due to the dependence of uncertainty quantification on some unknown regularity parameter such as function smoothness or signal sparsity. In most settings where the models indexed by different regularity parameters are nested, impossibility of adaptation can be proved \citep{low1997nonparametric,cai2004adaptation}. Such a hardness result can often be explained by the connection of adaptive confidence sets with certain hypothesis testing problems. This intuition was first developed in nonparametric estimation problems \citep{cai2006adaptive,robins2006adaptive,juditsky2003nonparametric,baraud2004confidence,hoffmann2011adaptive,bull2013adaptive,carpentier2013honest}, and it has also been extended to high-dimensional sparse and low-rank regression problems \citep{nickl2013confidence,cai2017confidence,carpentier2015signal,carpentier2018adaptive}. Despite the impossibility results in various settings, explicit constructions of adaptive confidence sets can be done by weakening the coverage property \citep{wahba1983bayesian,genovese2008adaptive,cai2014adaptive} or by adding additional assumptions on signals such as self-similarity, polished tail, or shape constraints \citep{gine2010confidence,nickl2016sharp,chernozhukov2014anti,szabo2015frequentist,dumbgen2003optimal,cai2013adaptive,van2017adaptive}. We refer the readers to Chapter 8 of \cite{gine2015mathematical} for a comprehensive study on this topic.

%%%%%%%%%%%%%%%%%%%%%%%%%%%%%%%%%%%%%
\subsection{Paper Organization} \label{sec:organization}
%%%%%%%%%%%%%%%%%%%%%%%%%%%%%%%%%%%%%%
The mathematical formulation of the adaptive robust confidence interval is set up in Section \ref{sec:formulation}. The solution to the Gaussian location model will be given in Section \ref{sec:adaptiveRCI}. In Section \ref{sec:t}, we introduce a robust testing problem and establish its equivalence to the construction of an adaptive robust confidence interval, which then leads to the results of general location models covered in Section \ref{sec:ARCI-general-distribution}. Finally, we will discuss several extensions of the problem in Section \ref{sec:extension}. All technical proofs will be deferred to the appendices. % presented in Section \ref{sec:main-body-proof} and the rest of the proofs are deferred to appendices.

%%%%%%%%%%%%%%%%%%%%%%%%%%%%%%%%%%%%%
\subsection{Notation}\label{sec:notations}
%%%%%%%%%%%%%%%%%%%%%%%%%%%%%%%%%%%%%
Define $[n] = \{1,\ldots,n\}$ for any positive integer $n$. For any two sequences $\{a_n\}$ and $\{b_n\}$, we write $a_n \asymp b_n$ if there exist constants $c, C>0$ such that $ca_n \leq b_n\leq Ca_n$ for all $n$; $a_n \lesssim b_n$ means that $a_n \leq C b_n$ holds for some constant $C > 0$ independent of $n$. For any interval $[L,U]$, we implicitly assume that it is an empty set if $U < L$. Given any probability distribution with cumulative distribution function (CDF) $F(t)$, its $q$-quantile is denoted as $F^{-1}(q) = \inf \{ t \in \bbR: F(t) \geq q \}$. In particular, given $n$ data points $X_1, \ldots, X_n \in \bbR$,  we denote $X_{(1)} \leq \cdots \leq X_{(n)}$ as the order statistics. Then $F_n^{-1}(q) = X_{(\lceil n q \rceil)}$, where $F_n(t) = \frac{1}{n} \sum_{i \in [n]} \indi\{ X_i \leq t \} $ is the empirical CDF. The notation $P^{\otimes n}$ means the product distribution of $P$ with $n$ i.i.d. copies. The total variation distance between two distributions $P$ and $Q$ are defined by $\TV(P,Q)=\sup_B|P(B)-Q(B)|$. We use $\mathbb{E}$ and $\mathbb{P}$ for generic expectation and probability operators whenever the distribution is clear from the context.

\section{Problem Setting} \label{sec:formulation}

Given i.i.d. observations 
\begin{equation} \label{def:model}
	X_1, \ldots, X_n \overset{iid}\sim P_{\epsilon,\theta, Q}= (1- \epsilon ) P_{\theta} + \epsilon Q, 
\end{equation}
 with some parametric model $P_\theta$ and a pre-specified error probability $\alpha\in(0,1)$, our goal is to construct a confidence interval $\widehat{\CI}=\widehat{\CI}(\{X_i \}_{i=1}^n)$ that covers the parameter $\theta$ with probability at least $1-\alpha$, while optimizing its length quantified by Lebesgue measure. To formulate the setting where the contamination proportion $\epsilon$ is unknown, we will require the coverage probability lower bound to hold uniformly over $\epsilon\in\cE$, where $\cE$ is some subset of $[0,1]$. In particular, we call $\widehat{\CI}$ a $(1-\alpha)$-level adaptive robust confidence interval (ARCI) over $\cE$, if 
\begin{equation}
\inf_{\epsilon \in \cE} \inf_{\theta, Q} P_{\epsilon, \theta, Q} \left( \theta \in \widehat{\CI} \right) \geq 1- \alpha. \label{eq:uni-robust-coverage}
\end{equation}
The set of all $(1-\alpha)$-level ARCIs is defined by
$$\cI_{\alpha}(\cE) = \left\{ \widehat{\CI} = [l(\{X_i \}_{i=1}^n), u(\{X_i \}_{i=1}^n) ]: \inf_{\epsilon \in \cE} \inf_{\theta, Q} P_{\epsilon, \theta, Q} \left( \theta \in \widehat{\CI} \right) \geq 1- \alpha  \right\}.$$
We define the optimal length at some $\epsilon\in\cE$ over all $(1-\alpha)$-level ARCIs by
\begin{equation} 
	r_{\alpha}( \epsilon, \cE) = \inf \left\{r \geq 0: \inf_{\widehat{\CI} \in  \cI_{\alpha}(\cE) }\sup_{\theta, Q} P_{\epsilon, \theta, Q} \left( |\widehat{\CI}| \geq r \right) \leq \alpha \right\}, \label{eq:minimax-length-ARCI}
\end{equation}
where $|\widehat{\CI}|$ denotes the length of the confidence interval.
Note that the error probability of the length guarantee (\ref{eq:minimax-length-ARCI}) can be set as a more general value. We use the same $\alpha$ in $\cI_{\alpha}(\cE)$ for the convenience of presentation.

In this paper, we will mainly focus on the setting $\cE=[0, \epsilon_{\max}]$, where $\epsilon_{\max}$ serves as a conservative upper bound for $\epsilon$. Since Huber's model is nested in the sense that $\{P_{\epsilon,\theta, Q}: \theta, Q\}\subset \{P_{\epsilon_{\max},\theta, Q}: \theta,Q\}$ for any $\epsilon\leq \epsilon_{\max}$, the uniform coverage guarantee (\ref{eq:uni-robust-coverage}) is equivalent to
$\inf_{\theta, Q} P_{\epsilon_{\max}, \theta, Q} \left( \theta \in \widehat{\CI} \right) \geq 1- \alpha$,
which immediately implies
$$r_{\alpha}( \epsilon, [0,\epsilon_{\max}])=r_{\alpha}( \epsilon, \{\epsilon_{\max}\}).$$
To get an intuitive understanding of the above quantity, we note the following inequalities,
\begin{equation}
r_{\alpha}( \epsilon, \{\epsilon\}) \leq r_{\alpha}( \epsilon, \{\epsilon_{\max}\}) \leq r_{\alpha}( \epsilon_{\max}, \{\epsilon_{\max}\}), \label{eq:interpretation}
\end{equation}
which holds for all $\epsilon\in[0,\epsilon_{\max}]$.
\begin{enumerate}
\item The first inequality of (\ref{eq:interpretation}) is a lower bound for $r_{\alpha}( \epsilon, \{\epsilon_{\max}\})$. Note that $r_{\alpha}( \epsilon, \{\epsilon\})$ is the optimal length of a confidence interval in the setting with known $\epsilon$. For example, consider Gaussian location model with Huber contamination $P_{\epsilon,\theta, Q}= (1- \epsilon ) N(\theta,1) + \epsilon Q$. It can be shown that
\begin{equation}
r_{\alpha}( \epsilon, \{\epsilon\})\asymp \frac{1}{\sqrt{n}}+\epsilon,\label{eq:median-rate}
\end{equation}
and the confidence interval based on sample median (\ref{eq:median-error-bd}) achieves this optimal length. The situation $r_{\alpha}( \epsilon, \{\epsilon\}) \asymp r_{\alpha}( \epsilon, \{\epsilon_{\max}\})$ then means the same optimal length can be achieved without the knowledge of $\epsilon$.
\item The second inequality  of (\ref{eq:interpretation}) provides an upper bound for $r_{\alpha}( \epsilon, \{\epsilon_{\max}\})$. When $\epsilon_{\max}$ is of some constant order, $r_{\alpha}( \epsilon_{\max}, \{\epsilon_{\max}\})\asymp 1$ and can be achieved by the conservative strategy (\ref{eq:conservative}). The worst-case scenario $r_{\alpha}( \epsilon, \{\epsilon_{\max}\}) \asymp r_{\alpha}( \epsilon_{\max}, \{\epsilon_{\max}\})$ means that one cannot do better than the conservative interval when $\epsilon$ is unknown.
\end{enumerate}
This paper will systematically characterize the rate $r_{\alpha}( \epsilon, \{\epsilon_{\max}\})$ together with its comparisons with $r_{\alpha}( \epsilon, \{\epsilon\})$ and $r_{\alpha}( \epsilon_{\max}, \{\epsilon_{\max}\})$.

% for some $\epsilon_{\max}$ being a small constant (e.g. $\epsilon_{\max}=0.05$) and a extension of the case $\epsilon_{\max}$ 

%%%%%%%%%%%%%%%%%%%%%%%%%%%%%%%%%%%%
\section{Adaptation Cost in Gaussian Location Model} \label{sec:adaptiveRCI}
%%%%%%%%%%%%%%%%%%%%%%%%%%%%%%%%%%%%
In this section, we consider the Gaussian location model $P_{\theta}=N(\theta,1)$, and the goal is to construct an optimal ARCI over $\cE=[0,\epsilon_{\max}]$. We will set $\epsilon_{\max}= 0.05$ for the convenience of presentation. An extension to a smaller or larger value of $\epsilon_{\max}$ will be discussed in Sections \ref{sec:tozero} and \ref{sec:49}.

%%%%%%%%%%%%%%%%%%%%%%%%%%%%%%%%%%%%%
\subsection{A Lower Bound of Optimal Length} \label{sec:optimality-of-ARCI}
%%%%%%%%%%%%%%%%%%%%%%%%%%%%%%%%%%%%%

For any $\widehat{\CI} \in  \cI_{\alpha}(\cE)$, we first present a lower bound for its length.
\begin{Theorem} \label{th:lower-bound-Gaussian} Suppose $n \geq 100$. For any $\alpha \in (0,1/4)$, there exists some constant $c$ only depending on $\alpha$, such that
\begin{equation*}
	r_{\alpha}(\epsilon, [0,0.05]) \geq c \left( \frac{1}{\sqrt{\log n}} + \frac{1}{\sqrt{\log(1/\epsilon)}} \right),
\end{equation*}
for any $\epsilon\in[0,0.05]$.
\end{Theorem}
The lower bound should be compared with the rate (\ref{eq:median-rate}), which is the optimal interval length when $\epsilon$ is known. Both terms in (\ref{eq:median-rate}) become logarithmic, which indicates a severe cost of not knowing the value of $\epsilon$. In particular, a special case of the lower bound reads,
\begin{equation}
r_{\alpha}(0, [0,0.05]) \gtrsim \frac{1}{\sqrt{\log n}}.
\end{equation}
In other words, one cannot construct a confidence interval with length of order $\frac{1}{\sqrt{n}}$ even when $\epsilon=0$, meaning there is no outlier in the data. A statistician may have a data set without any contamination. However, as long as the statistician still suspects the possibility of contamination, exponentially more samples must be collected to achieve the same accuracy of uncertainty quantification as before.

On the other hand, the lower bound rate is faster than $r_{\alpha}(0.05, [0,0.05])\asymp 1$, which implies that it is still possible to construct an ARCI that is strictly better than the conservative strategy.

%%%%%%%%%%%%%%%%%%%%%%%%%%%%%%%%%%%%%
\subsection{An Optimal Adaptive Confidence Interval} \label{sec:optimal-ARCI-Gaussian}
%%%%%%%%%%%%%%%%%%%%%%%%%%%%%%%%%%%%%

Recall the definitions of the empirical CDF $F_n(\cdot)$ and the empirical quantile function $F_n^{-1}(\cdot)$. Define the standard Gaussian CDF by $\Phi(t)=\mathbb{P}(N(0,1)\leq t)$. Our construction of the ARCI is given by,
\begin{equation}\label{eq:adaptive-CI}
\begin{split}
	\widehat{\CI} = \Big[ &\max_{1.6 \leq t \leq  \Phi^{-1}\left(1 - \sqrt{ \log (2/\alpha)/(2n) } \right)  } \left(F_n^{-1}( 2(1 - \Phi(t) ) )  + t -  \frac{2}{t} \right), \\
	& \min_{ 1.6 \leq t \leq  \Phi^{-1}\left(1 - \sqrt{ \log (2/\alpha)/(2n) } \right) } \left( F_n^{-1}( 1- 2(1 - \Phi(t) ) )  - t  + \frac{2}{t} \right) \Big].
\end{split}
\end{equation}
The formula (\ref{eq:adaptive-CI}) does not depend on the value of the contamination proportion $\epsilon$, and can be computed efficiently using order statistics. The coverage and length guarantees of $\widehat{\CI}$ are given by the following theorem.

\begin{Theorem} \label{thm:upper-bound-general-Gaussian}
There exists some universal constant $C>0$, such that for any $\alpha\in(0,1)$ and $n\geq C\log(1/\alpha)$, the interval $\widehat{\CI}$ defined in (\ref{eq:adaptive-CI}) satisfies
	\begin{equation*}
		\begin{split}
			\textnormal{(coverage)} &\quad \inf_{ \epsilon \in [0, 0.05], \theta, Q} P_{\epsilon, \theta, Q}\left( \theta \in\widehat{\CI} \right) \geq 1-\alpha,\\
			\textnormal{(length)} & \quad\inf_{\epsilon \in [0, 0.05], \theta, Q } P_{\epsilon, \theta, Q}\left( |\widehat{\CI}| \leq C' \left( \frac{1}{\sqrt{\log n } } + \frac{1}{\sqrt{ \log(1/\epsilon) }}\right) \right) \geq 1-\alpha,
		\end{split}
	\end{equation*}
where $C'>0$ is some constant only depending on $\alpha$.
\end{Theorem} 
By the definition of $r_{\alpha}(\epsilon, [0,0.05])$ given in (\ref{eq:minimax-length-ARCI}), Theorem \ref{thm:upper-bound-general-Gaussian} implies that for any $\epsilon \in [0, 0.05]$,
$$r_{\alpha}(\epsilon, [0,0.05])  \lesssim \frac{1}{\sqrt{\log n}} + \frac{1}{\sqrt{\log(1/\epsilon)}},$$
which matches the lower bound result in Theorem \ref{th:lower-bound-Gaussian}.

\subsection{Uncertainty Quantification of Quantiles}

The construction of the optimal ARCI in (\ref{eq:adaptive-CI}) depends on the empirical quantile function with various levels. To understand how to arrive at the specific construction (\ref{eq:adaptive-CI}) from uncertainty quantification of quantiles, let us write (\ref{eq:adaptive-CI}) equivalently as
\begin{equation}
\widehat{\CI} = \bigcap_{1.6 \leq t \leq  \Phi^{-1}\left(1 - \sqrt{ \log (2/\alpha)/(2n) } \right) }\left[\wh{\theta}_L(t)-\frac{2}{t},\wh{\theta}_R(t)+\frac{2}{t}\right], \label{eq:intersect-quan}
\end{equation}
where
\begin{eqnarray}
\label{eq:theta-L} \wh{\theta}_L(t) &=& F_n^{-1}( 2(1 - \Phi(t) ) )  + t, \\
\label{eq:theta-R} \wh{\theta}_R(t) &=& F_n^{-1}( 1- 2(1 - \Phi(t) ) )  - t.
\end{eqnarray}
We note that both $\wh{\theta}_L(t)$ and $\wh{\theta}_R(t)$ can be viewed as estimators of the location parameter $\theta$ for a sufficiently large $t$. To see why this is the case, let us suppose there is no contamination, so that $F_n(\cdot)\approx \Phi(\cdot -\theta)$. Then,
$$\wh{\theta}_L(t) \approx \theta + \Phi^{-1}(2(1-\Phi(t)))+t\approx \theta,$$
where the property $\Phi^{-1}(2(1-\Phi(t)))+t\approx 0$ holds for a large $t$ because of the Gaussian tail decay of $1-\Phi(t)$. A similar argument also shows that $\wh{\theta}_R(t)\approx \theta$ when $t$ is sufficiently large.

A remarkable property of the two estimators $\wh{\theta}_L(t)$ and $\wh{\theta}_R(t)$ is that each one enjoys a one-sided uncertainty quantification even when the data has contamination and when $t$ is not necessarily large.
\begin{Lemma}\label{th:estimation-guarantee}
Simultaneously for all $t\in\left[1.6, \Phi^{-1}\left(1 - \sqrt{ \log (2/\alpha)/(2n) } \right)\right]$, we have
$$\wh{\theta}_L(t)-\theta\leq \frac{2}{t}\quad\text{and}\quad\wh{\theta}_R(t)-\theta\geq-\frac{2}{t},$$
with $P_{\epsilon, \theta, Q}$-probability at least $1-\alpha$, uniformly over all $\epsilon\in[0,0.05]$, all $\theta\in\mathbb{R}$, and all $Q$.
\end{Lemma}

According to the one-sided uncertainty quantification of $\wh{\theta}_L(t)$ and $\wh{\theta}_R(t)$, one can immediately use $\left[\wh{\theta}_L(t)-\frac{2}{t},\wh{\theta}_R(t)+\frac{2}{t}\right]$ as a robust confidence interval for $\theta$. Moreover, since the result of Lemma \ref{th:estimation-guarantee} holds simultaneously for all $t$ in a given range, it is natural to intersect all such sets to achieve the smallest possible length. This leads to the formula (\ref{eq:intersect-quan}), and Lemma \ref{th:estimation-guarantee} directly implies the coverage property in Theorem \ref{thm:upper-bound-general-Gaussian} for $\widehat{\CI}$.

With a particular choice of $t$, the one-sided uncertainty quantification of $\wh{\theta}_L(t)$ and $\wh{\theta}_R(t)$ can be strengthened to two-sided.
\begin{Lemma}\label{th:estimation-length-guarantee}
For any $\epsilon\in[0,0.05]$, define
\begin{equation} 
		t_{\epsilon} = \Phi^{-1}\left( 1- \epsilon - \sqrt{ \log (2/\alpha)/(2n) }  \right). \label{def:t-constant}
\end{equation}
Then, we have
\begin{equation}
\wh{\theta}_L(t_{\epsilon})-\theta\geq 0\quad\text{and}\quad\wh{\theta}_R(t_{\epsilon})-\theta\leq 0, \label{eq:the-other-side}
\end{equation}
with $P_{\epsilon, \theta, Q}$-probability at least $1-\alpha$, uniformly over all $\theta\in\mathbb{R}$ and all $Q$.
\end{Lemma}

Lemma \ref{th:estimation-length-guarantee} directly leads to the length property in Theorem \ref{thm:upper-bound-general-Gaussian}. The condition of Theorem \ref{thm:upper-bound-general-Gaussian} implies that $t_{\epsilon}\in\left[1.6, \Phi^{-1}\left(1 - \sqrt{ \log (2/\alpha)/(2n) } \right)\right]$ for all $\epsilon\in[0,0.05]$. By the formula (\ref{eq:intersect-quan}), we can bound the length by
\begin{equation}\label{ineq:length-guarantee}
	\begin{split}
		|\widehat{\CI}| &\leq \min_{1.6 \leq t \leq  \Phi^{-1}\left(1 - \sqrt{ \log (2/\alpha)/(2n) } \right)}\left(\wh{\theta}_R(t)+\frac{2}{t}-\left(\wh{\theta}_L(t)-\frac{2}{t}\right)\right) \\
&\leq \wh{\theta}_R(t_{\epsilon})+\frac{2}{t_{\epsilon}}-\left(\wh{\theta}_L(t_{\epsilon})-\frac{2}{t_{\epsilon}}\right) \\
&\leq \frac{4}{t_{\epsilon}},
	\end{split}
\end{equation}
where the last inequality above is by $\wh{\theta}_R(t_{\epsilon})\leq \theta$ and $-\wh{\theta}_L(t_{\epsilon})\leq -\theta$, a rearrangement of (\ref{eq:the-other-side}). We thus achieve the optimal length $\frac{4}{t_{\epsilon}}\asymp \frac{1}{\sqrt{\log n}} + \frac{1}{\sqrt{\log (1/\epsilon)}}$ (see Lemma 12 in Appendix G.)

The above argument explains the adaptivity property of the confidence interval $\widehat{\CI}$. Given some $\epsilon\in[0,0.05]$, the effective level of $t$ in (\ref{eq:intersect-quan}) that governs the length control of $\widehat{\CI}$ is the $t_{\epsilon}$ in (\ref{def:t-constant}). A smaller contamination proportion $\epsilon$ results in a larger value of $t_{\epsilon}$, which then leads to a shorter confidence interval.

\section{A Robust Testing Problem}\label{sec:t}

We will introduce a robust testing problem in this section, which serves as the key mathematical ingredient in the construction of ARCI. We will show that the solution to robust testing can be inverted to recover the formula (\ref{eq:adaptive-CI}), thus providing an alternative construction of ARCI in addition to the uncertainty quantification of quantiles. Moreover, the information-theoretic lower bound of robust testing also leads to the confidence interval length lower bound in Theorem \ref{th:lower-bound-Gaussian}.

\subsection{Equivalence between Robust Testing and Robust Confidence Interval}\label{sec:test-equi}

Let us start with the general contamination distribution $P_{\epsilon,\theta,Q}=(1-\epsilon)P_{\theta}+\epsilon Q$ with some parametric model $P_{\theta}$. With $n$ i.i.d. observations $X_1, \ldots,X_n$ drawn from some distribution $P$, we consider the following two pairs of robust hypothesis testing:
\begin{equation}\label{eq:test-def}
	\begin{split} &\cH(\theta, \theta- r, \epsilon): \begin{array}{l}
		H_{0}: P \in \left\{ P_{\epsilon_{\max}, \theta, Q}: Q \right\} \quad  \textnormal { v.s. } \quad 
		H_{1}: P \in \left\{ P_{\epsilon, \theta - r, Q}: Q \right\},
	\end{array}\\
	&\cH(\theta, \theta + r, \epsilon): \begin{array}{l}
		H_{0}: P \in \left\{ P_{\epsilon_{\max}, \theta, Q}: Q \right\}\quad  \textnormal { v.s. } \quad
		H_{1}: P \in \left\{ P_{\epsilon, \theta + r, Q}: Q\right\},
	\end{array}
	\end{split}
\end{equation}
where $\theta\in\mathbb{R}$, $r>0$ and $\epsilon\in[0,\epsilon_{\max}]$. Take $\cH(\theta, \theta- r, \epsilon)$ for example, a statistician not only needs to distinguish between the parameters $\theta$ and $\theta-r$, but also needs to tell the difference between two contamination levels $\epsilon_{\max}$ and $\epsilon$. Though various forms of robust testing have been considered in the literature \citep{huber1965robust,lecam1973convergence,birge1979theoreme,chen2016general,diakonikolas2017statistical,diakonikolas2021sample,canonne2023full}, the specific problem (\ref{eq:test-def}) is new to the best of our knowledge. We will show the equivalence between an optimal solution to (\ref{eq:test-def}) and an optimal construction of ARCI, which extends the classical duality between hypothesis testing and confidence interval \citep{neyman1937outline}.

\subsubsection{ARCI via Robust Testing}

For each $\epsilon\in[0,\epsilon_{\max}]$, let $\phi_{\theta,\theta-r_{\epsilon},\epsilon}$ (resp. $\phi_{\theta,\theta+r_{\epsilon},\epsilon}$) be a testing function that solves $\cH(\theta, \theta- r_{\epsilon}, \epsilon)$ (resp. $\cH(\theta, \theta+ r_{\epsilon}, \epsilon)$). In other words, $\phi_{\theta,\theta-r_{\epsilon},\epsilon}$ is a binary measurable function of data, and the null hypothesis is rejected whenever $\phi_{\theta,\theta-r_{\epsilon},\epsilon}=1$. The subscript of $r_{\epsilon}$ emphasizes the possible dependence of the separation parameter on $\epsilon$. For different values of $\epsilon\in[0,\epsilon_{\max}]$, the difficulties of the corresponding testing problems vary, which is reflected in the separation parameter $r_{\epsilon}$.

A confidence set inverting the testing procedures $\{\phi_{\theta, \theta \pm r_{ \epsilon }, \epsilon}: \epsilon \in [0,\epsilon_{\max}] \}$ is given by
\begin{equation}
\widehat{\CI} = \left\{\theta: \phi_{\theta,\theta-r_{\epsilon},\epsilon}=\phi_{\theta,\theta+r_{\epsilon},\epsilon}=0\text{ for all }\epsilon \in [0,\epsilon_{\max}]\right\}. \label{def:CI-based-on-test}
\end{equation}
Intuitively speaking, the set $\widehat{\CI}$ collects all $\theta$ such that the null hypotheses of $\{ \cH(\theta, \theta \pm r_{ \epsilon }, \epsilon): \epsilon \in [0,\epsilon_{\max}] \}$ are not rejected against robust local alternatives. We note that the formula (\ref{def:CI-based-on-test}) does not necessarily result in an interval, even though we still use the notation $\widehat{\CI}$. The exact form of $\widehat{\CI}$ will depend on the parametric model $P_{\theta}$ and the choice of the testing functions. Indeed, specific solutions that we consider in the paper for Gaussian and general location models will be intervals.

\begin{Proposition}\label{prop:test-to-CI}
Suppose for some $\alpha\in(0,1/4)$, the testing functions $\{\phi_{\theta, \theta \pm r_{ \epsilon }, \epsilon}: \epsilon \in [0,\epsilon_{\max}],\theta\in\mathbb{R} \}$ satisfy %\footnote{The notation $\phi_{\theta, \theta \pm r_{ \epsilon }, \epsilon}$ means that the two testing functions $\phi_{\theta, \theta - r_{ \epsilon }, \epsilon}$ and $\phi_{\theta, \theta + r_{ \epsilon }, \epsilon}$ individually satisfy the simultaneous Type-1 error control and the Type-2 error control. The same rule applies to similar situations throughout the paper.}
	\begin{equation*}
		\begin{split}
			\textnormal{(simultaneous Type-1 error)} &\quad \sup_{Q} P_{\epsilon_{\max},\theta,Q}\left(\sup_{\epsilon\in[0,\epsilon_{\max}]}\phi_{\theta,\theta\pm r_{\epsilon},\epsilon}=1\right)\leq \alpha,\\
			\textnormal{(Type-2 error)} &\quad \sup_{Q} P_{\epsilon,\theta\pm r_{\epsilon},Q}\left(\phi_{\theta,\theta\pm r_{\epsilon},\epsilon}=0\right)\leq \alpha,
		\end{split}
	\end{equation*}
for all $\theta\in\mathbb{R}$ and $\epsilon\in[0,\epsilon_{\max}]$. Then, the confidence set $\widehat{\CI}$ defined in (\ref{def:CI-based-on-test}) satisfies the coverage property,
$$\inf_{ \epsilon \in [0,\epsilon_{\max}], \theta, Q} P_{\epsilon, \theta, Q}\left( \theta \in\widehat{\CI} \right) \geq 1-2\alpha.$$
If, in addition, $\widehat{\CI}$ is an interval, then its length can also be controlled,
$$\inf_{ \epsilon \in [0,\epsilon_{\max}], \theta, Q} P_{\epsilon,\theta, Q}\left( |\widehat{\CI}| \leq 2 r_{ \epsilon } \right) \geq 1-4\alpha.$$
\end{Proposition}
We remark that the notation $\phi_{\theta, \theta \pm r_{ \epsilon }, \epsilon}$ above means that the two testing functions $\phi_{\theta, \theta - r_{ \epsilon }, \epsilon}$ and $\phi_{\theta, \theta + r_{ \epsilon }, \epsilon}$ individually satisfy the simultaneous Type-1 error control and the Type-2 error control. The same rule applies to similar situations throughout the paper.

\subsubsection{Robust Testing via ARCI}

In the reverse direction, we can construct testing functions from an ARCI,
\begin{equation}
\phi_{\theta,\theta-r_{\epsilon},\epsilon}=\phi_{\theta,\theta+r_{\epsilon},\epsilon}=\indi \left(  \theta  \notin  \widehat{\CI} \right),\label{eq:reverse-t}
\end{equation}
which rejects the null whenever the null parameter is not covered by the ARCI.

\begin{Proposition}\label{prop:CI-imply-test}
Suppose $\widehat{\CI}$ is an interval that satisfies
\begin{equation*}
		\begin{split}
			\textnormal{(coverage)} &\quad \inf_{ \epsilon \in [0,\epsilon_{\max}], \theta, Q} P_{\epsilon, \theta, Q}\left( \theta \in\widehat{\CI} \right) \geq 1-\alpha.
		\end{split}
	\end{equation*}
Then, the testing functions $\{\phi_{\theta, \theta \pm r_{ \epsilon }, \epsilon}: \epsilon \in [0,\epsilon_{\max}] \}$ defined in (\ref{eq:reverse-t}) satisfy
\begin{equation*}
		\begin{split}
			 & \sup_QP_{\epsilon_{\max},\theta,Q}\left(\sup_{\epsilon\in[0,\epsilon_{\max}]}\phi_{\theta,\theta\pm r_{\epsilon},\epsilon}=1\right)\leq \alpha,
		\end{split}
	\end{equation*}
for all $\theta\in\mathbb{R}$. Suppose in addition to the coverage property, $\widehat{\CI}$ also satisfies
\begin{equation*}
	\begin{split}
		\textnormal{(length)} & \quad \inf_{\theta, Q } P_{\epsilon, \theta, Q}\left( |\widehat{\CI}| \leq r_{\epsilon}   \right) \geq 1-\alpha,
	\end{split}
\end{equation*}
for some $\epsilon\in[0,\epsilon_{\max}]$. Then, for the same $\epsilon$,
\begin{equation*}
 \sup_QP_{\epsilon,\theta\pm r_{\epsilon},Q}\left(\phi_{\theta,\theta\pm r_{\epsilon},\epsilon}=0\right)\leq 2\alpha,
\end{equation*} for all $\theta\in\mathbb{R}$.
\end{Proposition}

We will apply Proposition \ref{prop:CI-imply-test} to prove the optimal length lower bound in Theorem \ref{th:lower-bound-Gaussian}. If we can show that the robust testing problem $\cH(\theta, \theta- r_{\epsilon}, \epsilon)$ or $\cH(\theta, \theta+ r_{\epsilon}, \epsilon)$ cannot be solved with small Type-1 and Type-2 errors, this will exclude the possibility of ARCI with length smaller than $r_{\epsilon}$.

\subsection{Gaussian Location Model: Upper Bound}\label{sec:Gup}

In the special case where $P_{\theta}=N(\theta,1)$ with $\epsilon_{\max}=0.05$, we will show that the formula (\ref{def:CI-based-on-test}) recovers the optimal ARCI (\ref{eq:adaptive-CI}) with appropriate testing functions.  Let us start with the testing problem $\cH(\theta, \theta - r_{ \epsilon }, \epsilon)$. We define the testing function
\begin{equation}
\phi_{\theta, \theta- r_{\epsilon}, \epsilon} = \indi \left\{ \frac{1}{n}\sum_{i=1}^n \indi \left \{ X_i - (\theta - r_{\epsilon} ) \geq t_{\epsilon} \right\} \leq 2(1-\Phi(t_{\epsilon})) \right\}. \label{eq:test-gaussian}
\end{equation}
It rejects the null hypothesis whenever the testing statistic $\frac{1}{n}\sum_{i=1}^n \indi \left \{ X_i - (\theta - r_{\epsilon} ) \geq t_{\epsilon} \right\}$ is small. To understand the property of this testing statistic, we can compute its expectation under both the null and alternative hypotheses. Given a null distribution $P_{\epsilon_{\max}, \theta, Q}$, we have
\begin{equation}
P_{\epsilon_{\max}, \theta, Q}\left(X-(\theta-r_{\epsilon})\geq t_{\epsilon}\right)\geq (1-\epsilon_{\max})\mathbb{P}\left(N(0,1)\geq t_{\epsilon}-r_{\epsilon}\right). \label{eq:exp-l-n}
\end{equation}
Under an alternative distribution $P_{\epsilon,\theta-r,Q}$, we have
\begin{equation}
P_{\epsilon,\theta-r_{\epsilon},Q}\left(X-(\theta-r_{\epsilon})\geq t_{\epsilon}\right)\leq \mathbb{P}\left(N(0,1)\geq t_{\epsilon}\right)+\epsilon. \label{eq:exp-u-a}
\end{equation}
Therefore, in order that the testing statistic $\frac{1}{n}\sum_{i=1}^n \indi \left \{ X_i - (\theta - r_{\epsilon} ) \geq t_{\epsilon} \right\}$ can separate the two hypotheses, it is required that the empirical version concentrates around the expectation under both null and alternative. Moreover, the expectation lower bound under the null (\ref{eq:exp-l-n}) needs to be greater than the expectation upper bound under the alternative (\ref{eq:exp-u-a}). It turns out that these requirements are satisfied with
\begin{equation}
r_{\epsilon}=\frac{2}{t_{\epsilon}}, \label{eq:r-gaussian}
\end{equation}
and $t_{\epsilon}$ given by the formula (\ref{def:t-constant}).

With a symmetric argument, the other testing problem $\cH(\theta, \theta + r_{ \epsilon }, \epsilon)$ can be solved by
\begin{equation}
\phi_{\theta, \theta+ r_{\epsilon}, \epsilon} = \indi \left\{ \frac{1}{n}\sum_{i=1}^n \indi \left \{ X_i - (\theta + r_{\epsilon} ) \leq -t_{\epsilon} \right\} < 2(1-\Phi(t_{\epsilon})) \right\}, \label{eq:test+gaussian}
\end{equation}
where the parameters $r_{\epsilon}$ and $t_{\epsilon}$ are also set by (\ref{eq:r-gaussian}) and (\ref{def:t-constant}).

\begin{Lemma}\label{prop:upper-bound-Gaussian}
There exists some universal constant $C>0$, such that for any $\alpha\in(0,1)$ and $n\geq C\log(1/\alpha)$, the testing functions $\{\phi_{\theta, \theta \pm r_{ \epsilon }, \epsilon}: \epsilon \in [0,0.05] ,\theta\in\mathbb{R}\}$ defined by (\ref{eq:test-gaussian}) and (\ref{eq:test+gaussian}) with parameters $r_{\epsilon}$ and $t_{\epsilon}$ set by (\ref{eq:r-gaussian}) and (\ref{def:t-constant}) satisfy
\begin{equation*}
		\begin{split}
			\textnormal{(simultaneous Type-1 error)} &\quad \sup_QP_{0.05,\theta,Q}\left(\sup_{\epsilon\in[0,0.05]}\phi_{\theta,\theta\pm r_{\epsilon},\epsilon}=1\right)\leq \alpha,\\
			\textnormal{(Type-2 error)} &\quad \sup_QP_{\epsilon,\theta\pm r_{\epsilon},Q}\left(\phi_{\theta,\theta\pm r_{\epsilon},\epsilon}=0\right)\leq \alpha,
		\end{split}
	\end{equation*}
for all $\theta\in\mathbb{R}$ and $\epsilon\in[0,0.05]$.
\end{Lemma}

The result shows that the robust testing problem (\ref{eq:test-def}) can be solved for the Gaussian location model under the separation $r_{\epsilon}=\frac{2}{t_{\epsilon}}\asymp \frac{1}{\sqrt{\log n}}+\frac{1}{\sqrt{\log(1/\epsilon)}}$.

Now, let us use the formula (\ref{def:CI-based-on-test}) to turn the testing functions into an ARCI. It is clear that $\widehat{\CI}=\bigcap_{\epsilon\in[0,0.05]}\left(\left\{\theta: \phi_{\theta,\theta- r_{\epsilon},\epsilon}=0\right\}\bigcap \left\{\theta: \phi_{\theta,\theta+ r_{\epsilon},\epsilon}=0\right\}\right)$. Note that
\begin{eqnarray}
 \label{eq:interval-step1} \left\{\theta: \phi_{\theta,\theta- r_{\epsilon},\epsilon}=0\right\} &=& \left\{\theta:  \frac{1}{n}\sum_{i=1}^n \indi \left \{ X_i - (\theta - r_{\epsilon} ) \geq t_{\epsilon} \right\} > 2(1-\Phi(t_{\epsilon}))\right\} \\
&=& \left\{\theta: \frac{1}{n}\sum_{i=1}^n \indi \left \{ X_i - (\theta - r_{\epsilon} ) < t_{\epsilon} \right\} < 1-2(1-\Phi(t_{\epsilon}))\right\} \\
&=& \left\{\theta: \theta\leq F_n^{-1}\left(1-2(1-\Phi(t_{\epsilon}))\right)-t_{\epsilon}+r_{\epsilon}\right\} \\
\label{eq:interval-step2} &=& \left(-\infty, F_n^{-1}\left(1-2(1-\Phi(t_{\epsilon}))\right)-t_{\epsilon}+\frac{2}{t_{\epsilon}}\right],
\end{eqnarray}
where the last equality replaces $r_{\epsilon}$ by $2/t_{\epsilon}$ according to (\ref{eq:r-gaussian}). A symmetric argument gives
$$\left\{\theta: \phi_{\theta,\theta + r_{\epsilon},\epsilon}=0\right\}=\left[F_n^{-1}( 2(1 - \Phi(t_{\epsilon}) ) )  + t_{\epsilon} -  \frac{2}{t_{\epsilon}},\infty\right).$$
Thus, the formula (\ref{def:CI-based-on-test}) is simplified to
$$\widehat{\CI}=\bigcap_{\epsilon\in[0,0.05]}\left[F_n^{-1}( 2(1 - \Phi(t_{\epsilon}) ) )  + t_{\epsilon} -  \frac{2}{t_{\epsilon}},F_n^{-1}\left(1-2(1-\Phi(t_{\epsilon}))\right)-t_{\epsilon}+\frac{2}{t_{\epsilon}}\right],$$
which is essentially the same as (\ref{eq:adaptive-CI}). Therefore, the combination of Proposition \ref{prop:test-to-CI} and Lemma \ref{prop:upper-bound-Gaussian} leads to an alternative proof of the coverage and length properties of the optimal ARCI from a robust testing perspective.

\subsection{Gaussian Location Model: Lower Bound}\label{sec:Glow}

The following result shows that $r_{\epsilon}\asymp \frac{1}{\sqrt{\log n}}+\frac{1}{\sqrt{\log(1/\epsilon)}}$ is the smallest order of separation for the robust testing problems $\{ \cH(\theta, \theta \pm r_{ \epsilon }, \epsilon): \epsilon \in [0,0.05] \}$ to have solutions with small Type-1 and Type-2 errors. Below this rate, the null and alternative can no longer be separated with $n$ samples.
\begin{Lemma} \label{prop:test-lower-bound-gaussian}
	For any $n \geq 100$, $\alpha \in (0,1)$ and $\epsilon \in [0, 0.05]$, there exists some constant $c > 0$ only depending on $\alpha$, such that as long as
	$$r \leq c \left(\frac{1}{\sqrt{\log n}}  + \frac{1}{\sqrt{\log (1/\epsilon)}} \right),$$
	we have
		\begin{equation}
			\inf_{Q_0, Q_1} \TV(P^{\otimes n}_{0.05, \theta, Q_0}, P^{\otimes n}_{\epsilon, \theta \pm r, Q_1}) \leq \alpha, \label{eq:gau-tv-match}
		\end{equation}
for all $\theta\in\mathbb{R}$.
\end{Lemma}
Since (\ref{eq:gau-tv-match}) implies that the sum of Type-1 and Type-2 errors must be greater than $1-\alpha$ for any test, Lemma \ref{prop:test-lower-bound-gaussian} leads to the rate optimality of the testing procedures (\ref{eq:test-gaussian}) and (\ref{eq:test+gaussian}). In fact, the lower bound of Lemma \ref{prop:test-lower-bound-gaussian} is for each individual level of $\epsilon$, let alone the simultaneous Type-1 error control established by Lemma \ref{prop:upper-bound-Gaussian}. Combined with Proposition \ref{prop:CI-imply-test}, Lemma \ref{prop:test-lower-bound-gaussian} immediately implies the optimal length lower bound of an ARCI by Theorem \ref{th:lower-bound-Gaussian}.

The proof of (\ref{eq:gau-tv-match}) is constructive. Namely, we construct explicit distributions $Q_0$ and $Q_1$ so that $0.95N(\theta,1)+0.05Q_0$ and $(1-\epsilon)N(\theta-r,1)+\epsilon Q_1$ are close. Without loss of generality, we can just consider $\theta=r$. Here, we illustrate a simple idea by first setting $Q_1=N(0,1)$. Then, we need to choose some $Q_0$ such that $0.95N(r,1)+0.05Q_0$ is close to $N(0,1)$. Suppose we define the density function of $Q_0$ by
\begin{equation}
q_0(x)=20\phi(x)-19\phi(x-r), \label{eq:attempt-con}
\end{equation} where $\phi(\cdot)$ is the density of the $N(0,1)$. 
Then, the null density would exactly match the alternative. However, the formula (\ref{eq:attempt-con}) is not a valid density function, since for any $r$, there exists some $x$ such that the function value is negative. To overcome this issue, we apply a truncation and consider
\begin{equation}
q_0(x) = 20\phi(x)-19 \frac{\indi\left\{x - r\leq t\right\}}{\bbP(N(0,1) \leq t)} \phi(x-r). \label{eq:actual-con}
\end{equation}
Though the null and alternative do not exactly match due to the truncation, it can be shown that
$$\TV\left(0.95N(r,1)+0.05Q_0, N(0,1)\right)\leq \frac{\alpha}{n},$$
as long as we take $t\asymp\sqrt{\log n}$, which then implies (\ref{eq:gau-tv-match}) by some basic property of total variation. The benefit of the truncation is that now we only require
\begin{equation}
\phi(x)\geq \frac{0.95\phi(x-r)}{\bbP(N(0,1) \leq t)}\quad\text{for all}\quad x\leq t + r,\label{eq:compare-tail-example}
\end{equation}
in order that (\ref{eq:actual-con}) is a valid density function. With $t\asymp\sqrt{\log n}$, one can directly verify that (\ref{eq:compare-tail-example}) holds whenever $r\leq \frac{c}{\sqrt{\log n}}$ for some small constant $c>0$. To obtain the other term $\frac{1}{\sqrt{\log (1/\epsilon)}}$ in the lower bound rate, a slightly more complicated construction of $Q_0$ and $Q_1$ is required. We refer the readers to the proof of Lemma \ref{prop:test-lower-bound-gaussian} in Appendix C.

%%%%%%%%%%%%%%%%%%%%%%%%%%%%%%%%%%%%
\section{General Location Families} \label{sec:ARCI-general-distribution}
%%%%%%%%%%%%%%%%%%%%%%%%%%%%%%%%%%%%

A careful reader may notice that both the upper and lower bound arguments developed in Sections \ref{sec:Gup} and \ref{sec:Glow} depend crucially on the Gaussian tail property. One may naturally wonder whether adaptation costs are different for other distributions. The goal of this section is to provide a systematic answer to this question.

\subsection{A Laplace Example} \label{sec:laplace-example}

We first highlight the Laplace location model, where the density function of the distribution $P_{\theta}=\text{Laplace}(\theta,1)$ is given by
$$p_{\theta}(x)=\frac{1}{2}\exp\left(-|x-\theta|\right).$$
Recall the definition of optimal length of ARCI $r_{\alpha}(\epsilon,\cE)$ given by (\ref{eq:minimax-length-ARCI}).
\begin{Proposition}\label{prop:laplace}
There exists some constant $c>0$, such that
$$r_{\alpha}(\epsilon,[0,0.05])\geq c,$$
for any $\epsilon\in[0,0.05]$ and any $\alpha\in(0,1/4)$.
\end{Proposition}
In contrast to the Gaussian location model, for Laplace data, it is impossible to construct a confidence interval whose length shrinks as sample size increases when $\epsilon$ is unknown. Note that the result holds for $\epsilon=0$, which reads $r_{\alpha}(0,[0,0.05])\geq c$. This means a statistician may have an infinite supply of clean samples without contamination, but as long as they do not know whether the data is clean or not, accurate uncertainty quantification will still be impossible.

It is worth mentioning that when $\epsilon$ is known, $r_{\alpha}(\epsilon,\{\epsilon\})\asymp \frac{1}{\sqrt{n}}+\epsilon$, which is the same as the Gaussian case, since the minimax rate of location estimation is the same for the two distributions \citep{chen2018robust}. On the other hand, when $\epsilon$ is unknown, the adaptation costs are completely different. The lower bound for Laplace data in Proposition \ref{prop:laplace} implies that the conservative strategy of using $\epsilon_{\max}=0.05$ to quantify uncertainty is actually optimal. That is, we have $r_{\alpha}(\epsilon,[0,0.05])\asymp r_{\alpha}(0.05,[0,0.05])\asymp 1$ in this case.

The adaptation cost for the Laplace distribution is implied by the hardness of the following hypothesis testing problem,
\begin{eqnarray*}
H_0: && X_1,\cdots,X_n\overset{iid}\sim0.95\text{Laplace}(r,1)+0.05Q \\
H_1: && X_1,\cdots,X_n\overset{iid}\sim\text{Laplace}(0,1).
\end{eqnarray*}
For $r$ being a sufficiently small constant, we take $Q$ to be a distribution with density
\begin{equation}
q(x)=10\exp\left(-|x|\right)-\frac{19}{2}\exp(-|x-r|), \label{eq:laplace-Q}
\end{equation}
then
$$0.95\text{Laplace}(r,1)+0.05Q=\text{Laplace}(0,1).$$
In other words, one cannot tell whether the data is generated by $\text{Laplace}(0,1)$ or by $\text{Laplace}(r,1)$ together with $5\%$ contamination, which implies that a confidence interval with coverage property cannot be shorter than $r$, a consequence of Proposition \ref{prop:CI-imply-test}. Note that the construction (\ref{eq:laplace-Q}) is in the same form as (\ref{eq:attempt-con}) in the Gaussian case. However, the formula (\ref{eq:attempt-con}) with Gaussian densities is not a valid distribution, and thus a truncation (\ref{eq:actual-con}) is needed. For Laplace density, (\ref{eq:laplace-Q}) turns out to be valid, since $q(x)\geq 0$ for all $x\in\mathbb{R}$, as long as $r$ is set as a small constant.

\subsection{Theory for Location Models}

To understand the striking difference between Gaussian and Laplace, 
let us study a general location model $P_{\theta}$. That is, $X\sim P_{\theta}$ if and only if $X-\theta\sim P_0$. We assume that the distribution $P_0$ admits some density function $f(\cdot)$ that is continuous on its support and symmetric around $0$.  The CDF of $P_0$ is denoted by $F(t)=P_0(X\leq t)$. The quantity $r_{\alpha}(\epsilon,\cE)$ is defined similarly as (\ref{eq:minimax-length-ARCI}) with $\cE=[0,\epsilon_{\max}]$ for some small constant $\epsilon_{\max}$. We comment that the symmetry is to simplify the notation. There is no additional difficulty in extending beyond this assumption.

By leveraging the results in Section \ref{sec:test-equi}, optimal construction of ARCI will be solved via analyzing the robust testing problems $\{ \cH(\theta, \theta \pm r_{ \epsilon }, \epsilon): \epsilon \in [0,\epsilon_{\max}] \}$ defined in (\ref{eq:test-def}). To develop general upper and lower bounds, we need to define some additional quantities. For each $\epsilon\in[0,\epsilon_{\max}]$, we define
\begin{eqnarray}
\label{eq:qup} \overline{q}(\epsilon) &=& \frac{6}{(1-\epsilon_{\max})} \left(\epsilon + \frac{100 \log(32/\alpha)}{3(1-\epsilon_{\max})n} \right), \\
\label{eq:qlo} \underline{q}(\epsilon) &=& \frac{1}{2(1-\epsilon_{\max})}\left(\epsilon+\frac{\alpha(1-\epsilon_{\max})}{n}\right).
\end{eqnarray}
With these two quantities, we further define
\begin{eqnarray}
\quad \overline{r}(\epsilon) &=& \inf\left\{r\geq 0:  \sup_{r \leq t\leq r+F^{-1}(1-\overline{q}(\epsilon))}\left(1-F(t-r)-\frac{6}{1-\epsilon_{\max}}(1-F(t))\right)\geq 0\right\}, \\
\quad  \underline{r}(\epsilon) &=& \sup\left\{r\geq 0: \sup_{t\leq r+F^{-1}(1-\underline{q}(\epsilon))}\left(f(t-r)-\frac{1-\left(\epsilon\vee\frac{\alpha}{n}\right)}{1-\epsilon_{\max}}f(t)\right)\leq 0\right\}.
\end{eqnarray}
In fact, $\overline{r}(\epsilon)$ and $\underline{r}(\epsilon)$ will be the upper and lower bounds for the optimal separation of the testing problems $\cH(\theta, \theta \pm r_{ \epsilon }, \epsilon)$. We note that the quantities $\overline{q}(\epsilon)$, $\underline{q}(\epsilon)$, $\overline{r}(\epsilon)$ and $\underline{r}(\epsilon)$ also depend on $\alpha$ and $\epsilon_{\max}$, but since both $\alpha$ and $\epsilon_{\max}$ are treated as constants, we suppress the dependence for the conciseness of notation.

\subsubsection{A General Upper Bound}

An extension of the testing functions (\ref{eq:test-gaussian}) and (\ref{eq:test+gaussian}) for a general location model is given by
\begin{eqnarray}
\label{eq:test-general} \quad ~~ \phi_{\theta, \theta- r_{\epsilon}, \epsilon} &=& \indi \left\{ \frac{1}{n}\sum_{i=1}^n \indi \left \{ X_i - (\theta - r_{\epsilon} ) \geq t_{\epsilon} \right\} \leq \frac{3}{2}\left(1-F(t_{\epsilon})+\frac{10\log(4/\alpha)}{9n}+\epsilon\right) \right\}, \\
\label{eq:test+general}\quad ~~  \phi_{\theta, \theta+ r_{\epsilon}, \epsilon} &=& \indi \left\{ \frac{1}{n}\sum_{i=1}^n \indi \left \{ X_i - (\theta + r_{\epsilon} ) \leq -t_{\epsilon} \right\} < \frac{3}{2}\left(1-F(t_{\epsilon})+\frac{10\log(4/\alpha)}{9n}+\epsilon\right) \right\}.
\end{eqnarray}
The parameters $r_{\epsilon}$ and $t_{\epsilon}$ are set by 
\begin{equation} \label{eq:r-general}
	r_{\epsilon} = \overline{r}(\epsilon)
\end{equation}
\begin{equation}\label{def:t-constant-general}
	t_{\epsilon} = \argmax_t \left\{\overline{r}(\epsilon)\leq t\leq \overline{r}(\epsilon)+F^{-1}(1-\overline{q}(\epsilon)):  1-F(t-\overline{r}(\epsilon))-\frac{6}{1-\epsilon_{\max}} (1-F(t))\geq 0\right\}.
\end{equation}
%\begin{eqnarray}
%\label{eq:r-general} r_{\epsilon} &=& \overline{r}(\epsilon) , \\
%\label{def:t-constant-general} \quad \quad t_{\epsilon} &=& \argmax_t \left\{t\leq \overline{r}(\epsilon)+F^{-1}(1-\overline{q}(\epsilon)):  1-F(t-\overline{r}(\epsilon))-\frac{9}{1-\epsilon_{\max}} (1-F(t))\geq 0\right\}.
%\end{eqnarray}
The testing errors of (\ref{eq:test-general}) and (\ref{eq:test+general}) are bounded by the following theorem.
\begin{Theorem}\label{thm:upper-bound-general}
Suppose $n\geq 400$, and $\epsilon_{\max},\alpha\in(0,1)$ are constants satisfying $\overline{q}(\epsilon_{\max})\leq 1$. Then, the testing functions $\{\phi_{\theta, \theta \pm r_{ \epsilon }, \epsilon}: \epsilon \in [0,\epsilon_{\max}] ,\theta\in\mathbb{R}\}$ defined by (\ref{eq:test-general}) and (\ref{eq:test+general}) with parameters $r_{\epsilon}$ and $t_{\epsilon}$ set by (\ref{eq:r-general}) and (\ref{def:t-constant-general}) satisfy
\begin{equation*}
		\begin{split}
			\textnormal{(simultaneous Type-1 error)} &\quad \sup_QP_{\epsilon_{\max},\theta,Q}\left(\sup_{\epsilon\in[0,\epsilon_{\max}]}\phi_{\theta,\theta\pm \overline{r}(\epsilon),\epsilon}=1\right)\leq \frac{\alpha}{4},\\
			\textnormal{(Type-2 error)} &\quad \sup_QP_{\epsilon,\theta\pm \overline{r}(\epsilon),Q}\left(\phi_{\theta,\theta\pm \overline{r}(\epsilon),\epsilon}=0\right)\leq \frac{\alpha}{4},
		\end{split}
	\end{equation*}
for all $\theta\in\mathbb{R}$ and $\epsilon\in[0,\epsilon_{\max}]$.
\end{Theorem}

Using (\ref{def:CI-based-on-test}) and similar calculations as in \eqref{eq:interval-step1}-\eqref{eq:interval-step2}, we obtain the following ARCI,
\begin{equation}\label{eq:adaptive-CI-general}
\begin{split}
\widehat{\CI}=\bigcap_{\epsilon\in[0,\epsilon_{\max}]}&\left[F_n^{-1}\left( \frac{3}{2}\left(1-F(t_{\epsilon})+\frac{10\log(4/\alpha)}{9n}+\epsilon\right)\right)  + t_{\epsilon} -  r_{\epsilon},\right. \\
& \left. F_n^{-1}\left(1-\frac{3}{2}\left(1-F(t_{\epsilon})+\frac{10\log(4/\alpha)}{9n}+\epsilon\right)\right)-t_{\epsilon}+r_{\epsilon}\right].
\end{split}
\end{equation}
Its coverage and length properties are direct consequences of Proposition \ref{prop:test-to-CI} and Theorem \ref{thm:upper-bound-general}.

\begin{Corollary}\label{coro:length-upper-bound}
Suppose $n\geq 400$, and $\epsilon_{\max},\alpha\in(0,1)$ are constants satisfying $\overline{q}(\epsilon_{\max})\leq 1$. Then, the confidence interval $\widehat{\CI}$ defined by (\ref{eq:adaptive-CI-general}) with parameters $r_{\epsilon}$ and $t_{\epsilon}$ set by (\ref{eq:r-general}) and (\ref{def:t-constant-general}) satisfies
\begin{equation*}
		\begin{split}
			\textnormal{(coverage)} &\quad \inf_{ \epsilon \in [0, \epsilon_{\max}], \theta, Q} P_{\epsilon, \theta, Q}\left( \theta \in\widehat{\CI} \right) \geq 1-\alpha,\\
			\textnormal{(length)} & \quad\inf_{\epsilon \in [0, \epsilon_{\max}], \theta, Q } P_{\epsilon, \theta, Q}\left( |\widehat{\CI}| \leq 2\overline{r}(\epsilon) \right) \geq 1-\alpha.
		\end{split}
	\end{equation*}
\end{Corollary}

To summarize, we have established the optimal length upper bound $r_{\alpha}(\epsilon,[0,\epsilon_{\max}])\leq 2\overline{r}(\epsilon)$.

\subsubsection{A General Lower Bound}

The following theorem extends the result of Lemma \ref{prop:test-lower-bound-gaussian} from Gaussian to general location models.
\begin{Theorem}\label{th:test-lower-bound}
Suppose $\epsilon_{\max},\alpha\in(0,1)$ are constants satisfying $n\geq 1/\epsilon_{\max}$ and $\underline{q}(\epsilon_{\max})\leq 1$. Then, we have
$$
			\inf_{Q_0, Q_1} \TV(P^{\otimes n}_{\epsilon_{\max}, \theta, Q_0}, P^{\otimes n}_{\epsilon, \theta \pm \underline{r}(\epsilon), Q_1}) \leq \alpha,
		$$
for all $\theta\in\mathbb{R}$ and $\epsilon\in[0,\epsilon_{\max}]$.
\end{Theorem}
Together with Proposition \ref{prop:CI-imply-test}, Theorem \ref{th:test-lower-bound} immediately implies the following optimal length lower bound.
\begin{Corollary}\label{coro:length-lower-bound}
Suppose $\epsilon_{\max},\alpha\in(0,1/4)$ are constants satisfying $n\geq 1/\epsilon_{\max}$ and $\underline{q}(\epsilon_{\max})\leq 1$. Then, we have
$$r_{\alpha}(\epsilon,[0,\epsilon_{\max}])\geq \underline{r}(\epsilon),$$
for any $\epsilon\in[0,\epsilon_{\max}]$.
\end{Corollary}
Corollaries \ref{coro:length-upper-bound} and \ref{coro:length-lower-bound} jointly characterize the optimal length of ARCI,
\begin{equation}
\underline{r}(\epsilon)\leq r_{\alpha}(\epsilon,[0,\epsilon_{\max}])\leq 2\overline{r}(\epsilon),
\end{equation}
which holds for any $\epsilon\in[0,\epsilon_{\max}]$.

\subsubsection{Simplifying the Formulas}

The quantities $\overline{r}(\epsilon)$ and $\underline{r}(\epsilon)$ can be computed explicitly for specific distributions. Though $\overline{r}(\epsilon)$ is defined through survival functions, while $\underline{r}(\epsilon)$ uses density functions, the two definitions usually lead to the same rate as a function of $n$ and $\epsilon$. In fact, for a location model with density $f$ satisfying some assumptions, we can work with the following simpler formulas,
\begin{eqnarray}
 \label{eq:r-uparraw} r^{\uparrow}(\epsilon) &=& \inf\left\{r\geq 0: \frac{f(F^{-1}(1-\overline{q}(\epsilon)))}{f(r+F^{-1}(1-\overline{q}(\epsilon)))}\geq \frac{6}{1-\epsilon_{\max}}\right\}, \\
r^{\downarrow}(\epsilon) &=& \sup\left\{r\geq 0: \frac{f(F^{-1}(1-\underline{q}(\epsilon)))}{f(r+F^{-1}(1-\underline{q}(\epsilon)))}\leq \frac{1-\epsilon_{\max}/2}{1-\epsilon_{\max}}\right\}.
\end{eqnarray}

\begin{Lemma}\label{lem:r-arrow}
Suppose $n\geq 400$, and $\epsilon_{\max},\alpha\in(0,1/4)$ are constants satisfying $n\geq 1/\epsilon_{\max}$ and $\overline{q}(\epsilon_{\max})\leq 1$. Assume $f$ is unimodal, and for any $r\geq 0$, $\frac{f(t-r)}{f(t)}$ is nondecreasing for all $t\geq r$ when it is well defined. Then, we have
\begin{equation}
r^{\downarrow}(\epsilon)\leq \underline{r}(\epsilon)\leq r_{\alpha}(\epsilon,[0,\epsilon_{\max}])\leq 2\overline{r}(\epsilon)\leq 2r^{\uparrow}(\epsilon), \label{eq:seq-ineq}
\end{equation}
for any $\epsilon\in[0,\epsilon_{\max}/2]$. Moreover, the last inequality above continues to hold under unimodality and $\frac{f(t-r)}{f(t)}$ being nondecreasing for $t\geq r+F^{-1}(1-\overline{q}(\epsilon))$.
\end{Lemma}

\begin{Remark}
The validity of the inequalities (\ref{eq:seq-ineq}) requires $\epsilon\leq\epsilon_{\max}/2$, which is sufficient for our purpose. This is because when $\epsilon\in(\epsilon_{\max}/2,\epsilon_{\max}]$, all of the five quantities involved in (\ref{eq:seq-ineq}) are of constant order.
\end{Remark}

\subsection{Additional Examples}\label{sec:examples-add}

We will work out the formulas $\overline{r}(\epsilon)$ and $\underline{r}(\epsilon)$ (or $r^{\uparrow}(\epsilon)$ and $r^{\downarrow}(\epsilon)$ if the assumption is satisfied) for the following list of distributions:
\begin{enumerate}
\item $t$ distribution with degrees of freedome $\nu>0$: $f(t)\propto (1+t^2/\nu)^{-\frac{\nu+1}{2}}$.
\item Generalized Gaussian distribution with shape parameter $\beta>0$: $f(t)\propto \exp\left(-|t|^{\beta}\right)$.
\item Mollifier with shape parameter $\beta>0$: $f(t)\propto \exp\left(-\frac{1}{(1-t^2)^{\beta}}\right)\indi\{|t| < 1\}$.
\item Bates distribution with order $k\in\mathbb{N}$: $f$ is the density function of $\frac{1}{k}\sum_{j=1}^kU_j$ with $U_j\overset{iid}\sim\text{Uniform}[-1/2,1/2]$.
\end{enumerate}
The list is ordered from heavy to light tails. Both mollifier and Bates distributions are bounded. While the mollifier is analytically smooth at the boundary of the support, Bates only has a finite (including zero) order of derivative at the boundary of the support.

\begin{Theorem}\label{thm:example}
For any $\alpha\in(0,1/4)$, there exists some constant $c$ only depending on $\alpha$ and shape of the distribution, such that whenever $\frac{1}{n}+\epsilon_{\max}\leq c$, the following results hold.
\begin{enumerate}
\item $t$ distribution with $\nu>0$:
$$r_{\alpha}(\epsilon,[0,\epsilon_{\max}])\asymp 1,$$
for any $\epsilon\in[0,\epsilon_{\max}]$, where $\asymp$ suppresses dependence on $\epsilon_{\max}$, $\alpha$ and $\nu$.
\item Generalized Gaussian with $\beta>0$:
$$r_{\alpha}(\epsilon,[0,\epsilon_{\max}])\asymp \left(\frac{1}{\log n}+\frac{1}{\log(1/\epsilon)}\right)^{\frac{(\beta-1)_+}{\beta}},$$
for any $\epsilon\in[0,\epsilon_{\max}]$, where $\asymp$ suppresses dependence on $\epsilon_{\max}$, $\alpha$ and $\beta$ and for any $a \in \bbR$, $(a)_{+} = \max(a,0)$.
\item Mollifier with $\beta>0$:
$$r_{\alpha}(\epsilon,[0,\epsilon_{\max}])\asymp \left(\frac{1}{\log n}+\frac{1}{\log(1/\epsilon)}\right)^{\frac{\beta+1}{\beta}},$$
for any $\epsilon\in[0,\epsilon_{\max}]$, where $\asymp$ suppresses dependence on $\epsilon_{\max}$, $\alpha$ and $\beta$.
\item Bates with $k\in\mathbb{N}$:
$$r_{\alpha}(\epsilon,[0,\epsilon_{\max}])\asymp \left(\frac{1}{n}+\epsilon\right)^{\frac{1}{k}},$$
for any $\epsilon\in[0,\epsilon_{\max}]$, where $\asymp$ suppresses dependence on $\epsilon_{\max}$, $\alpha$ and $k$.
\end{enumerate}
\end{Theorem} % C_1\left(\frac{1}{n}+\epsilon\right)^{\frac{1}{k}}\leq 

It is clear from the above results that the optimal length of ARCI for a location model crucially depends on the tail of the density function. The adaptation cost of a distribution with a heavy tail is more severe than that of a distribution with a light tail. In particular, the Laplace tail serves as the boundary between shrinking versus non-shrinking optimal ARCI. It is worth noting that by explicitly tracking the dependence on the degrees of freedom $\nu$, a more precise formula for the rate of $t$ distribution is $\frac{1}{\sqrt{\nu}}+\frac{1}{F_{\nu}^{-1}(1-n^{-1}-\epsilon)}$, where $F_{\nu}(\cdot)$ is the $t$ distribution CDF. Interestingly, for a sufficiently large $\nu$ that diverges as $\frac{1}{n}+\epsilon\rightarrow 0$, this formula agrees with the Gaussian rate, since $\frac{1}{F_{\nu}^{-1}(1-n^{-1}-\epsilon)}$ dominates the two terms and $F_{\nu}^{-1}(\cdot)$ is close to $\Phi^{-1}(\cdot)$ for large $\nu$.

In comparison, when $\epsilon$ is known, the distribution shape does not affect the optimal length of the robust confidence interval. We have $r_{\alpha}(\epsilon,\{\epsilon\})\asymp \frac{1}{\sqrt{n}}+\epsilon$ for all distributions considered in Theorem \ref{thm:example}, which is achieved by quantifying the uncertainty of sample median. Rigorous proofs will be given in Appendix F. The only exceptional cases are $\beta\leq 1/2$ for generalized Gaussian and $k=1,2$ for Bates, in which case Fisher information does not exist, and thus optimal rates of location parameter estimation cannot be parametric.

For all $t$ distributions and generalized Gaussian with $\beta\leq 1$, since $r_{\alpha}(\epsilon,[0,\epsilon_{\max}])\asymp r_{\alpha}(\epsilon_{\max},[0,\epsilon_{\max}])\asymp 1$, an optimal ARCI is the conservative one (\ref{eq:conservative}).

For generalized Gaussian with $\beta>1$, mollifier, and Bates, an optimal ARCI is given by (\ref{eq:adaptive-CI-general}). One can also check that the condition of Lemma \ref{lem:r-arrow} is satisfied for these examples. Therefore, the parameters $r_{\epsilon}$ and $t_{\epsilon}$ can be set with much simpler forms than (\ref{eq:r-general}) and (\ref{def:t-constant-general}). With $r_{\epsilon}=r^{\uparrow}(\epsilon)$ and $t_{\epsilon}=r^{\uparrow}(\epsilon)+F^{-1}(1-\overline{q}(\epsilon))$, the length optimality in Theorem \ref{thm:example} can still be achieved.

%%%%%%%%%%%%%%%%%%%%%%%%%%%%%%%%%%%
\section{Some Extensions} \label{sec:extension}
%%%%%%%%%%%%%%%%%%%%%%%%%%%%%%%%%%%
\subsection{Results for $\epsilon_{\max} \to 0$} \label{sec:tozero}

The results in previous sections have focused on the setting where $\epsilon_{\max}$ is a small constant (e.g. $\epsilon_{\max} = 0.05$). In a situation where a statistician knows $\epsilon_{\max} \to 0$, the adaptation cost can be alleviated. For example, suppose $\epsilon_{\max}=\frac{1}{\sqrt{n}}$, and then the median interval (\ref{eq:median-RCI}) with $\epsilon$ replaced by $\epsilon_{\max}$ can be used to achieve the optimal rate $\frac{1}{\sqrt{n}}$. In this case, there is no adaptation cost. The following result extends Theorem \ref{th:lower-bound-Gaussian} and Theorem \ref{thm:upper-bound-general-Gaussian} by explicitly tracking the dependence on $\epsilon_{\max}$. 

\begin{Theorem} \label{th:optimal-eps-max}
	Consider $X_1, \ldots, X_n \overset{iid}\sim (1-\epsilon)N(\theta,1)+\epsilon Q$. There exists a universal constant $C>0$, such that for any $\alpha\in(0,0.1)$, $\epsilon_{\max} \in [0,0.05]$ and $n\geq C\log(1/\alpha)$, we have
	\begin{equation}
r_{\alpha}(\epsilon,[0,\epsilon_{\max}]) \asymp \frac{1}{\sqrt{n}} + \frac{\epsilon_{\max}}{\sqrt{\log(n\epsilon_{\max}^2+e)}} + \frac{\epsilon_{\max}}{\sqrt{\log\left(\frac{\epsilon_{\max}}{\epsilon}+e\right)}}, \label{eq:veryniubi}
\end{equation}  where $\asymp$ suppresses dependence on $\alpha$.
\end{Theorem}
When specialized to $\epsilon_{\max}=0.05$, the rate (\ref{eq:veryniubi}) recovers
$$r_{\alpha}(\epsilon,[0,0.05])\asymp \frac{1}{\sqrt{\log(n)}}+ \frac{1}{\sqrt{\log(1/\epsilon)}},$$
Another interesting example is
$$r_{\alpha}(\epsilon,[0,n^{-1/2}+\epsilon]) \asymp \frac{1}{\sqrt{n}}+\epsilon,$$
which implies no adaptation cost.
In general, different orders of $\epsilon_{\max}$ imply different adaptation costs. An optimal ARCI that achieves the rate (\ref{eq:veryniubi}) can also be derived by inverting robust testing procedures, and the details are deffered to Appendix E.1.

\subsection{Results for $\epsilon_{\max}=0.49$} \label{sec:49} In this section, we consider the extension when $\epsilon_{\max}$ is greater than $0.05$. Concretely, we present a simple modification of the optimal ARCI (\ref{eq:adaptive-CI}) for Gaussian data so that it will be adaptive for $\epsilon\in[0,0.49]$. It requires the following modification of Lemma \ref{th:estimation-guarantee}. Recall the definitions of $\wh{\theta}_L(t)$ and $\wh{\theta}_R(t)$ in (\ref{eq:theta-L}) and (\ref{eq:theta-R}).
\begin{Lemma}\label{th:estimation-guarantee-mod}
Simultaneously for all $t\in\left[4, \Phi^{-1}\left(1 - \sqrt{ \log (2/\alpha)/(2n) } \right)\right]$, we have
$$\wh{\theta}_L(t)-\theta\leq \frac{8}{t}\quad\text{and}\quad\wh{\theta}_R(t)-\theta\geq-\frac{8}{t},$$
with $P_{\epsilon, \theta, Q}$-probability at least $1-\alpha$, uniformly over all $\epsilon\in[0,0.99]$, all $\theta\in\mathbb{R}$, and all $Q$.
\end{Lemma}
In other words, a similar one-sided uncertainty quantification to Lemma \ref{th:estimation-guarantee} can be obtained for all $\epsilon\in[0,0.99]$ with more conservative constants. In fact, the constants $4$ and $8$ can be slightly improved if we only require the result to hold for $\epsilon\in[0,0.49]$. We use $0.99$ instead of $0.49$ so that the same Lemma \ref{th:estimation-guarantee-mod} can also be applied in Section \ref{sec:list-decodable} later for a different purpose.

Lemma \ref{th:estimation-guarantee-mod} immediately implies that the following interval
\begin{equation}
\bigcap_{4 \leq t \leq  \Phi^{-1}\left(1 - \sqrt{ \log (2/\alpha)/(2n) } \right) }\left[\wh{\theta}_L(t)-\frac{8}{t},\wh{\theta}_R(t)+\frac{8}{t}\right] \label{eq:just-coverage}
\end{equation}
has coverage probability at least $1-\alpha$ for all $\epsilon\in[0,0.49]$. However, the formula (\ref{eq:just-coverage}) does not guarantee length optimality. By Lemma \ref{th:estimation-length-guarantee}, the length optimality of (\ref{eq:just-coverage}) only holds when $t_{\epsilon}$ defined in (\ref{def:t-constant}) belongs to the interval $\left[4, \Phi^{-1}\left(1 - \sqrt{ \log (2/\alpha)/(2n) } \right)\right]$, which requires that $\epsilon$ needs to be small. Therefore, we need to further modify (\ref{eq:just-coverage}) to extend the length optimality to large values of $\epsilon$. This could be achieved by considering the conservative interval (\ref{eq:conservative}).
Note that the length of (\ref{eq:conservative}) is a constant, which is optimal when $\epsilon$ is greater than some constant. Thus, the intersection of (\ref{eq:just-coverage}) and (\ref{eq:conservative}) results in both coverage property and length optimality for all $\epsilon\in[0,0.49]$.

\begin{Proposition}\label{th:ARCI-large-eps1}
Consider $X_1, \ldots, X_n \overset{iid}\sim P_{\epsilon, \theta, Q}=(1-\epsilon)N(\theta,1)+\epsilon Q$.
There exist some universal constants $C,R>0$, such that for any $\alpha\in(0,1)$ and $n\geq C\log(1/\alpha)$, the interval
$$\widehat{\CI}=\left(\bigcap_{4 \leq t \leq  \Phi^{-1}\left(1 - \sqrt{ \log (2/\alpha)/(2n) } \right) }\left[\wh{\theta}_L(t)-\frac{8}{t},\wh{\theta}_R(t)+\frac{8}{t}\right]\right)\bigcap \left[\wh{\theta}_{\rm median}-R,\wh{\theta}_{\rm median}+R\right]$$
satisfies
	\begin{equation*}
		\begin{split}
			\textnormal{(coverage)} &\quad \inf_{ \epsilon \in [0, 0.49], \theta, Q} P_{\epsilon, \theta, Q}\left( \theta \in\widehat{\CI} \right) \geq 1-\alpha,\\
			\textnormal{(length)} & \quad\inf_{\epsilon \in [0, 0.49], \theta, Q } P_{\epsilon, \theta, Q}\left( |\widehat{\CI}| \leq C' \left( \frac{1}{\sqrt{\log n } } + \frac{1}{\sqrt{ \log(1/\epsilon) }}\right) \right) \geq 1-\alpha,
		\end{split}
	\end{equation*}
where $C'>0$ is some constant only depending on $\alpha$.
\end{Proposition}

For general location models, an extension of Corollary \ref{coro:length-upper-bound} to $\epsilon_{\max}=0.49$ can also be obtained by modifying (\ref{eq:seq-ineq}) with more conservative constants and intersection with (\ref{eq:conservative}).

\subsection{Robust Confidence Set and List-Decodable Learning}\label{sec:list-decodable}

In this section, we consider $\epsilon_{\max}=0.99$. This means in the worst case, only $1\%$ of the data set are generated from the specified parametric model $P_{\theta}$. For the purpose of estimating $\theta$, nothing is possible when there are more outliers than good samples, since the outliers can all be drawn from $P_{\theta'}$ with a different $\theta'$ that is arbitrarily far away from $\theta$. However, uncertainty quantification is still possible even with only $1\%$ clean samples, as long as we replace confidence intervals with more general confidence sets that are unions of intervals. This idea is resulted from an area called list decodable learning, which we will briefly introduce below.

\subsubsection{List Decodable Mean Estimation}

Consider $X_1, \ldots, X_n \overset{iid}\sim (1-\epsilon)N(\theta,1)+\epsilon Q$ with some $\epsilon\in[0,0.99]$. The goal of list decodable mean estimation is to compute a list $\mathcal{L}$, a set with finite cardinality, from data, such that one of the elements in the list is close to the parameter of interest. Namely, $\min_{\wh{\theta}\in\mathcal{L}}|\wh{\theta}-\theta|$ is bounded. Moreover, the cardinality of $\mathcal{L}$ should be $\#\mathcal{L}\lesssim 1$. The classical estimation problem is equivalent to requiring $\#\mathcal{L}=1$, which is only possible when $\epsilon<1/2$. The original list decodable learning framework was introduced by \cite{balcan2008discriminative} in a clustering context. The problem has gained much attention recently in robust estimation, especially in high-dimensional settings. We refer the readers to Chapter 5 of \cite{diakonikolas2023algorithmic} and references therein.

The idea of finding a list satisfying the above criteria is quite straightforward with one-dimensional data. Since we know $\epsilon_{\max}=0.99$, there are at least $(1-o(1))n/100$ samples that are actually generated from $N(\theta,1)$, and they cluster around $\theta$. Therefore, a point on the real line with a constant fraction of samples in its neighborhood serves as a reasonable candidate. The set of such candidates is defined by
$$\mathcal{H}=\left\{x\in\mathbb{R}: \frac{1}{n}\sum_{i=1}^n\indi\left(|X_i-x| \leq 2\right)\geq 0.005\right\}.$$
The final list is obtained by
\begin{equation}
\mathcal{L}=\argmax_L\left\{\#L: L\subset\mathcal{H},\min_{x,x'\in L:x\neq x'}|x-x'| > 4\right\}, \label{eq:list-mean}
\end{equation}
which is a maximal packing set of the candidate set $\mathcal{H}$. If the maximal packing set is not unique, we will take an arbitrary one as $\mathcal{L}$.

\begin{Proposition}\label{prop:list-d-m-e-c}
There exist some universal constants $C,C'>0$, such that for any $\alpha\in(0,1)$ and $n\geq C\log(1/\alpha)$, the list $\mathcal{L}$ defined in (\ref{eq:list-mean}) satisfies $\#\mathcal{L}\leq C'$ and 
\begin{equation}
\inf_{\theta,Q}P_{\epsilon,\theta,Q}\left(\min_{\wh{\theta}\in\mathcal{L}}|\wh{\theta}-\theta|\leq 4\right)\geq 1-\alpha, \label{eq:list-property}
\end{equation}
for any $\epsilon\in[0,0.99]$.
\end{Proposition}

\subsubsection{Adaptive Robust Confidence Set}

When the majority of data points are from contamination, we aim to construct a confidence set $\widehat{\CS}$ that satisfies the following properties with high probability: (a) it covers the parameter of interest, (b) its Lebesgue measure is controlled, and (c) it is a union of $K$ intervals with $K=O(1)$. Ideally, we also hope that the confidence set is reduced to a single interval when $\epsilon$ is sufficiently small.

Our construction is based on the interval (\ref{eq:just-coverage}) and the list (\ref{eq:list-mean}). The definition is given by
\begin{equation}
\widehat{\CS} = \left(\bigcap_{4 \leq t \leq  \Phi^{-1}\left(1 - \sqrt{ \log (2/\alpha)/(2n) } \right) }\left[\wh{\theta}_L(t)-\frac{8}{t},\wh{\theta}_R(t)+\frac{8}{t}\right]\right)\bigcap\left(\bigcup_{\wh{\theta}\in\mathcal{L}}\left[\wh{\theta}-4,\wh{\theta}+4\right]\right). \label{eq:arcs-final}
\end{equation}
The coverage property of (\ref{eq:arcs-final}) is guaranteed by Lemma \ref{th:estimation-guarantee-mod} and Proposition \ref{prop:list-d-m-e-c} for all $\epsilon\in[0.99]$. Moreover, $\widehat{\CS}$ is a union of at most $\#\mathcal{L}=O(1)$ intervals. Its volume, quantified by Lebesgue measure, is at most $8\#\mathcal{L}=O(1)$, which is optimal when $\epsilon$ is of constant order. When $\epsilon$ is sufficiently small, the volume of $\widehat{\CS}$ is bounded by the length of the interval (\ref{eq:just-coverage}), which is still optimal by Lemma \ref{th:estimation-length-guarantee}. These properties are summarized in the following theorem.

\begin{Theorem} \label{th:ARCI-list-decodable}
There exist some universal constants $C,C_1>0$, such that for any $\alpha\in(0,1)$ and $n\geq C\log(1/\alpha)$, the confidence set $\widehat{\CS}$ defined in (\ref{eq:arcs-final}) can be written as $\widehat{\CS}=\bigcup_{k=1}^K[\ell_k,u_k]$, with $\ell_1<u_1<l_2<u_2<\cdots<l_K<u_K$ and $K\leq C_1$,
and it satisfies
\begin{equation*}
		\begin{split}
			\textnormal{(coverage)} &\quad \inf_{ \epsilon \in [0, 0.99], \theta, Q} P_{\epsilon, \theta, Q}\left( \theta \in\widehat{\CS} \right) \geq 1-\alpha,\\
			\textnormal{(volume)} & \quad\inf_{\epsilon \in [0, 0.99], \theta, Q } P_{\epsilon, \theta, Q}\left( |\widehat{\CS}| \leq C' \left( \frac{1}{\sqrt{\log n } } + \frac{1}{\sqrt{ \log(1/\epsilon) }}\right) \right) \geq 1-\alpha, %\inf_{ \epsilon \in [0, 0.99], \theta, Q} P_{\epsilon, \theta, Q}\left( K \leq C' \right) \geq 1-\alpha,
		\end{split}
	\end{equation*}
where $|\widehat{\CS}|=\sum_{k=1}^K(u_k-l_k)$ denotes the volume, and $C'>0$ is some constant only depending on $\alpha$.
\end{Theorem}

One may be wondering if $\widehat{\CS}$ is actually reduced to a single interval when $\epsilon$ is sufficiently small. Unfortunately, the answer is no. According to its construction (\ref{eq:arcs-final}), it is possible that $\widehat{\CS}$ is a union of two intervals even when $\epsilon=0$. This is because there may be two intervals in the collection $\left\{\left[\wh{\theta}-4,\wh{\theta}+4\right]:\wh{\theta}\in\mathcal{L}\right\}$ that are arbitrarily close, and thus it is possible that both will survive the intersection with (\ref{eq:just-coverage}). To fix this issue, we can use the following simple modification,
\begin{equation}
\widehat{\CS} = \begin{cases}
\text{(\ref{eq:arcs-final})}, & |\text{(\ref{eq:just-coverage})}| \geq 8, \\
\text{(\ref{eq:just-coverage})}, & |\text{(\ref{eq:just-coverage})}| < 8.
\end{cases}\label{eq:arcs-mod}
\end{equation}
This construction (\ref{eq:arcs-mod}) inherits all the three properties in Theorem \ref{th:ARCI-list-decodable}, and it is reduced to a single interval when $\epsilon$ is sufficiently small.

%%%%%%%%%%%%%%%%%%%%%%%%%%%%%%%%%%%%%%%%%%%%%%
\subsection{Beyond Huber Contamination} \label{sec:TV-contamination}
%%%%%%%%%%%%%%%%%%%%%%%%%%%%%%%%%%%%%%%%%%%%%%

In this section, we consider a stronger contamination model \citep{hampel1968contributions}, where one observes
\begin{equation}
X_1, \ldots, X_n \overset{iid}\sim P\quad\text{for some}\quad P\in\{P:\TV(P,P_{\theta})\leq\epsilon\}.\label{eq:TV-model}
\end{equation}
It is well know that $\TV(P,P_{\theta})\leq\epsilon$ if and only if
$$P=P_{\theta}-\epsilon Q_1+\epsilon Q_2,$$
for some distributions $Q_1$ and $Q_2$. In other words, while Huber's setting only considers additive contamination, the total variation ball allows both addition and subtraction. A thorough discussion of various contamination models and their relations is given by \cite{diakonikolas2016robust}. Though the minimax rates of estimating $\theta$ are usually the same across different contamination models, a recent work \cite{canonne2023full} identifies the difference of information-theoretic limits when it comes to hypothesis testing. Given the equivalence between ARCI and robust testing established in Section \ref{sec:test-equi}, one would suspect that the optimal length of ARCI could also be different under the stronger setting (\ref{eq:TV-model}).

In fact, it is not hard to conclude from a similar argument to Proposition \ref{prop:CI-imply-test} that the optimal length of ARCI under (\ref{eq:TV-model}) is characterized by the following family of robust hypothesis testing,
\begin{equation}\label{eq:test-def-tv}
	\begin{split} &\cH^{\textnormal{tv}}(\theta, \theta- r, \epsilon): \begin{array}{l}
		H_{0}: P \in \left\{P: \TV(P,P_{\theta})\leq\epsilon_{\max} \right\} \quad  \textnormal { v.s. } \quad 
		H_{1}: P \in \left\{P: \TV(P,P_{\theta-r})\leq\epsilon \right\},
	\end{array}\\
	&\cH^{\textnormal{tv}}(\theta, \theta + r, \epsilon): \begin{array}{l}
		H_{0}: P \in \left\{P: \TV(P,P_{\theta})\leq\epsilon_{\max} \right\} \quad  \textnormal { v.s. } \quad
		H_{1}: P \in \left\{P: \TV(P,P_{\theta+r})\leq\epsilon \right\},
	\end{array}
	\end{split}
\end{equation}
which is a generalization of (\ref{eq:test-def}) to the setting of (\ref{eq:TV-model}).

Take $P_{\theta}=N(\theta,1)$ as an example. One can easily check that $P_{\theta-r}\in \left\{P: \TV(P,P_{\theta})\leq\epsilon_{\max} \right\}$ whenever $r\leq c\epsilon_{\max}$ for some universal constant $c>0$, which implies that the null and alternative of $\cH^{\textnormal{tv}}(\theta, \theta- r, \epsilon)$ cannot be separated even when $r$ is of constant order. The same conclusion holds for $\cH^{\textnormal{tv}}(\theta, \theta + r, \epsilon)$ as well. Therefore, the optimal length of ARCI under (\ref{eq:TV-model}) is of constant order, achieved by the conservative interval (\ref{eq:conservative}). In contrast, when $\epsilon$ is known, the confidence interval based on sample median (\ref{eq:median-RCI}) still achieves the length of order $\frac{1}{\sqrt{n}}+\epsilon$ under (\ref{eq:TV-model}) just as in Huber's contamination model.

%%%%%%%%%%%%%%%%%%%%%%%
\subsection{Unknown Scale Parameter} \label{sec:unknonw-variance}
%%%%%%%%%%%%%%%%%%%%%

The construction of the ARCI (\ref{eq:adaptive-CI}) crucially depends on the knowledge of variance. In this section, we study the following more general setting,
\begin{equation} \label{def:model-variance}
	X_1, \ldots, X_n \overset{iid}\sim P_{\epsilon,\theta, Q, \sigma}= (1- \epsilon ) N(\theta,\sigma^2) + \epsilon Q. 
\end{equation}
Suppose the value of $\sigma$ is given, a straightforward scaling argument leads to the following extension of (\ref{eq:adaptive-CI}),
\begin{equation}\label{eq:adaptive-CI-variance}
\begin{split}
	\widehat{\CI}(\sigma) = \Big[ &\max_{1.6 \leq t \leq  \Phi^{-1}\left(1 - \sqrt{ \log (2/\alpha)/(2n) } \right)  } \left(F_n^{-1}( 2(1 - \Phi(t) ) )  + t \sigma -  \frac{2\sigma}{t} \right), \\
	& \min_{ 1.6 \leq t \leq  \Phi^{-1}\left(1 - \sqrt{ \log (2/\alpha)/(2n) } \right)} \left( F_n^{-1}( 1- 2(1 - \Phi(t) ) )  - t \sigma  + \frac{2 \sigma}{t} \right) \Big].
\end{split}
\end{equation}
Theorem \ref{thm:upper-bound-general-Gaussian} continues to hold, and the order of the optimal length achieved by (\ref{eq:adaptive-CI-variance}) is
\begin{equation}
\sigma\left(\frac{1}{\sqrt{\log n}}+\frac{1}{\sqrt{\log(1/\epsilon)}}\right).\label{eq:sigma-length}
\end{equation}

When $\sigma$ is unknown, we can consider a robust estimator for it. It was proved by \cite{chen2018robust} that median absolute deviation (multiplied by a known constant) achieves
\begin{equation}
\frac{|\wh{\sigma}-\sigma|}{\sigma} = O_{\mathbb{P}}\left(\frac{1}{\sqrt{n}}+\epsilon\right), \label{eq:mad-error}
\end{equation}
and this rate is minimax optimal under (\ref{def:model-variance}). Since the rate (\ref{eq:mad-error}) is much faster than (\ref{eq:sigma-length}), one would naturally conjecture that the plug-in strategy $\widehat{\CI}(\wh{\sigma})$ would still achieve the length (\ref{eq:sigma-length}). This intuition turns out to be false. The ignorance of the scaling parameter changes the problem in a fundamental way.

To rigorously study robust confidence intervals that are adaptive to both $\epsilon$ and $\sigma$, we will need to extend the definition of the optimal confidence interval length (\ref{eq:minimax-length-ARCI}). We first define the set of all $(1-\alpha)$-level ARCIs over $\cE\times \Sigma$ by
$$\cI_{\alpha}(\cE,\Sigma) = \left\{ \widehat{\CI} = [l(\{X_i \}_{i=1}^n), u(\{X_i \}_{i=1}^n) ]: \inf_{\epsilon \in \cE}\inf_{\sigma\in\Sigma} \inf_{\theta, Q} P_{\epsilon, \theta, Q,\sigma} \left( \theta \in \widehat{\CI} \right) \geq 1- \alpha  \right\}.$$
The optimal length at some $(\epsilon,\sigma)\in\cE\times \Sigma$ over all $(1-\alpha)$-level ARCIs is defined by
$$
	r_{\alpha}( \epsilon,\sigma, \cE,\Sigma) = \inf \left\{r \geq 0: \inf_{\widehat{\CI} \in  \cI_{\alpha}(\cE,\Sigma) }\sup_{\theta, Q} P_{\epsilon, \theta, Q,\sigma} \left( |\widehat{\CI}| \geq r \right) \leq \alpha \right\}. 
$$
With this new notation, we know that the order of $r_{\alpha}\left(\epsilon,\sigma,[0,0.05],\{\sigma\}\right)$ is given by (\ref{eq:sigma-length}).
\begin{Theorem}\label{th:lower-bound-Gaussian-unknown-variance}
Suppose $\alpha\in(0,1/4)$ and $\sigma>0$. There exists some universal constant $c>0$, such that
\begin{equation}
r_{\alpha}\left(\epsilon,\sigma,[0,0.05],\left[\frac{\sigma}{2},\sigma\right]\right)\geq c\sigma, \label{eq:r-l-v}
\end{equation}
for any $\epsilon\in[0,0.05]$.
\end{Theorem}

Theorem \ref{th:lower-bound-Gaussian-unknown-variance} rules out the possibility of shrinking adaptive confidence intervals when the variance parameter is unknown. The lower bound rate $\sigma$ can simply be achieved by the conservative strategy
$$[\wh{\theta}_{\rm median}-\wh{\sigma}R,\wh{\theta}_{\rm median}+\wh{\sigma}R],$$
where $\wh{\sigma}$ is an estimator of $\sigma$ satisfying (\ref{eq:mad-error}), and $R$ is some sufficiently large constant. The choice $\left[\frac{\sigma}{2},\sigma\right]$ in (\ref{eq:r-l-v}) actually implies that the statistician knows the order of the true scale parameter, but the lower bound (\ref{eq:r-l-v}) still rules out any non-trivial interval length. In particular, plugging the largest value of the interval $\left[\frac{\sigma}{2},\sigma\right]$ in some procedure such as (\ref{eq:adaptive-CI-variance}) does not lead to adaptation to the unknown $\epsilon$. One has to know the exact value of the scale parameter. In contrast, if the order of $\epsilon$ is known, then this is equivalent to knowing both $\epsilon$ and $\sigma$ in the sense that
$$r_{\alpha}\left(\epsilon,\sigma,\left[\frac{\epsilon}{2},\epsilon\right],(0,\infty)\right)\asymp r_{\alpha}\left(\epsilon,\sigma,\{\epsilon\},\{\sigma\}\right)\asymp \sigma\left(\frac{1}{\sqrt{n}} + \epsilon\right),$$
which is achieved by the median confidence interval (\ref{eq:median-RCI}) by plugging the largest value of $\left[\frac{\epsilon}{2},\epsilon\right]$ and an estimator of $\sigma$ satisfying (\ref{eq:mad-error}). The subtlety is summarized in Table \ref{tab:rate}.

\begin{table}[h]
	\centering
	\begin{tabular}{c | c | c}
		\hline
	 	& known $\epsilon$ & unknown $\epsilon$\\
	 \hline 
	 known $\sigma$ & $\sigma\left(\frac{1}{\sqrt{n}} + \epsilon\right)$ & $\sigma\left(\frac{1}{\sqrt{\log n}}+\frac{1}{\sqrt{\log(1/\epsilon)}}\right)$ \\
	 \hline 
	 unknown $\sigma$  & $\sigma\left(\frac{1}{\sqrt{n}} + \epsilon\right)$ & $\sigma$\\
	 \hline 
	\end{tabular} \caption{Optimal length of confidence interval for $\theta$ with samples $X_1, \ldots, X_n \overset{iid}\sim (1- \epsilon ) N(\theta,\sigma^2) + \epsilon Q$.}
	\label{tab:rate}
\end{table}

The proof of Theorem \ref{thm:upper-bound-general-Gaussian} is based on the following construction. There exist distributions $Q_0$ and $Q_1$, such that
\begin{equation}
0.95N(\theta, (0.99\sigma)^2)+0.05 Q_0=(1-\epsilon)N(\theta-c\sigma,\sigma^2)+\epsilon Q_1 \label{eq:construct-var-unknown}
\end{equation}
holds for some universal constant $c>0$ independent of $\sigma$. In other words, without the value of the scale parameter being fixed, one cannot tell the difference between the location parameters $\theta$ and $\theta-c\sigma$ in the presence of contamination. We note that the construction (\ref{eq:construct-var-unknown}) would be impossible when the two distributions have the same scale parameter. With this construction, Theorem \ref{th:lower-bound-Gaussian-unknown-variance} is an immediate consequence of a similar argument to Proposition \ref{prop:CI-imply-test}. An analogous result to Theorem \ref{th:lower-bound-Gaussian-unknown-variance} can also be proved for other location models considered in Section \ref{sec:examples-add} with a similar construction.

\vskip 0.2in
\bibliographystyle{apalike}
\bibliography{reference}

		\newpage
		\appendix
		\
			\begin{center}
			{\LARGE Supplement to "Adaptive Robust Confidence Intervals"	
				
			}		
			\medskip
			
			{\large Yuetian Luo \quad and\quad Chao Gao}
			\medskip
		\end{center}
		
		%\begin{abstract}
		In this supplement, we provide a table of contents and all of the technical proofs.
		%\end{abstract}
		\tableofcontents

%%%%%%%%%%%%%%%%%%%%%%%%%%%%%%%%%%%%%
\section{Discussion of Universal Inference and HulC} \label{sec:median-bias}
%%%%%%%%%%%%%%%%%%%%%%%%%%%%%%%%%%%%
In this section, we provide a brief discussion on the universal inference \citep{wasserman2020universal}, its robust version \citep{park2023robust}, and the HulC \citep{kuchibhotla2023hulc} procedures to illustrate the difficulty of applying them to the setting \eqref{def:model}.

\begin{itemize}
	\item (Universal Inference) In this procedure, one first splits the data into two sets $\cD_0$ and $\cD_1$. Then use $\cD_1$ to compute an estimator $\htheta$ for $\theta$. Finally, the confidence set from universal inference is given by
	\begin{equation*}
		\widehat{\CI} = \left\{ \theta: \frac{ \prod_{X_i \in \cD_0 } dP_{\theta}(X_i) }{ \prod_{X_i \in \cD_0 } dP_{\htheta}(X_i) } \geq \alpha \right\}.
	\end{equation*} 
	When applied to Huber's contamination model, one way to modify the above formula is to replace $P_{\theta}$ by $P_{\epsilon,\theta,Q}$, but this cannot be done without the knowledge of $Q$ or $\epsilon$. 
	\item (Robust Universal Inference) The robust universal inference can be applied to Huber's contamination model $P_{\epsilon,\theta,Q}$. However, the goal of robust universal inference is to cover a target that is different from the model parameter $\theta$. The target functional is defined by
	\begin{equation*}
		\theta^* = \argmin_{\theta' \in \bbR} \rho(P_{\theta'}|| P_{\epsilon, \theta, Q}),
	\end{equation*} 
	where $\rho(\cdot\|\cdot)$ is some divergence function. It is not hard to consider an example of a confidence interval that covers $\theta^*$ but not $\theta$. For instance, consider $n=\infty$, which means one can directly access the distribution $P_{\epsilon, \theta, Q}$. Since $\theta^*$ can be computed from $P_{\epsilon, \theta, Q}$, the set $\{\theta^*\}$ is a valid construction in the large sample limit to cover $\theta^*$, but it will not cover $\theta$.

	\item  (HulC) The HulC procedure is based on a (nearly) median unbiased estimator. Specifically, one first splits the data into $B \geq 1$ folds and computes $\htheta_1, \ldots, \htheta_B$ independent estimators of $\theta$ based on each fold. Then the HulC confidence interval is $\widehat{\CI} = [ \min_{i=1, \ldots, B} \htheta_i, \max_{i = 1, \ldots, B} \htheta_i ] $ with the following coverage guarantee
\begin{equation*}
	\bbP(\theta \notin \widehat{\CI}) \leq (1/2 - \Delta)^B + (1/2 + \Delta)^B,
\end{equation*} where $\Delta = \max_{i = 1, \ldots, B} \textnormal{Med-Bias}_{\theta}(\htheta_i)$ and $\textnormal{Med-Bias}_{\theta}(\htheta_i)$ is the median bias of $\htheta_i$ defined as
	\begin{equation*}
	\textnormal{Med-Bias}_{\theta}(\htheta_i) = \left( \frac{1}{2} - \min\{ \bbP(\htheta_i \geq \theta), \bbP(\htheta_i \leq \theta) \} \right)_{+}. 
\end{equation*}  
	
	However, under the contamination setting, finding a median unbiased estimator is a nontrivial task. For example, it can be shown that the sample median can have a median bias as least as large as $1/2 - \exp(-n \epsilon^2/2)$. As a consequence, if each fold contains $\lfloor n/B \rfloor $ many samples, we need $B$ to be at least of order $n\epsilon^2/\log(n \epsilon^2)$ to maintain validity, and this number requires the knowledge of $\epsilon$. To see why the median bias of the sample median is large, let us consider independent $X_1, \ldots, X_n$ with exactly $1-\epsilon$ proportion generated from $P_{\theta}=N(\theta,1)$ and the rest of them being infinite. Denote $G \subseteq [n]$ as the index set of good samples. Then,
\begin{equation*}
	\begin{split}
	 \bbP( \htheta_{\textnormal{median}} \leq \theta ) &= \bbP\left( \sum_{i=1}^n \indi(X_i \leq \theta) \geq n/2 \right) \\
	 & =P_{\theta }\left( \sum_{i \in G} \indi(X_i \leq \theta) \geq n/2 \right)\\
	& = \bbP \left( \textnormal{Binomial}( (1-\epsilon)n, 1/2 ) - (1-\epsilon)n/2 \geq \epsilon n/2 \right) \\
	& \leq \exp( - n\epsilon^2 /2 ),
	\end{split}
\end{equation*} by Hoeffding's inequality. Thus, $\textnormal{Med-Bias}_{\theta}(\htheta_{\textnormal{median}}) \geq 1/2 - \exp( - n\epsilon^2 /2 )$. 
\end{itemize}

%%%%%%%%%%%%%%%%%%%%%%%%%%%%%%%%%%%%%%%%%%%
\section{Proofs in Section \ref{sec:adaptiveRCI}} \label{sec:app-adaptiveCI-additional-result}
%%%%%%%%%%%%%%%%%%%%%%%%%%%%%%%%%%%%%%%%%%
%%%%%%%%%%%%%%%%%%%%%%%%%%%%%%%%%%%%%%
\subsection{Proof of Theorem \ref{th:lower-bound-Gaussian} }
%%%%%%%%%%%%%%%%%%%%%%%%%%%%%%%%%%%%%

We prove the claim by a contradiction argument. Take $c$ to be the small constant in Lemma \ref{prop:test-lower-bound-gaussian}. Suppose $r_{\alpha}(\epsilon, [0,0.05]) > c \left( \frac{1}{\sqrt{\log(1/\epsilon)}} + \frac{1}{\sqrt{\log(n)}} \right)$ for all $\epsilon \in [0, 0.05]$ does not hold, i.e., there exists a $\epsilon \in [0,0.05]$ and an ARCI $\widehat{\CI}$ such that the following conditions hold
\begin{equation} \label{ineq:contradiction-condition}
	\inf_{ \epsilon \in [0, 0.05], \theta, Q} P_{\epsilon, \theta, Q}\left( \theta \in\widehat{\CI} \right) \geq 1-\alpha \quad \textnormal{ and } \quad \sup_{\theta, Q} P_{\epsilon, \theta, Q} \left( |\widehat{\CI}| \geq r \right) \leq \alpha,
\end{equation} with $r = c \left( \frac{1}{\sqrt{\log(1/\epsilon)}} + \frac{1}{\sqrt{\log(n)}} \right)$. Then given any $\theta \in \bbR$, by Proposition \ref{prop:CI-imply-test}, we can design a test using $\widehat{\CI}$ with the following guarantee
\begin{equation} \label{ineq:typeI+II}
\begin{split}
	 & \sup_{Q } P_{\epsilon_{\max}, \theta , Q } \left( \phi_{\theta, \theta- r, \epsilon} = 1  \right) + \sup_{Q } P_{ \epsilon, \theta - r, Q  } \left( \phi_{\theta, \theta- r, \epsilon} = 0  \right) \\
	\leq & \sup_{Q } P_{\epsilon_{\max}, \theta , Q } \left( \sup_{\epsilon \in [0, \epsilon_{\max} ]} \phi_{\theta, \theta- r, \epsilon} = 1  \right) + \sup_{Q } P_{ \epsilon, \theta - r, Q  } \left( \phi_{\theta, \theta- r, \epsilon} = 0  \right)  \leq  3 \alpha.
\end{split}
	\end{equation}

On the other hand, by Lemma \ref{prop:test-lower-bound-gaussian} and the Neyman-Person Lemma, we have the following lower bound for testing $\cH(\theta, \theta - r, \epsilon)$, 
	\begin{equation} \label{ineq:neyman-person}
		\begin{split}
		&\inf_{ \phi \in \{0,1 \} } \left\{ \sup_{Q } P_{\epsilon_{\max}, \theta, Q } \left( \phi = 1  \right) + \sup_{Q } P_{ \epsilon , \theta - r , Q  } \left( \phi = 0  \right) \right\}\\
			 &\geq 1 - \inf_{Q_0, Q_1} \TV\left(P^{\otimes n}_{\epsilon_{\max},\theta, Q_0}, P^{\otimes n}_{\epsilon, \theta - r , Q_1} \right) \\
			& \geq 1 - \alpha.
		\end{split}
	\end{equation}
	This contradicts with \eqref{ineq:typeI+II} since $\alpha< 1/4$. Thus, the assumption does not hold, and this finishes the proof of this theorem.

%%%%%%%%%%%%%%%%%%%%%%%%%%%%%%
\subsection{Proof of Theorem \ref{thm:upper-bound-general-Gaussian}}
%%%%%%%%%%%%%%%%%%%%%%%%%%%%%%%%%
First, we note that the coverage guarantee directly follows from Lemma \ref{th:estimation-guarantee}. Now, we move to the length guarantee, it is easy to check that $t_{\epsilon} \in \left[1.6, \Phi^{-1}\left(1 - \sqrt{ \log (2/\alpha)/(2n) } \right)\right]$ for all $\epsilon \in [0,0.05]$ when $n/\log(1/\alpha)$ is sufficiently large. Since Lemma \ref{th:estimation-guarantee} and \ref{th:estimation-length-guarantee} both only use concentration from the DKW inequality (see Lemma \ref{lm:DKW}), by the calculation in \eqref{ineq:length-guarantee}, we have $|\widehat{\CI}| \leq 4/t_{\epsilon}$ with probability at least $1-\alpha$. Finally, we have $1/t_{\epsilon} \lesssim  \frac{1}{\sqrt{\log(1/\epsilon)}} + \frac{1}{\sqrt{\log( n/(\log(1/\alpha)) )}}$ for all $\epsilon \in [0, \epsilon_{\max}]$ by Lemma \ref{lm:t-epsilon-bound}.

%%%%%%%%%%%%%%%%%%%%%%%%%%%%%
\subsection{Proof of Lemma \ref{th:estimation-guarantee}}
%%%%%%%%%%%%%%%%%%%%%%%%%%%%%

	For simplicity, we will denote the range of $t$ we are considering as $[t_{\min}, t_{\max}]$. By the DKW inequality (see Lemma \ref{lm:DKW}), we have with probability at least $1 - \alpha$, the following event holds:
	\begin{equation*}
		(A) = \{ \textnormal{for all } x \in \bbR, |F_n(x) - F_{\epsilon, \theta, Q}(x)| \leq \sqrt{ \log (2/\alpha)/(2n) } \},
	\end{equation*}where $F_{\epsilon, \theta, Q }(\cdot)$ denotes the CDF of $P_{\epsilon, \theta, Q}$.
	Next, let us deduce a sufficient condition on the error guarantee of $\htheta_{R}(t)$ given the event $(A)$ happens: for any $t \in [t_{\min}, t_{\max}]$,
	\begin{equation} \label{eq:right-estimator-reduction}
		\begin{split}
			& \quad \quad \htheta_{R}(t)  - \theta \geq -2/t \\
			&\Longleftrightarrow \theta + t - 2/t \leq F_n^{-1}( 1- 2(1 - \Phi(t) ) ) \\
			&\Longleftrightarrow \sum_{j=1}^n \indi\left( X_j < \theta + t - 2/t  \right) < n - 2n( 1 - \Phi(t) ) \\
			&\Longleftarrow \sum_{j=1}^n \indi\left( X_j \leq \theta + t - 2/t  \right) < n - 2n( 1 - \Phi(t) ) \\
			& \Longleftrightarrow 2( 1 - \Phi(t) ) < 1 - F_n(\theta +t -2/t) \\
			& \overset{(A)}\Longleftarrow 2( 1 - \Phi(t) ) < 1 - F_{\epsilon, \theta, Q }(\theta +t -2/t) - \sqrt{ \log (2/\alpha)/(2n) } \\
			& \Longleftrightarrow 2( 1 - \Phi(t) ) <P_{\epsilon, \theta, Q}(X > \theta +t -2/t) - \sqrt{ \log (2/\alpha)/(2n) }\\
			& \Longleftarrow 2( 1 - \Phi(t) ) < (1 - \epsilon_{\max} )P_{\theta}(X \geq \theta +t -2/t) - \sqrt{ \log (2/\alpha)/(2n) },
		\end{split}
	\end{equation} where $P_\theta$ denotes the distribution of $N(\theta,1)$. Similarly, for $\htheta_{L}(t)$, we have for any $t \in [t_{\min}, t_{\max}]$,
	\begin{equation}\label{eq:left-estimator-reduction}
		\begin{split}
			& \quad \quad \htheta_{L}(t)  - \theta \leq \frac{2}{t} \\
			& \Longleftrightarrow \theta - t + 2/t \geq F_n^{-1}( 2(1 - \Phi(t) ) ) \\
			& \Longleftrightarrow F_n\left(\theta - t + 2/t \right) \geq 2 (1 - \Phi(t) ) \\
			& \overset{(A)}\Longleftarrow  F_{\epsilon, \theta, Q} (\theta - t + 2/t ) - \sqrt{ \log (2/\alpha)/(2n) } \geq 2 (1 - \Phi(t) ) \\
			& \Longleftarrow (1 - \epsilon_{\max} )P_{\theta}(X \geq \theta +t -2/t) - \sqrt{ \log (2/\alpha)/(2n) }\geq 2 (1 - \Phi(t) ).
		\end{split}
	\end{equation} Based on \eqref{eq:right-estimator-reduction} and \eqref{eq:left-estimator-reduction}, to show $\htheta_{R}(t)  - \theta \geq -2/t$ and $\htheta_{L}(t)  - \theta \leq 2/t$ hold simultaneously for all $t \in [t_{\min}, t_{\max}]$ and uniformly over all $\epsilon\in[0,0.05]$, all $\theta\in\mathbb{R}$, and all $Q$ with probability $1 - \alpha$, it is enough to show 
	\begin{equation} \label{ineq:second-step-reduction}
		\begin{split}
			& 2( 1 - \Phi(t) ) < (1 - \epsilon_{\max} )P_{\theta}(X \geq \theta +t -2/t) - \sqrt{ \log (2/\alpha)/(2n) }, \forall t \in [t_{\min}, t_{\max}] \\
			\Longleftrightarrow & 2 P_{0}(X \geq t) < (1 - \epsilon_{\max} )P_{0}(X \geq t -2/t) - \sqrt{ \log (2/\alpha)/(2n) }, \forall t \in [t_{\min}, t_{\max}] \\
			\Longleftarrow & \left\{ \begin{array}{l}
				P_{0}(X \geq t) \geq \sqrt{ \log (2/\alpha)/(2n) }, \forall t \in [t_{\min}, t_{\max}] \\
				(1 - \epsilon_{\max} )P_{0}(X \geq t -2/t) > 3 P_{0}(X \geq t),\forall t \in [t_{\min}, t_{\max}]
			\end{array} \right. \\
			\overset{ (a) }\Longleftarrow & \left\{ \begin{array}{l}
				t_{\max} \leq \Phi^{-1} ( 1- \sqrt{ \log (2/\alpha)/(2n) } ), \\
				(1 - \epsilon_{\max} )\phi(x - 2/t) > 3 \phi(x),\forall x \geq t, \forall t \in [t_{\min}, t_{\max}]
			\end{array} \right. \\
			\Longleftrightarrow & \left\{ \begin{array}{l}
				t_{\max} \leq \Phi^{-1} ( 1- \sqrt{ \log (2/\alpha)/(2n) } ), \\
				\exp( 2 - 2/t^2_{\min} ) > \frac{3}{1- \epsilon_{\max} },
			\end{array} \right.
		\end{split}
	\end{equation} where in (a), $\phi(\cdot)$ denotes the density of $N(0,1)$ and it is because if $(1 - \epsilon_{\max} )\phi(x - 2/t) > 3 \phi(x),\forall x \geq t$, then
	\begin{equation} \label{ineq:density-ratio-to-tail-prob}
	\begin{split}
		(1 - \epsilon_{\max} )P_{0}(X \geq t -2/t) = (1 - \epsilon_{\max} )P_{2/t}(X \geq t) & = \int_{x \geq t} (1 - \epsilon_{\max} )\phi(x - 2/t) dx  \\
		& >  \int_{x \geq t} 3 \phi(x) dx = 3 P_{0}(X \geq t).
	\end{split}
	\end{equation}

	Plug in $\epsilon_{\max} = 0.05$. Then $t_{\min} = 1.6$ and $t_{\max} = \Phi^{-1} ( 1- \sqrt{ \log (2/\alpha)/(2n) } )$ satisfy the above conditions. This finishes the proof of this lemma.

%%%%%%%%%%%%%%%%%%%%%%%%%%%%%%%%%%%%%%%%
\subsection{Proof of Lemma \ref{th:estimation-length-guarantee}}
%%%%%%%%%%%%%%%%%%%%%%%%%%%%%%%%%%%%%5

	First, by the DKW inequality (see Lemma \ref{lm:DKW}), we have with probability at least $1 - \alpha$, the following event holds:
	\begin{equation*}
		(A) = \{ \textnormal{for all } x \in \bbR, |F_n(x) - F_{\epsilon, \theta, Q}(x)| \leq \sqrt{ \log (2/\alpha)/(2n) } \},
	\end{equation*}where $F_{\epsilon, \theta, Q }(\cdot)$ denotes the CDF of $P_{\epsilon, \theta, Q}$.
	
	Given $(A)$ happens, for any $t$, we have
	\begin{equation} \label{ineq:right-estimator-gua2}
		\begin{split}
			\htheta_{R}(t)  - \theta \leq 0 &\Longleftrightarrow F_n^{-1}( 1- 2(1 - \Phi(t) ) )  - t - \theta \leq 0 \\
			& \Longleftrightarrow 1 - F_n(\theta + t) \leq 2(1 - \Phi(t) ) \\
			& \overset{(A)}\Longleftarrow 1 - F_{\epsilon, \theta, Q}(\theta + t) + \sqrt{ \log (2/\alpha)/(2n) } \leq 2(1 - \Phi(t) ) \\
			& \Longleftrightarrow P_{\epsilon, \theta, Q} ( X > \theta + t ) + \sqrt{ \log (2/\alpha)/(2n) } \leq 2(1 - \Phi(t) ) \\
			& \Longleftarrow P_{\theta} ( X \geq \theta + t ) + \epsilon + \sqrt{ \log (2/\alpha)/(2n) } \leq 2(1 - \Phi(t) )\\
			& \Longleftrightarrow 1 - \Phi(t) + \epsilon + \sqrt{ \log (2/\alpha)/(2n) } \leq 2(1 - \Phi(t) ) \\
			& \Longleftrightarrow 1 - \Phi(t) \geq \epsilon + \sqrt{ \log (2/\alpha)/(2n) }.
		\end{split}
	\end{equation} Similarly, for any $t$, we have
	\begin{equation} \label{ineq:left-estimator-gua2}
		\begin{split}
			0 \leq \htheta_{L}(t)  - \theta &\Longleftrightarrow F_n^{-1}( 2(1 - \Phi(t) ) )  + t - \theta \geq 0 \\
			& \Longleftrightarrow \sum_{j=1}^n \indi\left( X_j < \theta -t \right) < 2n(1 - \Phi(t) ) \\
			& \Longleftarrow \sum_{j=1}^n \indi\left( X_j \leq  \theta -t \right) < 2n(1 - \Phi(t) ) \\
			& \Longleftrightarrow 2(1 - \Phi(t) ) > F_n(\theta -t ) \\
			& \overset{(A)}\Longleftarrow 2(1 - \Phi(t) ) > F_{ \epsilon, \theta, Q }(\theta -t ) + \sqrt{ \log (2/\alpha)/(2n) } \\
			& \Longleftarrow 2(1 - \Phi(t) ) \geq  P_{\theta}( X \leq \theta -t ) + \epsilon +  \sqrt{ \log (2/\alpha)/(2n) } \\
			& \Longleftrightarrow 2(1 - \Phi(t) ) \geq   1- \Phi(t) + \epsilon +  \sqrt{ \log (2/\alpha)/(2n) }\\
			& \Longleftrightarrow 1 - \Phi(t) \geq \epsilon + \sqrt{ \log (2/\alpha)/(2n) }.
		\end{split}
	\end{equation}
	Notice that if we plug in $ t = t_{\epsilon} = \Phi^{-1}\left( 1- \epsilon- \sqrt{ \log(2/\alpha)/(2n) } \right) $, then the sufficient conditions on the right-hand sides of \eqref{ineq:right-estimator-gua2} and \eqref{ineq:left-estimator-gua2} will be satisfied. This finishes the proof of this lemma.

\section{Proofs in Section \ref{sec:t}} \label{app:proof-test}

%%%%%%%%%%%%%%%%%%%%%%%%%
\subsection{Proof of Proposition \ref{prop:test-to-CI}}
%%%%%%%%%%%%%%%%%%%%%%%%%%

	Let us begin with the coverage guarantee. For any $\epsilon \in [0, \epsilon_{\max}]$ and $\theta \in \bbR$,
	\begin{equation} \label{eq:coverage-proof}
		\begin{split}
			&\sup_{Q} P_{\epsilon,\theta, Q}\left( \theta \notin \widehat{\CI} \right) \\
			&= \sup_{Q} P_{\epsilon,\theta, Q}\left(  \sup_{\epsilon \in [0, \epsilon_{\max}]} \phi_{\theta, \theta -r_{\epsilon}, \epsilon} = 1 \text{ or }  \sup_{\epsilon \in [0, \epsilon_{\max}]} \phi_{\theta, \theta +r_{\epsilon}, \epsilon} = 1  \right)\\
			& \leq  \sup_{Q}  P_{\epsilon,\theta, Q}\left(  \sup_{\epsilon \in [0, \epsilon_{\max}]} \phi_{\theta, \theta -r_{\epsilon}, \epsilon} = 1  \right) + \sup_{Q}  P_{\epsilon,\theta, Q}\left( \sup_{\epsilon \in [0, \epsilon_{\max}]} \phi_{\theta, \theta +r_{\epsilon}, \epsilon} = 1  \right) \\
			& \overset{(a)}\leq \sup_{Q}  P_{\epsilon_{\max},\theta, Q}\left(\sup_{\epsilon \in [0, \epsilon_{\max}]} \phi_{\theta, \theta -r_{\epsilon}, \epsilon} = 1  \right) +  \sup_{Q}  P_{\epsilon_{\max},\theta, Q}\left(\sup_{\epsilon \in [0, \epsilon_{\max}]} \phi_{\theta, \theta +r_{\epsilon}, \epsilon} = 1  \right) \\
			& \leq 2\alpha,
		\end{split}
	\end{equation}  where (a) is because $\{ P_{\epsilon, \theta, Q }: Q, \theta \} \subseteq \{ P_{\epsilon_{\max}, \theta, Q}: Q, \theta \}$ since $\epsilon \in [0, \epsilon_{\max}]$; the last inequality is by the simultaneous Type-1 error control of the tests.

	Next, we show the length guarantee. For any $\theta \in \bbR$ and $\epsilon \in [0, \epsilon_{\max}]$,
	\begin{equation} \label{eq:length-guarantee}
		\begin{split}
			\sup_{Q} P_{\epsilon,\theta, Q}\left( |\widehat{\CI}| \geq 2 r_{ \epsilon } \right) &\leq \sup_{Q} P_{\epsilon,\theta, Q}\left( |\widehat{\CI}| \geq 2 r_{ \epsilon }, \theta \in \widehat{\CI} \right) \\
			& \quad + \sup_{Q} P_{\epsilon,\theta, Q}\left( \theta \notin \widehat{\CI} \right).
		\end{split}
	\end{equation} Notice that the second term on the right-hand side of \eqref{eq:length-guarantee} is bounded by $2\alpha$ by the coverage guarantee we have shown in the first part. To finish the proof, we just need to show the first term on the right-hand side of \eqref{eq:length-guarantee} is bounded by $2\alpha$,
	\begin{equation*}
		\begin{split}
			 & \sup_{Q} P_{\epsilon,\theta, Q}\left( |\widehat{\CI}| \geq 2 r_{ \epsilon }, \theta \in \widehat{\CI} \right) \\
			&\overset{(a)}\leq \sup_{Q} P_{\epsilon,\theta, Q}\left( \theta + r_{ \epsilon } \in \widehat{\CI} \text{ or } \theta - r_{ \epsilon } \in \widehat{\CI} \right) \\
			& \leq \sup_{Q} P_{\epsilon,\theta, Q}\left( \theta + r_{ \epsilon } \in \widehat{\CI}  \right) + \sup_{Q} P_{\epsilon,\theta, Q}\left( \theta - r_{ \epsilon } \in \widehat{\CI}  \right)\\
			& \leq \sup_{Q} P_{\epsilon,\theta, Q}\left( \phi_{\theta + r_{ \epsilon }, \theta, \epsilon} =0  \right) + \sup_{Q} P_{\epsilon,\theta, Q}\left( \phi_{\theta - r_{ \epsilon }, \theta, \epsilon} = 0 \right) \\
			& \leq  \sup_{\theta,Q} P_{\epsilon,\theta - r_{ \epsilon }, Q}\left( \phi_{\theta, \theta - r_{ \epsilon }, \epsilon} =0  \right) + \sup_{\theta,Q} P_{\epsilon,\theta + r_{ \epsilon }, Q}\left( \phi_{\theta, \theta + r_{ \epsilon }, \epsilon} = 0 \right) \\
			& \leq 2 \alpha,
		\end{split}
	\end{equation*} where (a) is by the assumption $\widehat{\CI}$ is an interval and the last inequality is by the Type-2 error control of the tests $\phi_{\theta, \theta + r_{ \epsilon }, \epsilon}$ and $\phi_{\theta, \theta - r_{ \epsilon }, \epsilon}$. %This finishes the proof. 

%%%%%%%%%%%%%%%%%%%%%%%%%%%%%%%%%
\subsection{Proof of Proposition \ref{prop:CI-imply-test}}
%%%%%%%%%%%%%%%%%%%%%%%%%%%%%%%%%
The proof of Proposition \ref{prop:CI-imply-test} is similar to that of Proposition 8.3.6 in \cite{gine2015mathematical}.  
 We start with the simultaneous Type-1 error guarantee for $\phi_{\theta, \theta- r_{\epsilon}, \epsilon}$ and $\phi_{\theta, \theta+ r_{\epsilon}, \epsilon}$.
\begin{equation*}
	\begin{split}
		 & \sup_{Q} P_{\epsilon_{\max}, \theta , Q } \left( \sup_{ \epsilon \in [0, \epsilon_{\max}] } \phi_{\theta, \theta- r_{\epsilon}, \epsilon} = 1  \right)  = \sup_{Q} P_{\epsilon_{\max}, \theta, Q } \left( \theta \notin  \widehat{\CI} \right)  \leq \alpha, \\
		 &\sup_{Q} P_{\epsilon_{\max}, \theta , Q } \left( \sup_{ \epsilon \in [0, \epsilon_{\max}] } \phi_{\theta, \theta + r_{\epsilon}, \epsilon} = 1  \right)  = \sup_{Q} P_{\epsilon_{\max}, \theta, Q } \left( \theta \notin  \widehat{\CI} \right)  \leq \alpha,
	\end{split}
\end{equation*}	where the equality is by the definition of $\phi_{\theta, \theta- r_{\epsilon}, \epsilon}$ and $\phi_{\theta, \theta+ r_{\epsilon}, \epsilon}$ and the inequality is by the coverage guarantee of $\widehat{\CI}$. Next, we bound the Type-2 error of $\phi_{\theta, \theta- r_{\epsilon}, \epsilon}$ and the same result holds for $\phi_{\theta, \theta+ r_{\epsilon}, \epsilon}$.	For any $\epsilon \in [0, \epsilon_{\max}]$ such that $\inf_{\theta, Q } P_{\epsilon, \theta, Q}\left( |\widehat{\CI}| \leq r_{\epsilon}   \right) \geq 1 - \alpha$, we have
\begin{equation*}  %\label{ineq:type-II-error}
		\begin{split}
			& \sup_{Q} P_{ \epsilon, \theta - r_{\epsilon}, Q  } \left( \phi_{\theta, \theta- r_{\epsilon}, \epsilon} = 0  \right) \\
			&\overset{(a)}= \sup_{Q} P_{ \epsilon, \theta - r_{\epsilon}, Q  } \left( \theta   \in \widehat{\CI}   \right) \\
			& \leq \sup_{Q} P_{\epsilon, \theta - r_{\epsilon}, Q } \left( \theta \in  \widehat{\CI}, \theta -  r_{\epsilon}\in \widehat{\CI} \right) + \sup_{Q} P_{\epsilon, \theta - r_{\epsilon}, Q } \left( \theta -  r_{\epsilon}\notin  \widehat{\CI} \right) \\
			& \overset{(b)}\leq \sup_{Q} P_{\epsilon, \theta - r_{\epsilon}, Q } \left( |\widehat{\CI}| \geq r_{\epsilon}\right) + \sup_{Q} P_{\epsilon, \theta - r_{\epsilon}, Q } \left( \theta -  r_{\epsilon}\notin  \widehat{\CI} \right)  \\
			& \leq \alpha + \sup_{Q} P_{\epsilon, \theta - r_{\epsilon}, Q } \left( |\widehat{\CI}| \geq r_{\epsilon}\right) \leq \alpha + \sup_{\theta,Q } P_{\epsilon, \theta, Q } \left( |\widehat{\CI}| \geq r_{\epsilon}\right) \\
			& \leq 2 \alpha.
		\end{split}
	\end{equation*}
where (a) is by the definition of $\phi_{\theta, \theta- r_{\epsilon}, \epsilon}$ and in (b), we use the fact $\widehat{\CI}$ is an interval, the second to last inequality is by the adaptive coverage property of $\widehat{\CI}$ and the last inequality is by the length guarantee of $\widehat{\CI}$.

%%%%%%%%%%%%%%%%%%%%%%%%%%%%%%%%%
\subsection{Proof of Lemma \ref{prop:upper-bound-Gaussian}}
%%%%%%%%%%%%%%%%%%%%%%%%%%%%%%%%%
Due to symmetry, here we only present the proof for the error control of $\phi_{\theta, \theta- r_{\epsilon}, \epsilon}$ for solving $\cH(\theta, \theta- r_{\epsilon}, \epsilon)$.
	Let us begin with the simultaneous Type-1 error control, and we will show that an equivalent statement is true: for any $\theta \in \bbR$,
	\begin{equation*}
		\inf_{Q}P_{ \epsilon_{\max}, \theta, Q }\left( \phi_{\theta, \theta- r_{\epsilon}, \epsilon} = 0 \text{ for all } \epsilon \in [0,0.05] \right) \geq 1- \alpha.
		\end{equation*} First, by the DKW inequality (see Lemma \ref{lm:DKW}), we have with probability at least $1 - \alpha$, the following event holds:
	\begin{equation*}
		(A) = \{ \textnormal{for all } x \in \bbR, |F_n(x) - F_{\epsilon_{\max}, \theta, Q}(x)| \leq \sqrt{ \log (2/\alpha)/(2n) } \},
	\end{equation*}where $F_{\epsilon_{\max}, \theta, Q }(\cdot)$ denotes the CDF of $P_{\epsilon_{\max}, \theta, Q}$. Notice that $r_{\epsilon} = 2/t_{\epsilon}$ by definition, then under $P_{ \epsilon_{\max}, \theta, Q }$, given $(A)$ happens, for any $\epsilon \in [0,0.05]$,
	\begin{equation} \label{ineq:type-I-error-reduction}
		\begin{split}
			\phi_{\theta, \theta- 2/t_{\epsilon}, \epsilon} = 0 &\Longleftrightarrow \sum_{j =1}^n \indi(X_j - (\theta - 2/t_{\epsilon}) \geq t_{\epsilon}) > 2n( 1-  \Phi(t_{\epsilon}) ) \\
			& \Longleftarrow \sum_{j =1}^n \indi(X_j > t_{\epsilon} + \theta - 2/t_{\epsilon}) > 2n( 1-  \Phi(t_{\epsilon}) ) \\
			&\Longleftarrow \frac{1}{n} \sum_{j=1}^n \indi\left( X_j \leq \theta + t_{\epsilon} - 2/t_{\epsilon}  \right) < 1 - 2( 1 - \Phi(t_{\epsilon}) ) \\
			& \overset{(A)}\Longleftarrow 2( 1 - \Phi(t_{\epsilon}) ) < 1 - F_{\epsilon_{\max}, \theta, Q }(\theta +t_{\epsilon} -2/t_{\epsilon}) - \sqrt{ \log (2/\alpha)/(2n) } \\
			& \Longleftarrow 2P_{0}(X \geq t_{\epsilon}) < (1 - \epsilon_{\max} )P_{\theta}(X \geq \theta +t_{\epsilon} -2/t_{\epsilon}) - \sqrt{ \log (2/\alpha)/(2n) }\\
			& \Longleftarrow \left\{ \begin{array}{l}
				P_{0}(X \geq t_{\epsilon}) \geq \sqrt{ \log (2/\alpha)/(2n) }  \\
				(1 - \epsilon_{\max} )P_{0}(X \geq t_{\epsilon} -2/t_{\epsilon}) > 3 P_{0}(X \geq t_{\epsilon})
			\end{array} \right. \\
			& \overset{\eqref{ineq:density-ratio-to-tail-prob} }\Longleftarrow \left\{ \begin{array}{l}
				t_{\epsilon} \leq \Phi^{-1} ( 1- \sqrt{ \log (2/\alpha)/(2n) } ), \\
				(1 - \epsilon_{\max} ) \phi(x -2/t_{\epsilon} ) > 3 \phi(x),\forall x \geq t_{\epsilon},
			\end{array} \right. \\
			& \Longleftrightarrow  \left\{ \begin{array}{l}
				t_{\epsilon} \leq \Phi^{-1} ( 1- \sqrt{ \log (2/\alpha)/(2n) } ), \\
				\exp( 2 - 2/t^2_{\epsilon} ) > \frac{3}{1- \epsilon_{\max} }.
			\end{array} \right.
		\end{split}
	\end{equation} 
	
 In addition, since $t_{\epsilon}$ is a decreasing function with respect to $\epsilon$, to guarantee $\phi_{\theta, \theta- 2/t_{\epsilon}, \epsilon} = 0$ for all $\epsilon \in [0,0.05]$, by 
 \eqref{ineq:type-I-error-reduction}, it is enough to have
	\begin{equation*}
		t_{0} \leq \Phi^{-1} ( 1- \sqrt{ \log (2/\alpha)/(2n) } ) \, \textnormal{ and } \, \exp( 2 - 2/t^2_{0.05} ) > \frac{3}{1- \epsilon_{\max} }.
	\end{equation*} It is easy to check the above conditions are satisfied when $n/\log(1/\alpha)$ is large enough. Thus, we have shown that when $(A)$ happens, then $\phi_{\theta, \theta- 2/t_{\epsilon}, \epsilon} = 0$ for all $\epsilon \in [0,0.05]$. This finishes the proof for simultaneous Type-1 error control.
	
	Now, let us bound the Type-2 error. For any $\theta, Q$,
		\begin{equation*}
			\begin{split}
				&P_{ \epsilon, \theta - r_{\epsilon}, Q }\left( \phi_{\theta, \theta- r_{\epsilon}, \epsilon} = 0 \right) \\
				=& P_{ \epsilon, \theta - r_{\epsilon}, Q }\left( \sum_{j=1}^n \indi(X_j - (\theta - r_{\epsilon}) \geq t_{\epsilon})  > 2n( 1-  \Phi(t_{\epsilon}) )  \right) \\
				= & P_{ \epsilon, \theta, Q }\left( \sum_{j=1}^n \indi(X_j - \theta \geq t_{\epsilon})  > 2n( 1-  \Phi(t_{\epsilon}) )\right) \\
				= & P_{ \epsilon, \theta, Q }\left( \sum_{j=1}^n \indi(X_j  \geq \theta +  t_{\epsilon}) - n P_{ \epsilon, \theta, Q }(X \geq \theta + t_{\epsilon})  > 2n( 1-  \Phi(t_{\epsilon}) ) - n P_{ \epsilon, \theta, Q }(X \geq \theta + t_{\epsilon})\right) \\
				\leq &  P_{ \epsilon, \theta, Q }\left( \sum_{j=1}^n \indi(X_j  \geq \theta +  t_{\epsilon}) - n P_{ \epsilon, \theta, Q }(X \geq \theta + t_{\epsilon})  > 2n( 1-  \Phi(t_{\epsilon}) ) - n ( \epsilon + P_{\theta}(X \geq \theta + t_{\epsilon}) )\right) \\
				\overset{(a)}= &  P_{ \epsilon, \theta, Q }\left( \sum_{j=1}^n \indi(X_j  \geq \theta +  t_{\epsilon}) - n P_{ \epsilon, \theta, Q }(X \geq \theta + t_{\epsilon})  > \sqrt{n\log(2/\alpha)/2} \right) \\
				\overset{(b)}\leq & \alpha
			\end{split}
		\end{equation*} where in (a), we plug in the value of $t_{\epsilon}$ and (b) is by Hoeffding's inequality.

%%%%%%%%%%%%%%%%%%%%%%%%%%%%%%%%%
\subsection{Proof of Lemma \ref{prop:test-lower-bound-gaussian}}
%%%%%%%%%%%%%%%%%%%%%%%%%%%%%%%%%
 This result can be deduced from the general result in Theorem \ref{th:test-lower-bound}, but here, we provide a clean proof for the Gaussian case with an explicit construction of the hard contamination distributions. Due to symmetry and location family property, we just need to show $\inf_{Q_0, Q_1} \TV(P^{\otimes n}_{\epsilon_{\max}, r, Q_0}, P^{\otimes n}_{\epsilon, 0, Q_1}) \leq \alpha$. We will show the result separately for two terms: when $r \leq c/\sqrt{\log n}$ or $r \leq c/\sqrt{\log(1/\epsilon )}$ for some small $c$ depending on $\alpha$ only, we can find $Q_0$ and $Q_1$ such that $\TV(P^{\otimes n}_{ \epsilon_{\max}, r, Q_0}, P^{\otimes n}_{\epsilon, 0, Q_1}) \leq \alpha$.

{\bf (Case 1: $r \leq c/\sqrt{\log n}$)} Take $Q_1 = N(0, 1)$ and define the density function of $Q_0$ by 
\begin{equation} \label{eq:gaussian-q0}
q_0(x) = \frac{1}{\epsilon_{\max}}\phi(x)-\frac{1-\epsilon_{\max}}{\epsilon_{\max} } \frac{\indi\left\{x - r\leq t\right\}}{\bbP(N(0,1) \leq t)} \phi(x-r).
\end{equation} It is easy to check $\int q_0(x) dx = 1$, so to guarantee $Q_0$ is a legitimate probability distribution, we need 
\begin{equation} \label{eq:Gaussian-valid-density-condition}
	\begin{split}
		&\phi(x) \geq \frac{(1-\epsilon_{\max}) \phi(x - r)  }{\bbP(N(0,1) \leq t)} \,\text{ for all } x \leq r + t \\
		\Longleftrightarrow & 1 \geq \frac{(1-\epsilon_{\max}) \exp(xr - r^2/2)  }{\bbP(N(0,1) \leq t)} \,\text{ for all } x \leq r + t  \\
		\Longleftrightarrow & 1 \geq \frac{(1-\epsilon_{\max}) \exp(tr + r^2/2)  }{\bbP(N(0,1) \leq t)}.
	\end{split}
\end{equation} At the same time,
\begin{equation} \label{eq:Gaussian-TV-bound}
	\begin{split}
		\TV( P_{\epsilon_{\max}, r, Q_0}, P_{\epsilon, 0, Q_1} ) &= \frac{1}{2} \int | (1-\epsilon_{\max} ) \phi(x - r) + \epsilon_{\max} q_0(x) - \phi(x) | dx \\
		& = \frac{1}{2} \int (1-\epsilon_{\max} ) \left|1  - \frac{\indi\left\{x - r\leq t\right\}}{\bbP(N(0,1) \leq t)}\right| \phi(x - r) dx \\
		& = (1-\epsilon_{\max}) \bbP(N(0,1) \geq t).
	\end{split}
\end{equation} We take $t = \Phi^{-1}(1 - \alpha/n)$, then 
\begin{equation*}
	\begin{split}
		\TV(P^{\otimes n}_{ \epsilon_{\max} , r, Q_0}, P^{\otimes n}_{\epsilon, 0, Q_1}) \overset{(a)}\leq n \TV( P_{\epsilon_{\max}, r, Q_0}, P_{\epsilon, 0, Q_1} ) \overset{(b)}\leq \alpha,
	\end{split}
\end{equation*} where (a) is by the property of TV distance on the product measure; (b) is by \eqref{eq:Gaussian-TV-bound} and the choice of $t$. Finally, when $t = \Phi^{-1}(1 - \alpha/n)$ and $r \leq c/\sqrt{\log n} $ for some sufficiently small constant $c$ depending on $\alpha$ only, we have $r \leq c'/t$ by Lemma \ref{lm:normal-quantile} with another small $c' > 0$ depending on $c$ only and the following holds
\begin{equation*}
	\begin{split}
		\exp(tr + r^2/2) \leq \exp(c' + c^2/\log n) \leq \frac{1 - \epsilon_{\max}/2 }{1 - \epsilon_{\max} } \leq \frac{1  - \alpha/n }{1  - \epsilon_{\max} },
	\end{split}
\end{equation*} where the last inequality is because $n \geq 100$ so that $\epsilon_{\max}/2 \geq \alpha/n$. This shows that \eqref{eq:Gaussian-valid-density-condition} holds and $Q_0$ is a valid distribution.

{\bf (Case 2: $r \leq c/\sqrt{\log (1/\epsilon )}$)} It is enough to show that when $\epsilon \in [0, \epsilon_{\max}/2]$, if $r \leq c/\sqrt{\log (1/\epsilon )}$, then $\inf_{Q_0,Q_1}\TV(P^{\otimes n}_{ \epsilon_{\max} , r, Q_0}, P^{\otimes n}_{\epsilon, 0, Q_1}) \leq \alpha$. This is because when $\epsilon \in [\epsilon_{\max}/2, \epsilon_{\max}]$, there exists another constant $c'$ such that when $r \leq c'/\sqrt{\log (1/\epsilon )} \leq c/\sqrt{\log (2/\epsilon_{\max} )}$, then $$\inf_{Q_0,Q_1}\TV(P^{\otimes n}_{ \epsilon_{\max} , r, Q_0}, P^{\otimes n}_{\epsilon, 0, Q_1}) \leq \inf_{Q_0,Q_1} \TV(P^{\otimes n}_{ \epsilon_{\max} , r, Q_0}, P^{\otimes n}_{\epsilon_{\max}/2, 0, Q_1}) \leq  \alpha,$$ where the first inequality is because $\{ P_{\epsilon_{\max}/2, 0, Q_1}: Q_1 \} \subseteq \{P_{\epsilon, 0, Q_1}: Q_1 \}$ when $\epsilon \geq \epsilon_{\max}/2$.

In this case, we take $Q_1 = N( \sqrt{ \log(1/ \epsilon ) }, 1 )$ and define the density of $Q_0$ as
\begin{equation*}
	q_0(x) = \frac{(1-\epsilon) \phi(x) + \epsilon q_1(x) - (1 - \epsilon_{\max} )\phi(x - r) }{ \epsilon_{\max} } \, \textnormal{ for all } x \in \bbR.
\end{equation*} The intuition is that given $Q_0$ is a valid distribution, the choice of $Q_0$ above can exactly match $P_{\epsilon_{\max}, r, Q_0 }$ and $P_{\epsilon, 0, Q_1 }$, i.e., $\TV(P^{\otimes n}_{ \epsilon_{\max} , r, Q_0}, P^{\otimes n}_{\epsilon_{\max}/2, 0, Q_1}) = 0$. The choice of $Q_1$ comes from the requirement that $Q_0$ needs to be a valid distribution. 

Next, we verify that when $r \lesssim 1/\sqrt{\log(1 /\epsilon )} $, $Q_0$ is a valid distribution. It is easy to check $\int q_0(x) dx = 1$, so we just need to show $q_0(x) \geq 0$ for all $x \in \bbR$.

{\bf Regime 1: $x \geq 2 \sqrt{\log(1/\epsilon )} $.} In this case
\begin{equation} \label{ineq:large-x-regime}
	\begin{split}
		 &\epsilon q_0(x) \geq (1 - \epsilon_{\max} )\phi(x - r) \\
		 \Longleftrightarrow &  \epsilon \phi\left(x -  \sqrt{\log(1/\epsilon )}\right) \geq (1 - \epsilon_{\max} )\phi(x - r) \\
		  \Longleftrightarrow & \exp\left( x \sqrt{\log(1/\epsilon )} - \log(1/\epsilon )/2 \right) \geq \frac{1 - \epsilon_{\max} }{\epsilon} \exp(xr - r^2/2) \\
		  \Longleftarrow & \exp\left( x (\sqrt{\log(1/\epsilon )}-r) - \log(1/\epsilon )/2 \right)  \geq \frac{1 - \epsilon_{\max} }{\epsilon} \\
		 \overset{(a)} \Longleftarrow & \exp\left( 3x \sqrt{\log(1/\epsilon )}/4 - \log(1/\epsilon )/2 \right)  \geq \frac{1 - \epsilon_{\max} }{\epsilon}
	\end{split}
\end{equation} where in (a) we use the fact that $r \lesssim 1/\sqrt{\log(1 /\epsilon )} \leq \sqrt{\log(1/\epsilon )}/4$. Notice that the last condition in \eqref{ineq:large-x-regime} is satisfied for all $x \geq  2 \sqrt{\log(1/\epsilon )} $. So in this regime, $q_0(x) \geq 0$.

{\bf Regime 2: $x \leq 2 \sqrt{\log(1/\epsilon )} $.} In this case 
\begin{equation}\label{ineq:small-x-regime}
	\begin{split}
		(1-\epsilon) \phi(x) \geq (1 - \epsilon_{\max} )\phi(x - r) &\overset{(a)}\Longleftarrow (1-\epsilon_{\max}/2) \phi(x) \geq (1 - \epsilon_{\max} )\phi(x - r) \\
		& \Longleftrightarrow \frac{1-\epsilon_{\max}/2}{1 - \epsilon_{\max}}  \geq \exp(xr - r^2/2) \\
		& \Longleftarrow \frac{1-\epsilon_{\max}/2}{1 - \epsilon_{\max}}  \geq \exp(xr)
	\end{split}
\end{equation} where in (a) we use the fact $\epsilon \leq \epsilon_{\max}/2$. Notice that the last condition in \eqref{ineq:small-x-regime} is true when $r \lesssim 1/\sqrt{\log(1/\epsilon )}$ as $x \leq 2 \sqrt{\log(1/\epsilon )} $. This shows that in this regime, $q_0(x) \geq 0$. This finishes the proof.

%%%%%%%%%%%%%%%%%%%%%%%%%%%%%%%%%%%%%%%%%%%%%
\section{Proofs in Section \ref{sec:ARCI-general-distribution}}
%%%%%%%%%%%%%%%%%%%%%%%%%%%%%%

\subsection{Proof of Proposition \ref{prop:laplace}} The result follows from the generalized Gaussian result with $\beta = 1$ in Theorem \ref{thm:example} and its intuition has also been provided in Section \ref{sec:laplace-example}.

%%%%%%%%%%%%%%%%%%%%%%%%%%%%%%
\subsection{Proof of Theorem \ref{thm:upper-bound-general}}
%%%%%%%%%%%%%%%%%%%%%%%%%%%%%%
To show Theorem \ref{thm:upper-bound-general}, let us first introduce a more general result for solving the testing problems $\cH(\theta, \theta \pm r_{\epsilon}, \epsilon)$. Due to symmetry, we will state the result for $\cH(\theta, \theta - r_{\epsilon}, \epsilon)$ only. Suppose we use tests 
\begin{equation} \label{eq:test}
	\begin{split}
		\phi_{\theta, \theta- r_{\epsilon}, \epsilon} &= \indi \left\{ \sum_{j=1}^n \indi \left \{ X_j - (\theta - r_{\epsilon} ) \geq t_{\epsilon} \right\} \leq \rho_{\epsilon} \right\}
	\end{split}
\end{equation} for solving a family of robust testing problems $\{\cH(\theta, \theta - r_{\epsilon}, \epsilon): \epsilon \in [0, \epsilon_{\max}]\}$. Then, we have the following result.

\begin{Theorem} \label{th:testing-upper-general-coeff} Consider the testing problems $\{\cH(\theta, \theta- r_{\epsilon}, \epsilon): \epsilon \in [0, \epsilon_{\max}] \}$ in \eqref{eq:test-def} and the testing procedure  $\phi_{\theta, \theta- r_{\epsilon}, \epsilon}$ in \eqref{eq:test}. Suppose $n \geq 400$. For any $\alpha \in (0,1)$ and $\eta \in (0,1)$, if $r_{\epsilon}$ and $(t_{\epsilon}, \rho_{\epsilon})$ in $\phi_{\theta, \theta- r_{\epsilon}, \epsilon}$ satisfy the following three conditions,
		\begin{equation} \label{ineq:testing-upper-bound-general-cond1} 
	\begin{split}
				 &\rho_{\epsilon} \geq (1 + \eta) n \left(  \epsilon + P_{0}(X \geq t_{\epsilon} )\right) +  \left(2/3 + 1/(2\eta)\right) \log(1/\alpha),\\
				 & (1- \epsilon_{\max} )^2 P_{r_{\epsilon}}(X \geq t_{\epsilon}) (1 - P_{r_{\epsilon}}(X \geq t_{\epsilon})) \geq 100 \log(8/\alpha)/n,  \\
				 &(1- \epsilon_{\max} )  n P_{r_{\epsilon}}(X \geq t_{\epsilon}) \geq 2\rho_{\epsilon} .
	\end{split}
	\end{equation} for all $\epsilon \in [0, \epsilon_{\max}]$, then 
	\begin{equation} \label{ineq:general-testing-guarantee}
		\begin{split}
			\textnormal{(simultaneous Type-1 error)} &\quad \sup_QP_{\epsilon_{\max},\theta,Q}\left(\sup_{\epsilon\in[0,\epsilon_{\max}]}\phi_{\theta,\theta -r_{\epsilon},\epsilon}=1\right)\leq \alpha,\\
			\textnormal{(Type-2 error)} &\quad \sup_QP_{\epsilon,\theta - r_{\epsilon},Q}\left(\phi_{\theta,\theta - r_{\epsilon},\epsilon}=0\right)\leq \alpha,
		\end{split}
	\end{equation} for all $\theta\in\mathbb{R}$ and $\epsilon\in[0,\epsilon_{\max}]$.
\end{Theorem} Theorem \ref{th:testing-upper-general-coeff} will be proved in the following subsection. Now, let us prove Theorem \ref{thm:upper-bound-general} using Theorem \ref{th:testing-upper-general-coeff}. Specifically, we replace $\alpha$ with $\alpha/4$ and set $\eta = 1/2$ in Theorem \ref{th:testing-upper-general-coeff}, then Theorem \ref{th:testing-upper-general-coeff} shows when the following three conditions
	\begin{equation} \label{ineq:final-two-conditions-simple}
		\begin{split}
				 &\rho_{\epsilon} \geq \frac{3}{2} n \left(  \epsilon + P_{0}(X \geq t_{\epsilon} )\right) +  \frac{5}{3} \log(4/\alpha),\\
				 & (1- \epsilon_{\max} )^2 P_{r_{\epsilon}}(X \geq t_{\epsilon}) (1 - P_{r_{\epsilon}}(X \geq t_{\epsilon})) \geq 100 \log(32/\alpha)/n,\\
				 & (1- \epsilon_{\max} )  n P_{r_{\epsilon}}(X \geq t_{\epsilon}) \geq 2\rho_{\epsilon},
		\end{split}
	\end{equation} hold for all $\epsilon \in [0, \epsilon_{\max}]$,	then \eqref{ineq:general-testing-guarantee} holds. Now let us further simplify the three conditions in \eqref{ineq:final-two-conditions-simple}, %When $n \geq 100$, we have $\eta(n, \alpha/4 ) \leq 2\log( 2(\log n)/\alpha )/n$, thus
\begin{equation} \label{ineq:three-conditions-with-constant}
	\begin{split}
		\eqref{ineq:testing-upper-bound-general-cond1}
			&\Longleftarrow \left\{ \begin{array}{l}
				\rho_{\epsilon} = \frac{3}{2} n \left(  \epsilon + P_{0}(X \geq t_{\epsilon} )\right) +  \frac{5}{3} \log(4/\alpha), \\
				t_{\epsilon} \geq r_{\epsilon},\\
				P_{r_{\epsilon}}(X \geq t_{\epsilon}) \geq \frac{200}{(1- \epsilon_{\max})^2} \log(32/\alpha)/n,\\
				(1- \epsilon_{\max} )  n P_{r_{\epsilon}}(X \geq t_{\epsilon}) \geq 3 n \left(  \epsilon + P_{0}(X \geq t_{\epsilon} )\right) +  \frac{10}{3} \log(4/\alpha)
			\end{array} \right. \\
			&\Longleftarrow \left\{ \begin{array}{l}
				\rho_{\epsilon} = \frac{3}{2} n \left(  \epsilon + P_{0}(X \geq t_{\epsilon} )\right) +  \frac{5}{3} \log(4/\alpha), \\
				t_{\epsilon} \geq r_{\epsilon},\\
				P_{r_{\epsilon}}(X \geq t_{\epsilon}) \geq \frac{200}{(1- \epsilon_{\max})^2} \log(32/\alpha)/n,\\
				\frac{1}{2} (1- \epsilon_{\max} )  nP_{r_{\epsilon}}(X \geq t_{\epsilon})  \geq 3 nP_{0}(X \geq t_{\epsilon} ),\\
				 \frac{1}{2} (1- \epsilon_{\max} )  n P_{r_{\epsilon}}(X \geq t_{\epsilon}) \geq 3 n \epsilon +   \frac{10}{3} \log(4/\alpha),
			\end{array} \right. \\
			&\Longleftrightarrow \left\{ \begin{array}{l}
				\rho_{\epsilon} = \frac{3}{2} n \left(  \epsilon + P_{0}(X \geq t_{\epsilon} )\right) +  \frac{5}{3} \log(4/\alpha), \\
				t_{\epsilon} \geq r_{\epsilon},\\
				1 - F(t_{\epsilon} - r_{\epsilon} ) -  \frac{6}{1 - \epsilon_{\max} } (1 - F(t_{\epsilon})) \geq 0, \\
				P_{r_{\epsilon}}(X \geq t_{\epsilon}) \geq \bar{q}(\epsilon ).
			\end{array} \right.
	\end{split}
\end{equation} Here in the last expression, we use the definition of $\overline{q}(\epsilon) = \frac{6}{(1-\epsilon_{\max})} (\epsilon + \frac{100 \log(32/\alpha)}{3(1-\epsilon_{\max})n} )$. Notice that $\bar{r}( \epsilon)$ in the theorem is the smallest $r$ allowed given the conditions in \eqref{ineq:three-conditions-with-constant} are satisfied and it can be achieved since in our setting CDF function is a continuous function. In addition, $t_{\epsilon}$ is also chosen to satisfy the conditions in \eqref{ineq:three-conditions-with-constant} given $\bar{r}(\epsilon )$. This finishes the proof of this theorem.

%%%%%%%%%%%%%%%%%%%%%%%%%%%%%%%%%%%%%%%%%%%%%
\subsubsection{Proof of Theorem \ref{th:testing-upper-general-coeff}}
%%%%%%%%%%%%%%%%%%%%%%%%%%%%%%%%%%%%%%%%%%%%
	 Let us begin with the Type-2 error control. Fix any $\theta, Q$, we have
	\begin{equation*}
		\begin{split}
			&P_{ \epsilon, \theta - r_{\epsilon}, Q }\left( \phi_{\theta, \theta- r_{\epsilon}, \epsilon} = 0 \right) \\
			& =P_{ \epsilon, \theta - r_{\epsilon}, Q }\left( \sum_{j=1}^n \indi \left \{ X_j - (\theta - r_{\epsilon} ) \geq t_{\epsilon} \right\} > \rho_{\epsilon} \right) \\
			& = P_{ \epsilon, 0, Q }\left( \sum_{j=1}^n \indi \left \{ X_j  \geq t_{\epsilon} \right\} > \rho_{\epsilon} \right)  \\
			& \leq  P_{ \epsilon, 0, Q }\left( \sum_{j=1}^n \indi \left \{ X_j  \geq t_{\epsilon} \right\} - n P_{ \epsilon, 0, Q }(X \geq t_{\epsilon}) \geq \rho_{\epsilon} - n P_{ \epsilon, 0, Q }(X \geq t_{\epsilon}) \right) \\
			& \leq \exp\left( - \frac{ \frac{1}{2} \left( \rho_{\epsilon} - n P_{ \epsilon, 0, Q }(X \geq t_{\epsilon}) \right)^2 }{n P_{ \epsilon, 0, Q }(X \geq t_{\epsilon}) \left( 1- P_{ \epsilon, 0, Q }(X \geq t_{\epsilon}) \right) + ( \rho_{\epsilon} - n P_{ \epsilon, 0, Q }(X \geq t_{\epsilon}) )/3 }  \right),
		\end{split}
	\end{equation*} where the last inequality is by Bernstein inequality and the condition $\rho_{\epsilon} \geq n P_{ \epsilon, 0, Q }(X \geq t_{\epsilon})$ as we will see later.
	
	Let us denote $A = \rho_{\epsilon} - n P_{ \epsilon, 0, Q }(X \geq t_{\epsilon})$ and $B = n P_{ \epsilon, 0, Q }(X \geq t_{\epsilon}) \left( 1- P_{ \epsilon, 0, Q }(X \geq t_{\epsilon}) \right)$. Then, to achieve Type-2 error control, we need the condition
	\begin{equation} \label{eq:type-II-error}
		\begin{split}
			& \exp\left( - \frac{ \frac{1}{2} A^2 }{ B + A/3 } \right) \leq \alpha \\
			\Longleftrightarrow & A \geq \frac{1}{3} \log(1/\alpha) + \sqrt{ \frac{1}{9} \log^2(1/\alpha)  + 2 B \log(1/\alpha) } \\
			\Longleftrightarrow & \rho_{\epsilon} \geq n P_{ \epsilon, 0, Q }(X \geq t_{\epsilon}) +  \frac{1}{3} \log(1/\alpha) + \sqrt{ \frac{1}{9} \log^2(1/\alpha)  + 2 B \log(1/\alpha) }\\
			\Longleftarrow & \rho_{\epsilon} \geq n P_{ \epsilon, 0, Q }(X \geq t_{\epsilon}) +  \frac{2}{3} \log(1/\alpha) + \sqrt{ 2  n P_{ \epsilon, 0, Q }(X \geq t_{\epsilon}) \log(1/\alpha) }.
		\end{split}
	\end{equation} Notice that we do have $\rho_{\epsilon} \geq n P_{ \epsilon, 0, Q }(X \geq t_{\epsilon})$ being satisfied based on \eqref{eq:type-II-error}. 
	
	Next, we move on to the simultaneous Type-1 error control. We will show an equivalent statement is true:
	\begin{equation*}
		\inf_{Q}P_{ \epsilon_{\max}, \theta, Q }\left( \phi_{\theta, \theta- r_{\epsilon}, \epsilon} = 0 \text{ for all } \epsilon \in [0, \epsilon_{\max}] \right) \geq 1- \alpha.
		\end{equation*} By the DKW inequality with the variance bound (see Lemma \ref{lm:ratio-DKW}), we have the following event happening with probability at least $1-\alpha$, 
		\begin{equation*}
		\begin{split}
			(A_1) = \Big\{ & \textnormal{for all } x \in \bbR \textnormal{ such that } F_{\epsilon_{\max}, \theta, Q}(x)(1-F_{\epsilon_{\max}, \theta, Q}(x)) \geq 100 \log(8/\alpha)/n, \\
			& \quad |F_n(x) - F_{\epsilon_{\max}, \theta, Q}(x)| \leq  F_{\epsilon_{\max}, \theta, Q}(x)(1 - F_{\epsilon_{\max}, \theta, Q}(x))/2  \Big\},
		\end{split}
	\end{equation*} $F_{\epsilon_{\max}, \theta, Q }(\cdot)$ denotes the CDF of $P_{\epsilon_{\max}, \theta, Q}$.

	Fix any $\theta,Q$ and any $\epsilon \in [0, \epsilon_{\max}]$, under the null $P_{ \epsilon_{\max}, \theta, Q }$, we have
	 \begin{equation*}
	 	\begin{split}
	 		& \phi_{\theta, \theta - r_{\epsilon}, \epsilon} = 0 \\
	 		&\Longleftrightarrow \sum_{j=1}^n \indi\left( X_j - (\theta - r_{\epsilon}) \geq t_{\epsilon} \right) > \rho_{\epsilon} \\
	 		& \Longleftrightarrow \sum_{j=1}^n \indi\left( X_j < \theta - r_{\epsilon} + t_{\epsilon} \right) < n - \rho_{\epsilon} \\
	 		& \Longleftarrow F_n(\theta + t_{\epsilon} - r_{\epsilon}) < 1 - \rho_{\epsilon}/n \\
	 		&  \Longleftrightarrow \rho_{\epsilon}/n < 1 - F_n(\theta + t_{\epsilon} - r_{\epsilon})\\
	 		& \overset{(A_1)}\Longleftarrow \left\{\begin{array}{l}
	 			 F_{\epsilon_{\max}, \theta, Q}(\theta + t_{\epsilon} - r_{\epsilon}) (1 -  F_{\epsilon_{\max}, \theta, Q}(\theta + t_{\epsilon} - r_{\epsilon})) \geq 100 \log(8/\alpha)/n\\
	 			 \rho_{\epsilon}/n < \frac{1}{2} \left(1 - F_{\epsilon_{\max}, \theta, Q}(\theta + t_{\epsilon} - r_{\epsilon}) \right)
	 		\end{array} \right..
	 	\end{split}
	 \end{equation*}

	In summary, in order to have the simultaneous Type-1 and pointwise Type-2 error control, we need the following two conditions to hold for all $\epsilon \in [0, \epsilon_{\max}]$,
	\begin{equation} \label{eq:two-conditions}
		\begin{split}
			&\left\{ \begin{array}{l}
				\rho_{\epsilon} \geq n P_{ \epsilon, 0, Q }(X \geq t_{\epsilon}) +  \frac{2}{3} \log(1/\alpha) + \sqrt{ 2  n P_{ \epsilon, 0, Q }(X \geq t_{\epsilon}) \log(1/\alpha) },  \\
				F_{\epsilon_{\max}, \theta, Q}(\theta + t_{\epsilon} - r_{\epsilon}) (1 -  F_{\epsilon_{\max}, \theta, Q}(\theta + t_{\epsilon} - r_{\epsilon})) \geq 100 \log(8/\alpha)/n\\
	 			 \rho_{\epsilon}/n < \frac{1}{2} \left(1 - F_{\epsilon_{\max}, \theta, Q}(\theta + t_{\epsilon} - r_{\epsilon}) \right)
			\end{array} \right.\\
			\overset{(a)}\Longleftarrow  &\left\{ \begin{array}{l}
				\rho_{\epsilon} \geq n P_{ \epsilon, 0, Q }(X \geq t_{\epsilon}) +  \frac{2}{3} \log(1/\alpha) + \eta n P_{ \epsilon, 0, Q }(X \geq t_{\epsilon}) + \frac{1}{2\eta} \log(1/\alpha), \\
				F_{\epsilon_{\max}, \theta, Q}(\theta + t_{\epsilon} - r_{\epsilon}) (1 -  F_{\epsilon_{\max}, \theta, Q}(\theta + t_{\epsilon} - r_{\epsilon})) \geq 100 \log(8/\alpha)/n\\
	 			P_{ \epsilon_{\max}, \theta , Q }(X > \theta + t_{\epsilon} -r_{\epsilon} ) > 2\rho_{\epsilon}/n
			\end{array} \right.
		\end{split}
	\end{equation} where (a) holds for any $0<\eta < 1$. 
	
	 Since $Q$ could be arbitrary, we have $$P_{ \epsilon, 0, Q }(X \geq t) \leq \epsilon + P_{0}(X \geq t )  \text{ and } (1 - \epsilon_{\max})  P_{r_{\epsilon} }(X \geq t) \leq P_{ \epsilon_{\max}, \theta , Q }(X > \theta + t -r_{\epsilon} ).$$ Thus, the last three conditions in \eqref{eq:two-conditions} are implied by the following three conditions:
	\begin{equation*} %\label{ineq:final-two-conditions}
		\begin{split}
				 &\rho_{\epsilon} \geq (1 + \eta) n \left(  \epsilon + P_{0}(X \geq t_{\epsilon} )\right) +  \left(2/3 + 1/(2\eta)\right) \log(1/\alpha),\\
				 & (1- \epsilon_{\max} )^2 P_{r_{\epsilon}}(X \geq t_{\epsilon}) (1 - P_{r_{\epsilon}}(X \geq t_{\epsilon})) \geq 100 \log(8/\alpha)/n,  \\
				 &(1- \epsilon_{\max} )  n P_{r_{\epsilon}}(X \geq t_{\epsilon}) \geq 2\rho_{\epsilon} .
		\end{split}
	\end{equation*}	
	These are the three conditions in \eqref{ineq:testing-upper-bound-general-cond1}. We finish the proof.

\subsection{Proof of Corollary \ref{coro:length-upper-bound}}
Following the exact same calculation as in \eqref{eq:interval-step1}-\eqref{eq:interval-step2}, we have that the ARCI in \eqref{eq:adaptive-CI-general} is the $\widehat{\CI}$ in \eqref{def:CI-based-on-test} using the tests in \eqref{eq:test-general} and \eqref{eq:test+general}. Thus, the coverage and length guarantee of $\widehat{\CI}$ in \eqref{eq:adaptive-CI-general} is given by a combination of Theorem \ref{thm:upper-bound-general} and Proposition \ref{prop:test-to-CI}.

%%%%%%%%%%%%%%%%%%%%%%%%%%%%%%%%%%%%%%
\subsection{Proof of Theorem \ref{th:test-lower-bound}}
%%%%%%%%%%%%%%%%%%%%%%%%%%%%%%

First, it is easy to observe the minimax separation rate for $\cH(\theta, \theta- r, \epsilon)$ and $\cH(\theta, \theta + r, \epsilon)$ are the same by symmetry, so we provide the proof for $\cH(\theta, \theta- r, \epsilon)$ only. In addition, since $\theta$ is known, we can shift the locations of both the null and the alternative distributions, and it would not affect the hardness of the testing problem. Thus, for simplicity, we would let $\theta = r$ and the hardness of the testing problem $\cH(\theta, \theta- r, \epsilon)$ is the same as the following one: given $\{X_i \}_{i=1}^n \overset{i.i.d.}\sim P$, 
\begin{equation*}
	H_0: P \in \{ (1 - \epsilon_{\max} )P_{r } + \epsilon_{\max} Q: Q\} \quad \text{v.s.} \quad  H_1: P \in \{ (1 - \epsilon )P_{0} + \epsilon Q: Q \}.
\end{equation*}
%We denote $f_\theta(\cdot)$ and $F_\theta(\cdot)$ as the density and CDF functions of $P_{\theta}$, respectively. 
We divide the rest of the proof into two cases.

{\bf Case I ($\frac{\alpha}{n} \geq \frac{\epsilon}{1 - \epsilon_{\max}}$)}. In this case, we will construct distributions $Q_0$ and $Q_1$ such that $\TV(P_{\epsilon_{\max}, r, Q_0}, P_{\epsilon, 0, Q_1}) \leq \frac{\alpha(1 - \epsilon_{\max})}{n}$. Let $Q_1 = P_{0}$. Given some $t \geq 0 $, which is to be specified later, let $Q_0$ be the distribution with the following density $q_0(x)$ with respect to the Lebesgue measure:
\begin{equation}\label{def:Q}
		q_0(x) = \frac{f(x) - (1-\epsilon_{\max} ) f(x-r) \bar{c} \indi( x - r \leq t ) }{ \epsilon_{\max} }, \, \forall x \in \bbR.
\end{equation} Here we set $\bar{c} := \frac{1}{P_{r}( X - r \leq t )  } = \frac{1}{P_{0}( X \leq t )  }$, where the second equality is because $P_\theta$ is a location family. By this choice of $\bar{c}$, it is easy to check 
\begin{equation*}
	\begin{split}
		\int q_0(x) dx = \frac{1 -  (1 - \epsilon_{\max} ) }{\epsilon_{\max} } = 1. 
	\end{split}
\end{equation*} Thus to guarantee $Q_0$ is a valid distribution, we just need the condition 
\begin{equation} \label{eq:valid-dist-condition}
\begin{split}
	\inf_{x \in \bbR } q_0(x) \geq 0 \quad  &\Longleftrightarrow \quad \inf_{x \in \bbR} \left\{  f(x) - (1-\epsilon_{\max} ) f(x-r) \bar{c} \indi( x - r \leq t )  \right\} \geq 0 \\
	& \Longleftrightarrow \inf_{x \leq r + t } \left\{  f(x) - (1-\epsilon_{\max} ) f(x-r) \bar{c}  \right\} \geq 0 \\
	& \Longleftrightarrow \sup_{ x \leq  r + t  } \left\{ f(x-r) - \frac{f(x)}{(1 - \epsilon_{\max} ) \bar{c} } \right\} \leq 0.
\end{split}
\end{equation}

Now, given constructed $Q_0$ and $Q_1$, let us bound $\TV(P_{\epsilon_{\max}, r, Q_0}, P_{\epsilon, 0, Q_1})$. 
\begin{equation*}
	\begin{split}
		\TV(P_{\epsilon_{\max}, r, Q_0}, P_{\epsilon, 0, Q_1}) &= \frac{1}{2} \int \left| (1 - \epsilon_{\max} ) f(x-r) + \epsilon_{\max} q_0(x) - f(x) \right| dx \\
		 & \overset{ \eqref{def:Q} }=  \frac{1}{2} \int \left| (1 - \epsilon_{\max} ) f(x-r) -(1-\epsilon_{\max} ) f(x-r) \bar{c} \indi( x - r \leq t ) \right| dx \\
		 & = \frac{1}{2} \int (1 - \epsilon_{\max} ) f(x-r) \left| 1 -\bar{c} \indi( x - r \leq t ) \right| dx \\
		 & =  \frac{1}{2} \int_{x \leq r + t} (1 - \epsilon_{\max} ) f(x-r) (\bar{c}-1) dx + \frac{1}{2} \int_{x \geq r + t} (1 - \epsilon_{\max} ) f(x-r) dx \\
		 & = \frac{1 - \epsilon_{\max} }{2} \left( (\bar{c} - 1) P_{0}( X \leq t ) +  P_{0}( t \leq X )  \right) \\
		 & \overset{(a)}= \frac{1 - \epsilon_{\max} }{2} \left( 1 - P_0(X \leq t) +  P_{0}( t \leq X )  \right) \\
		 & = (1 - \epsilon_{\max}) P_{0}( X \geq t ).
	\end{split}
\end{equation*} where in (a) we plug in the definition of $\bar{c}$. Now we take $t = F^{-1}( 1- \frac{\alpha}{n})$, then we have
\begin{equation} \label{eq:TV_bound}
	\begin{split}
		\TV(P_{\epsilon_{\max}, r, Q_0}, P_{\epsilon, 0, Q_1}) \leq \alpha(1 - \epsilon_{\max})/n
	\end{split}
\end{equation} given the condition in \eqref{eq:valid-dist-condition} is satisfied. When $t = F^{-1}( 1- \frac{\alpha}{n})$, the condition in \eqref{eq:valid-dist-condition} becomes
\begin{equation*}
	\begin{split}
	 \sup_{ x \leq  r + F^{-1}( 1- \frac{\alpha}{n}) } \left\{ f(x-r) - \frac{1 - \frac{\alpha}{n} }{1 - \epsilon_{\max} } f(x)  \right\} \leq 0.
	\end{split}
\end{equation*} In summary, if 
\begin{equation} \label{eq:lowerbound-cond1}
	\begin{split}
		r & \in \left\{r \geq 0:  \sup_{ x \leq  r + F^{-1}( 1- (\frac{\alpha}{2n} +  \frac{\epsilon}{2(1- \epsilon_{\max})} ) ) } \left\{ f(x-r) - \frac{1 - (\frac{\alpha}{n} \vee \epsilon) }{1 - \epsilon_{\max} } f(x)  \right\} \leq 0  \right\}\\
		& \subseteq   \left\{r \geq 0:  \sup_{ x \leq  r + F^{-1}( 1- \frac{\alpha}{n}) } \left\{ f(x-r) - \frac{1 - (\frac{\alpha}{n} \vee \epsilon) }{1 - \epsilon_{\max} } f(x)  \right\} \leq 0 \right\}\\
		& \overset{(a)}=\left\{r \geq 0:  \sup_{ x \leq  r + F^{-1}( 1- \frac{\alpha}{n}) } \left\{ f(x-r) - \frac{1 - \frac{\alpha}{n} }{1 - \epsilon_{\max} } f(x)  \right\} \leq 0 \right\},
	\end{split}
\end{equation} where (a) is because $\frac{\alpha}{n} \geq \frac{\epsilon}{1 - \epsilon_{\max}}$, then we have the TV bound in \eqref{eq:TV_bound} holds. Notice that the set on the right hand side of \eqref{eq:lowerbound-cond1} is nonempty since $r = 0$ always satisfies the constraint because $\frac{1- (\epsilon \vee \frac{\alpha}{n}) }{1- \epsilon_{\max} } \geq 1$ by the conditions on $\epsilon$ and $n$. This finishes the proof of Case I.

{\bf Case II ($\frac{\alpha}{n} < \frac{\epsilon}{1 - \epsilon_{\max}}$).} In this case, we need a different construction of $Q_0$ and $Q_1$. Specifically, we would like to design $Q_0$ and $Q_1$ such that
\begin{equation} \label{eq:equi-distribution}
	(1 - \epsilon_{\max} )P_{r } + \epsilon_{\max} Q_0 \overset{d}= (1 - \epsilon )P_{0} + \epsilon Q_1.
\end{equation} Denote the densities of $Q_0$ and $Q_1$ as $q_0(x)$ and $q_1(x)$, respectively, then \eqref{eq:equi-distribution} is equivalent to
\begin{equation} \label{eq:equivalent-condition}
	\begin{split}
		(1 - \epsilon_{\max} )f(x-r) + \epsilon_{\max} q_0(x) = (1 - \epsilon )f(x) + \epsilon q_1(x),\, \forall x \in \bbR.
	\end{split}
\end{equation}

Let $S = \{x: (1 - \epsilon_{\max} ) f(x-r) > (1 - \epsilon )f(x) \}$. We first consider the case $S$ has measure zero under $P_r$. Then we can let $q_1(x)$ to be any density function supported on $\bbR \setminus S$, and then set 
\begin{equation*} % \label{def:q0}
	\begin{split}
		q_0(x) =  \left\{\begin{array}{ll}
			\frac{(1 - \epsilon )f(x) + \epsilon q_1(x) - (1 - \epsilon_{\max} )f(x-r)}{ \int_{ \bbR \setminus S } ((1 - \epsilon )f(x) + \epsilon q_1(x) - (1 - \epsilon_{\max} )f(x-r)) dx }, & \, \forall x \in \bbR \setminus S, \\
			0,  & x \in S. 
		\end{array} \right.
	\end{split}
\end{equation*} Then $q_0(x)$ is a valid probability density function (PDF) by the definition of $S$. Thus, if 
\begin{equation} 
	r \in \mathcal{R}_1:=  \left\{ r \geq 0: (1 - \epsilon_{\max} ) f(x-r) \leq (1 - \epsilon )f(x) \text{ for } P_{r} \text{ a.s.} \right\},
\end{equation} then there exists $Q_0$ and $Q_1$, such that $\TV(P_{\epsilon_{\max}, r, Q_0}, P_{\epsilon, 0, Q_1}) = 0$

Now, let us consider the setting $S$ has a positive measure under $P_r$. Then we can construct $q_1(x)$ as follows:
\begin{equation} \label{def:q1}
	\begin{split}
		q_1(x) = \left\{ \begin{array}{ll}
			c^*\frac{1 - \epsilon_{\max} }{\epsilon} f(x-r), & \text{ if } x \in S, \\
			0 , & \text{ otherwise},
		\end{array}  \right.
	\end{split}
\end{equation} where $c^* = \frac{1}{\int_{S} \frac{1 - \epsilon_{\max} }{\epsilon} f(x-r) dx}$. It is clear $\int q_1(x) dx = 1$ and is a valid PDF. Given $q_1(x)$, $q_0(x)$ is constructed so that \eqref{eq:equivalent-condition} holds, i.e.,
\begin{equation*}
	\begin{split}
		q_0(x) = \frac{(1 - \epsilon )f(x) + \epsilon q_1(x) - (1 - \epsilon_{\max} )f(x-r)}{\epsilon_{\max}},\, \forall x \in \bbR.
	\end{split}
\end{equation*} It is easy to check $\int q_0(x)  dx = 1$, so to ensure $q_0(x)$ is a valid PDF, we just need to show $q_0(x) \geq 0$ for all $x \in \bbR$. When $x \in S$, 
\begin{equation*}
	(1 - \epsilon )f(x) + \epsilon q_1(x) - (1 - \epsilon_{\max} )f(x-r) = (1 - \epsilon )f(x) + (c^* - 1)(1 - \epsilon_{\max} ) f(x-r),
\end{equation*} which is greater or equal to zero when $c^* \geq 1$. When $x \notin S$, 
\begin{equation*}
	(1 - \epsilon )f(x) + \epsilon q_1(x) - (1 - \epsilon_{\max} )f(x-r) \geq (1 - \epsilon )f(x) - (1 - \epsilon_{\max} )f(x-r) \geq 0,
\end{equation*} by the definition of $S$. Thus to guarantee $q_0(x)$ is a valid PDF, a sufficient condition is $c^* \geq 1$, i.e.,
\begin{equation} \label{eq:q1-valide-distribution-condition}
	\int_{S} \frac{1 - \epsilon_{\max} }{\epsilon} f(x-r) dx \leq 1.
\end{equation} 

This proves if
\begin{equation} \label{eq:lowerbound-cond2}
	r \in \mathcal{R}_2:= \left\{ r \geq 0: \int_{ \{x: (1 - \epsilon_{\max} ) f(x-r) > (1 - \epsilon )f(x) \} } \frac{1 - \epsilon_{\max} }{\epsilon} f(x-r) dx \leq 1  \right\},
\end{equation} then there exists $Q_0$ and $Q_1$, such that $\TV(P_{\epsilon_{\max}, r, Q_0}, P_{\epsilon, 0, Q_1}) = 0$. At the same time, it is easy to observe that $\mathcal{R}_1 \subseteq \mathcal{R}_2$. So combining these two, we have proved that if $r \in \mathcal{R}_2$, then there exists $Q_0$ and $Q_1$, such that $\TV(P_{\epsilon_{\max}, r, Q_0}, P_{\epsilon, 0, Q_1}) = 0$.

Next, let us rewrite \eqref{eq:lowerbound-cond2} so that it can be combined with \eqref{eq:lowerbound-cond1}. 
\begin{equation} \label{eq:lowerbound-cond2-rewrite}
	\begin{split}
		& \left\{ r\geq 0 : \int_{ \{x: (1 - \epsilon_{\max} ) f(x-r) > (1 - \epsilon )f(x) \}} \frac{1 - \epsilon_{\max} }{\epsilon} f(x-r) dx \leq 1  \right\} \\
		 & = \left\{ r \geq 0: \int_{ \{x: (1 - \epsilon_{\max}  ) f(x-r) > (1 - \epsilon )f(x)\} } f(x-r) dx \leq \frac{\epsilon}{1- \epsilon_{\max} }  \right\}\\
		 & \overset{(a)}\supseteq \left\{ r \geq 0: \sup_{ x \leq  r + F^{-1}( 1- \frac{\epsilon}{1- \epsilon_{\max} } ) } \left\{ f(x-r) - \frac{1 - \epsilon }{1 - \epsilon_{\max} } f(x)  \right\} \leq 0  \right\},
	\end{split}
\end{equation} where (a) is because if $r$ belongs to the latter set, then 
\begin{equation*}
	\{x: (1 - \epsilon_{\max}  ) f(x-r) > (1 - \epsilon )f(x)\} \subseteq \left\{x: x \geq r+ F^{-1} \left( 1- \frac{\epsilon}{1- \epsilon_{\max} }  \right)\right\},
\end{equation*} so 
\begin{equation*}
	\int_{\{x: (1 - \epsilon_{\max}  ) f(x-r) > (1 - \epsilon )f(x)\} } f(x-r) dx \leq \int_{ \{x:  x \geq r+F^{-1  }( 1- \frac{\epsilon}{1- \epsilon_{\max} }  )\}} f(x-r) dx = \frac{\epsilon}{1- \epsilon_{\max} },
\end{equation*} which implies $r$ also belongs to the previous set.

 Therefore, a conbination of \eqref{eq:lowerbound-cond2} and \eqref{eq:lowerbound-cond2-rewrite} implies that if 
\begin{equation} \label{eq:lowerbound-cond2new}
\begin{split}
	r & \in  \left\{ r \geq 0: \sup_{x \leq r + F^{-1  } \left( 1- ( \frac{\alpha}{2n} + \frac{\epsilon}{2(1- \epsilon_{\max}) })  \right) } \left\{ f(x-r) - \frac{1 - (\epsilon \vee \frac{\alpha}{n}) }{1 - \epsilon_{\max} } f(x)  \right\} \leq 0   \right\} \\
	 &\subseteq  \left\{ r \geq 0: \sup_{ x \leq  r + F^{-1}( 1- \frac{\epsilon}{1- \epsilon_{\max} } ) } \left\{ f(x-r) - \frac{1 - (\epsilon \vee \frac{\alpha}{n}) }{1 - \epsilon_{\max} } f(x)  \right\} \leq 0  \right\} \\ 
	& \subseteq \left\{ r \geq 0: \sup_{ x \leq  r + F^{-1}( 1- \frac{\epsilon}{1- \epsilon_{\max} } ) } \left\{ f(x-r) - \frac{1 - \epsilon }{1 - \epsilon_{\max} } f(x)  \right\} \leq 0  \right\},
\end{split}
\end{equation} then there exists $Q_0$ and $Q_1$, such that $\TV(P_{\epsilon_{\max}, r, Q_0}, P_{\epsilon, 0, Q_1}) = 0$. Notice that the set in \eqref{eq:lowerbound-cond2new} is nonempty since $r = 0$ belongs to the set as $\frac{1- (\epsilon \vee \frac{\alpha}{n}) }{1- \epsilon_{\max} } \geq 1$.

Combining Case I and Case II, we have if 
\begin{equation} \label{eq:lowerbound-cond-final}
	r = \underline{r}(\epsilon) :=  \sup \left\{ r \geq 0: \sup_{x \leq r + F^{-1  } \left( 1- ( \frac{\alpha}{2n} + \frac{\epsilon}{2(1- \epsilon_{\max}) })  \right) } \left\{ f(x-r) - \frac{1 - (\epsilon \vee \frac{\alpha}{n}) }{1 - \epsilon_{\max} } f(x)  \right\} \leq 0   \right\},
\end{equation} then there exists $Q_0$ and $Q_1$, such that $\TV(P_{\epsilon_{\max}, r, Q_0}, P_{\epsilon, 0, Q_1}) \leq \frac{\alpha (1 - \epsilon_{\max} ) }{n} \leq \frac{\alpha }{n} $. Notice that the supremum in \eqref{eq:lowerbound-cond-final} can be achieved as $f(x-r) $ and $f(x)$ are continuous. Finally, by the property of TV distance on a product measure, 
\begin{equation*}
	\TV(P^{\otimes n}_{\epsilon_{\max}, r, Q_0}, P^{\otimes n}_{\epsilon, 0, Q_1}) \leq n \TV(P_{\epsilon_{\max}, r, Q_0}, P_{\epsilon, 0, Q_1}) \leq \alpha.
\end{equation*}

%%%%%%%%%%%%%%%%%%%%%%%%%%%%%
\subsection{Proof of Corollary \ref{coro:length-lower-bound}}
%%%%%%%%%%%%%%%%%%%%%%%%%%%%
The proof follows from the same contradiction argument as in the proof of Theorem \ref{th:lower-bound-Gaussian}. For simplicity, we omit it here.

%%%%%%%%%%%%%%%%%%%%%%%%%%%%%%%%%%%%%%
\subsection{Proof of Lemma \ref{lem:r-arrow} }
%%%%%%%%%%%%%%%%%%%%%%%%%%%%%%%%%%%%%%
  First, Corollaries \ref{coro:length-upper-bound} and \ref{coro:length-lower-bound} jointly applies that $\underline{r}(\epsilon)\leq r_{\alpha}(\epsilon,[0,\epsilon_{\max}])\leq 2\overline{r}(\epsilon)$. Next, we show $r^{\downarrow}(\epsilon)\leq \underline{r}(\epsilon)$. First,
		\begin{equation} \label{ineq:underline-r-lower-bound}
	\begin{split}
		 &r^{\downarrow}(\epsilon) = \sup \Big\{ r \geq 0:  \frac{f(x-r)}{f(x)} \leq \frac{1 - \epsilon_{\max}/2 }{1 - \epsilon_{\max} }  \textnormal{ for } x= r + F^{-1  }   \left( 1- \underline{q}(\epsilon)  \right) \Big\} \\
		&\overset{ (a) }= \sup \Big\{ r \geq 0:  \sup_{ x \leq r + F^{-1  } \left( 1- \underline{q}(\epsilon)  \right) } \left\{ f(x-r) - \frac{1 - \epsilon_{\max}/2 }{1 - \epsilon_{\max} } f(x)  \right\} \leq 0  \Big\} \\
		&\overset{(b)}\leq \sup \Big\{ r \geq 0:  \sup_{ x \leq r + F^{-1  } \left( 1- \underline{q}(\epsilon)  \right) } \left\{ f(x-r) - \frac{1 - (\epsilon \vee \frac{\alpha}{n}) }{1 - \epsilon_{\max} } f(x)  \right\} \leq 0   \Big\} = \underline{r}(\epsilon)
	\end{split}
	\end{equation} (a) is because $\sup_{ x \leq r + F^{-1  }   \left( 1- \underline{q}(\epsilon)  \right)} \frac{f(x-r)}{f(x)} = \frac{f(F^{-1  }   \left( 1- \underline{q}(\epsilon)  \right))}{f(r + F^{-1  }   \left( 1- \underline{q}(\epsilon)  \right))}$ as $\sup_{ x \leq r} \frac{f(x-r)}{f(x)} = \frac{f(0)}{f(r)}$ by the unimodality of $f(x)$ and $\frac{f(x-r)}{f(x)}$ is a nondecreasing function of $x$ when $x \geq r$; (b) holds since $\frac{1- \epsilon_{\max}/2 }{1- \epsilon_{\max} } \leq \frac{1- (\epsilon \vee \frac{\alpha}{n}) }{1- \epsilon_{\max} }$ due to the constraints on $\epsilon,n$ and $\alpha$.

	Next, we show the upper bound for $\bar{r}(\epsilon)$ under unimodality and the weaker condition $\frac{f(x-r)}{f(x)}$ is a nondecreasing function of $x$ when $x \geq r + F^{-1} ( 1 - \bar{q}(\epsilon) )$. For any $r \geq 0$ such that $\frac{f(x-r)}{f(x)} \geq \frac{6}{1 - \epsilon_{\max}}$ for $x = r + F^{-1} ( 1 - \bar{q}(\epsilon) )$, we have $\frac{f(t-r)}{f(t)} \geq \frac{6}{1 - \epsilon_{\max}}$ for any $t \geq x$ due to the nondecreasing property of $\frac{f(x-r)}{f(x)}$. As a result of that, we have
\begin{equation*} 
	\begin{split}
		P_{r}(X \geq x) = \int_{x}^\infty f(t - r) dt \geq \int_x^\infty  \frac{6}{1 - \epsilon_{\max}} f(t) dt = \frac{6}{1 - \epsilon_{\max}} P_{0}(X \geq x).
	\end{split}
\end{equation*}  Thus, 
\begin{equation}\label{ineq:density-ratio-to-tail-ratio}
	\begin{split}
		& \Big\{r \geq 0 :  \frac{P_{r}(X \geq x)}{P_{0}(X \geq x )} \geq \frac{6}{1 - \epsilon_{\max}} \textnormal{ for } x = r + F^{-1} ( 1 - \bar{q}(\epsilon) )  \Big\} \\
		 \supseteq & \left\{ r\geq 0: \frac{f(x-r)}{f(x)} \geq \frac{6}{1 - \epsilon_{\max}} \textnormal{ for }  x = r + F^{-1} ( 1 - \bar{q}(\epsilon) ) \right\},
	\end{split}
\end{equation} As a result, 
	\begin{equation*}
		\begin{split}
			\bar{r}( \epsilon)=&\inf  \Big\{r \geq 0 : \sup_{x \leq  r + F^{-1  } \left( 1- \bar{q}(\epsilon)  \right) } \left\{ P_{r}(X \geq x) - \frac{6}{1 - \epsilon_{\max}}P_{0}(X \geq x ) \right\}  \geq 0\Big\}\\
			\leq &  \inf \Big\{r \geq 0 :  \frac{P_{r}(X \geq x)}{P_{0}(X \geq x )} \geq \frac{6}{1 - \epsilon_{\max}} \textnormal{ for } x = r + F^{-1} ( 1 - \bar{q}(\epsilon) )  \Big\} \\
			\overset{ \eqref{ineq:density-ratio-to-tail-ratio} }\leq &  \inf \left\{ r\geq 0: \frac{f(x-r)}{f(x)} \geq \frac{6}{1 - \epsilon_{\max}} \textnormal{ for }  x = r + F^{-1} ( 1 - \bar{q}(\epsilon) ) \right\} = r^{\uparrow}(\epsilon).
		\end{split}
	\end{equation*} This finishes the proof of this lemma.

\subsection{Proof of Theorem \ref{thm:example}}
We will prove the result separately for each example in the following subsections. In each example, we use $h_r(x)$ to denote $ \frac{f(x-r)}{f(x)}$. Recall $\underline{q}(\epsilon) = \frac{\epsilon}{2(1- \epsilon_{\max}) } +  \frac{\alpha}{2n}$ and $\bar{q}(\epsilon) = \frac{6}{(1-\epsilon_{\max})} \left(\epsilon + \frac{100 \log(32/\alpha)}{3(1-\epsilon_{\max})n} \right)$, we set $\underline{t} = F^{-1} \left( 1- \underline{q}(\epsilon) \right)$ and $\bar{t} = F^{-1} \left( 1- \bar{q}(\epsilon) \right)$ throughout the proof.

%%%%%%%%%%%%%%%%%%%%%%%%%%%%%%
\subsubsection{$t$ distribution} \label{proof:t-distribution}
%%%%%%%%%%%%%%%%%%%%%%%%%%%%%
 First, the upper bound is given by the interval centering at the sample median with constant length (see Proposition \ref{prop:minimax-rate-others-upper-bound}). Next, we show the length lower bound. Since $r_{\alpha}(\epsilon, [0,\epsilon_{\max} ])$ is an nondecreasing function with respect to $\epsilon$, it is enough to show $r_{\alpha}(0, [0,\epsilon_{\max} ]) \gtrsim 1$. The main tool we will use is Corollary \ref{coro:length-lower-bound}. It is clear that $f(x)$ is unimodal. Now, let us take a look at the monotonicity of $h_r(x) = \left( \frac{\nu + x^2}{\nu + (x-r)^2}  \right)^{ \frac{\nu + 1}{2}}$ when $x \geq r$ for any $r \geq 0$. Let us denote $s(x) = \frac{\nu + x^2}{\nu + (x-r)^2}$, then $s'(x) = \frac{2r( \nu -x(x-r) )}{(\nu + (x-r)^2)}$. Thus $s(x)$ first increases when $x \in [r, (r + \sqrt{r^2 + 4\nu})/2]$ and decreases when $x \geq (r + \sqrt{r^2 + 4\nu})/2$. As a result, $h_r(x)$ achieves its maximum at $x = (r + \sqrt{r^2 + 4\nu})/2$. Then,
\begin{equation} \label{ineq:t-r-lower-bound}
	\begin{split}
		r_{\alpha}(0, [0,\epsilon_{\max} ]) \overset{ \textnormal{Corollary }\ref{coro:length-lower-bound} }\geq \underline{r}(0) & := \sup \left\{r \geq 0: \sup_{  x \leq r + F^{-1}( 1- \frac{\alpha}{2n} )  } \frac{f(x - r)}{f(x)} \leq \frac{1- \alpha/n }{1- \epsilon_{\max} } \right\} \\
		& \overset{(a)}\geq \sup \left\{r \geq 0: \sup_{ x \leq r + F^{-1}( 1- \frac{\alpha}{2n} )  } \frac{f(x-r)}{f(x)} \leq \frac{1- \epsilon_{\max}/2 }{1- \epsilon_{\max} } \right\} \\
		& \geq  \sup \left\{r \geq 0: h_r((r + \sqrt{r^2 + 4\nu})/2) \leq \frac{1- \epsilon_{\max}/2 }{1- \epsilon_{\max} } \right\},
	\end{split}
\end{equation} where $(a)$ is because we can set $n$ to be sufficiently large so that $\alpha/n \leq \epsilon_{\max}/2 $.  Finally, let
\begin{equation*}
	\begin{split}
		k(r) :=  h_r((r + \sqrt{r^2 + 4\nu})/2) = \left( \frac{\nu + \left( \frac{ \sqrt{r^2 + 4 \nu} + r }{2} \right)^2 }{\nu + \left( \frac{ \sqrt{r^2 + 4 \nu} - r }{2} \right)^2} \right)^{(\nu+1)/2}.
	\end{split}
\end{equation*} Notice that $k(0) = 1$ and $k(r) \to \infty$ as $r \to \infty$. Thus $\sup \left\{r \geq 0: h_r((r + \sqrt{r^2 + 4\nu})/2) \leq \frac{1- \epsilon_{\max}/2 }{1- \epsilon_{\max} } \right\}$ on the right-hand side of \eqref{ineq:t-r-lower-bound} is some universal positive constant independent of $n$. This finishes the proof.

%%%%%%%%%%%%%%%%%%%%%%%%%%%%%%
\subsubsection{Generalized Gaussian distribution}
%%%%%%%%%%%%%%%%%%%%%%%%%%%%%
 The proof strategy is similar to the t-distribution case. Let us first take a look at the monotonicity of $h_r(x) = \frac{f(x-r)}{f(x)} = \exp(x^\beta - (x - r)^\beta )$ when $x \geq r$. Let $s(x) = x^\beta - (x - r)^\beta$, then $s'(x) = \beta( \frac{1}{x^{1-\beta}} - \frac{1}{(x-r)^{1-\beta}} )$. So $h_r(x)$ is a nonincreasing function of $x$ when $x \geq r$ for any $r \geq 0$ when $\beta \in (0,1]$, and $h_r(x)$ is a nondecreasing function of $x$ when $x \geq r$ for any $r \geq 0$ when $\beta > 1$. We divide the proof into two different cases: (1) $\beta \in (0,1]$ and (2) $\beta > 1$.

{\noindent \bf Case 1: $\beta \in (0,1]$}.  First, the upper bound is given by the interval centering at the sample median with constant length (see Proposition \ref{prop:minimax-rate-others-upper-bound}). Now, let us provide the length lower bound. Again, it is enough to show $r_{\alpha}(0, [0,\epsilon_{\max} ]) \gtrsim 1$ It is clear that $f(x)$ is unimodal and since $h_r(x)$ is a nonincreasing function of $x$ when $x \geq r$ for any $r \geq 0$, then $\sup_{  x \leq r + F^{-1}( 1- \frac{\alpha}{2n} )  } h_r(x) = h_r(r)$ and
\begin{equation*}
		\begin{split}
			r_{\alpha}(0, [0,\epsilon_{\max} ]) \overset{ \textnormal{Corollary }\ref{coro:length-lower-bound} }\geq\underline{r}(0) & := \sup \left\{r \geq 0: \sup_{  x \leq r + F^{-1}( 1- \frac{\alpha}{2n} )  } h_r(x) \leq \frac{1- \alpha/n }{1- \epsilon_{\max} } \right\} \\
			& = \sup \left\{r \geq 0: h_r(r) \leq \frac{1- \alpha/n }{1- \epsilon_{\max} } \right\} \\
			 & \geq \sup  \left\{r \geq 0: h_r(r) \leq \frac{1- \epsilon_{\max}/2 }{1- \epsilon_{\max} } \right\} \\
			& = \left[ \log\left( \frac{1- \epsilon_{\max}/2 }{1- \epsilon_{\max} } \right) \right]^{1/\beta}.
		\end{split}
	\end{equation*}This finishes the proof for case 1.
	
\vskip.2cm	
{\noindent \bf Case 2: $\beta > 1$}.	Let us begin with the lower bound. It is enough to show that when $\epsilon \in [0, \epsilon_{\max}/2]$, 
	\begin{equation*}
		r_{\alpha}(\epsilon, [0,\epsilon_{\max}]) \gtrsim \left( \frac{1}{\log n} + \frac{1}{\log(1/ \epsilon )} \right)^{ \frac{\beta - 1}{\beta} }.
	\end{equation*} Since if we can show the above result, then when $\epsilon \in [\epsilon_{\max}/2, \epsilon_{\max}]$, 
	\begin{equation*}
		r_{\alpha}(\epsilon, [0,\epsilon_{\max}]) \geq r_{\alpha}(\epsilon_{\max}/2, [0,\epsilon_{\max}]) \asymp 1.
	\end{equation*}

	 By Corollary \ref{coro:length-lower-bound}, we only need to compute a lower bound for $\underline{r}(\epsilon)$ in the generalized Gaussian distribution with $\beta > 1$. Since $h_r(x)$ is a nondecreasing function of $x$ when $x \geq r$ for any $r \geq 0$, by Lemma \ref{lem:r-arrow},
	\begin{equation*}
		\begin{split}
			\underline{r}(\epsilon) \geq r^{\downarrow}(\epsilon)  &=  \sup \left\{r \geq 0:  h_r\left( r + \underline{t} \right) \leq \frac{1- \epsilon_{\max}/2 }{1- \epsilon_{\max} }  \right\} \\
			 &\geq  \sup \left\{r \geq 0:  h_r\left( r + \underline{t} \right) \leq \frac{1- \epsilon_{\max}/2 }{1- \epsilon_{\max} }, r/\underline{t} \leq 1/2  \right\} \\
			& \overset{(a)}\geq \sup \left\{r \geq 0: \exp( c \beta r \underline{t}^{\beta -1} ) \leq \frac{1 - \epsilon_{\max}/2 }{1- \epsilon_{\max} }, r/\underline{t} \leq 1/2  \right\} \\
			& = \sup \left\{r \geq 0: r \leq \frac{\log((1 -\epsilon_{\max}/2 )/(1 - \epsilon_{\max} ))}{c \beta \underline{t}^{\beta -1}}, r/\underline{t} \leq 1/2  \right\}\\
			& \overset{(b)}=\frac{\log((1 -\epsilon_{\max}/2 )/(1 - \epsilon_{\max} ))}{c \beta \underline{t}^{\beta -1}}  \\
			& \overset{ \textnormal{Lemma } \ref{lm:generalized-gaussian-quantile}}\geq  \frac{\log((1 -\epsilon_{\max}/2 )/(1 - \epsilon_{\max} ))}{c \beta (\log(1/\underline{q}(\epsilon)) )^{(\beta -1)/\beta}} \\
			& \gtrsim \frac{1}{(\log(1/\underline{q}(\epsilon)) )^{(\beta -1)/\beta}} \gtrsim \left( \frac{1}{\log n} + \frac{1}{\log(1/ \epsilon )} \right)^{ \frac{\beta - 1}{\beta} },
		\end{split}
	\end{equation*} where (a) is because when $r \leq \underline{t}/2$,
		\begin{equation*}
		\begin{split}
			 h_r(r + \underline{t})  = \exp( (r + \underline{t})^\beta - \underline{t}^\beta ) =  \exp( \underline{t}^\beta (1 + r/\underline{t})^\beta - \underline{t}^\beta ) \leq \exp( c \beta r \underline{t}^{\beta -1} ),
		\end{split}
	\end{equation*} where in the last inequality, we use the property $(1 + r/\underline{t})^\beta \leq 1 + c \beta r /\underline{t}$ for some constant $c$ depending on $\beta$ only given $r/\underline{t} \leq 1/2$; (b) is because when $n$ is sufficiently large and $\epsilon_{\max}$ is sufficiently small, then $\underline{t}$ is bigger than a large constant and $r/\underline{t} \leq 1/2$ can always be satisfied when we set $r = \frac{\log((1 -\epsilon_{\max}/2 )/(1 - \epsilon_{\max} ))}{c \beta \underline{t}^{\beta -1}}$ given $\beta > 1$.

	 % $x = r + F^{-1} ( 1 - \bar{q}(\epsilon) )$ and
	 Let us now move on to the calculation of the upper bound. Since $h_r(x)$ is a nondecreasing function of $x$ when $x \geq r$, then by Lemma \ref{lem:r-arrow}, it is enough to show $r^{\uparrow}(\epsilon) \lesssim \left( \frac{1}{\log n} + \frac{1}{\log(1/ \epsilon )} \right)^{ \frac{\beta - 1}{\beta} }$, where $r^{\uparrow}(\epsilon) $ is defined in \eqref{eq:r-uparraw}. If we assume $r/\bar{t} \leq 1/2$, then 
	 \begin{equation}\label{ineq:gx-bound2-generalized-gaussian}
	\begin{split}
		h_r(r+\bar{t}) = \exp( - \bar{t}^\beta + (r + \bar{t})^\beta ) = \exp( \bar{t}^\beta (1 + r/\bar{t})^\beta - \bar{t}^\beta ) \overset{(a)}\geq \exp( c' \beta r \bar{t}^{\beta - 1} ),
	\end{split}
\end{equation}where in (a), we use the property $(1 + r/\bar{t})^\beta \geq 1 + c' \beta r/\bar{t}$ for some constant $c'$ depending on $\beta$ only given $r/\bar{t} \leq 1/2$. Then by Lemma \ref{lem:r-arrow}, we have  
	 \begin{equation} \label{ineq:generalized-Gaussian-hatr-upper-bound}
	 	\begin{split}
	 			r^{\uparrow}(\epsilon) =&   \inf \Big\{r \geq 0 :  h_r(r + \bar{t}) \geq \frac{6}{1 - \epsilon_{\max}} \Big\} \\
	 			\leq & \inf \Big\{r \geq 0 :  h_r(r + \bar{t}) \geq \frac{6}{1 - \epsilon_{\max}}, r/\bar{t} \leq 1/2\Big\} \\
	 			\overset{ \eqref{ineq:gx-bound2-generalized-gaussian} }\leq & \inf \Big\{r \geq 0 :  \exp( c' \beta r \bar{t}^{\beta - 1} ) \geq \frac{6}{1 - \epsilon_{\max}}, r/\bar{t} \leq 1/2\Big\} \\
	 			= & \inf \Big\{r \geq 0 :  r \geq \frac{\log(6/(1 - \epsilon_{\max}) )}{c' \beta \bar{t}^{\beta - 1}},  r/\bar{t} \leq 1/2 \Big\} \\
	 			\overset{(a)}= & \frac{\log(6/(1 - \epsilon_{\max}) )}{c' \beta \bar{t}^{\beta - 1}} \\
	 			\overset{ \textnormal{Lemma } \ref{lm:generalized-gaussian-quantile}}\leq & \frac{\log(6/(1 - \epsilon_{\max}) )}{c' \beta \left( \log(1/\bar{q}(\epsilon))/2 \right)^{(\beta - 1)/\beta} } \lesssim \frac{1}{\left( \log(1/\bar{q}(\epsilon)) \right)^{(\beta - 1)/\beta} } \\
	 			\lesssim &  \left( \frac{1}{\log n} + \frac{1}{\log(1/ \epsilon )} \right)^{ \frac{\beta - 1}{\beta} },
	 	\end{split}
	 \end{equation} where (a) is because when $\epsilon_{\max}$ is small enough and $n/\log(32/\alpha)$ is large enough, we have $\bar{t} \geq C$ for sufficiently large $C$, so $r/\bar{t} \leq 1/2$ will be satisfied when we take $r = \frac{\log(6/(1 - \epsilon_{\max}) )}{c' \beta \bar{t}^{\beta - 1}}$ when $\beta > 1$.

%%%%%%%%%%%%%%%%%%%%%%%%%%%%%%
\subsubsection{Mollifier distribution}
%%%%%%%%%%%%%%%%%%%%%%%%%%%%%

	Let us begin with the proof of the lower bound. First, it is for the same reason as in the generalized Gaussian with $\beta > 1$, it is enough to show that when $\epsilon \in [0, \epsilon_{\max}/2]$, 
	\begin{equation*}
		r_{\alpha}(\epsilon, [0,\epsilon_{\max}]) \gtrsim \left( \frac{1}{\log n} + \frac{1}{\log(1/ \epsilon )} \right)^{ \frac{\beta + 1}{\beta} }.
	\end{equation*}

By Corollary \ref{coro:length-lower-bound}, we only need to compute a lower bound for $\underline{r}(\epsilon)$ in the mollifier distribution with $\beta > 0$. Note that $f(x)$ is clearly unimodal. Next, let us check the monotonicity of $h_r(x) = \frac{f(x-r)}{f(x)} = \exp\left( \frac{1}{(1 - x^2)^\beta} - \frac{1}{(1 - (x-r)^2 )^\beta} \right)$ when $x \geq r$.
	Let $s(x) = \frac{1}{(1 - x^2)^\beta} - \frac{1}{(1 - (x-r)^2 )^\beta}$, then 
	\begin{equation*}
	\begin{split}
		s'(x) &= - \beta( 1- x^2 )^{-(\beta + 1)} \cdot (-2x) + \beta ( 1 - (x-r)^2  )^{-(\beta + 1)}\cdot (-2(x -r)) \\
		& = \frac{2x \beta ( 1 - (x-r)^2 )^{\beta + 1} - 2(x - r) \beta  ( 1- x^2 )^{\beta + 1}}{( 1- x^2 )^{\beta + 1} ( 1 - (x-r)^2 )^{\beta + 1} } \\
		& \geq \frac{2x \beta ( 1 - x^2 )^{\beta + 1} - 2(x - r) \beta  ( 1- x^2 )^{\beta + 1}}{( 1- x^2 )^{\beta + 1} ( 1 - (x-r)^2 )^{\beta + 1} } \geq 0,
	\end{split}
	\end{equation*} so $s(x)$ is an nondecreasing function of $x$ when $x \geq r$. In addition, when $r \leq \frac{(1 - \underline{t}^2)}{6(\beta \vee 1)}$, then  
	\begin{equation} \label{ineq:mollifier-sup-lower-bound}
		\begin{split}
			 h_r(r + \underline{t}) & = \exp\left( \frac{1}{(1 - (r + \underline{t})^2)^\beta} - \frac{1}{(1 - \underline{t}^2 )^\beta} \right)  \\
			& = \exp \left\{  \frac{1}{(1 - \underline{t}^2 )^\beta} \left[ \left( \frac{1 - \underline{t}^2}{1 - (\underline{t}+r)^2} \right)^\beta -1 \right]  \right\} \\
			& =  \exp \left\{  \frac{1}{(1 - \underline{t}^2 )^\beta} \left[  \frac{1}{(1 - \frac{r^2 + 2\underline{t}r}{1 - \underline{t}^2} )^\beta}  -1 \right]  \right\} \\
			& \overset{(a)}\leq \exp \left\{  \frac{1}{(1 - \underline{t}^2 )^\beta} \left[  \frac{1}{(1 - \frac{3r}{1 - \underline{t}^2} )^\beta}  -1 \right]  \right\} \\
			& \overset{(b)}\leq  \exp \left\{  \frac{1}{(1 - \underline{t}^2 )^\beta} \left[  \frac{1}{1 - \frac{3c r \beta}{1 - \underline{t}^2}}  -1 \right]  \right\}  \\
			& \overset{(c)}\leq  \exp \left\{  \frac{1}{(1 - \underline{t}^2 )^\beta}\frac{3c' r \beta}{1 - \underline{t}^2} \right\} = \exp \left\{ \frac{3c' r \beta}{(1 - \underline{t}^2)^{\beta + 1}} \right\},
		\end{split}
	\end{equation} where in (a), we use the property $r^2 + 2\underline{t}r \leq r^2 + 2r \leq 3r$ since $r,\underline{t} < 1$; in (b), we use the property $(1 - \frac{3r}{1 - \underline{t}^2} )^\beta \geq 1 - c\frac{3r \beta}{1 - \underline{t}^2}$ for some constant $c < 1$ depending on $\beta$ only given $\frac{3r}{1 - \underline{t}^2} \leq 1/2$; (c) is because $\frac{3c r \beta}{1 - \underline{t}^2} \leq 1/2$.
	
Thus, by Lemma \ref{lem:r-arrow},
\begin{equation*}
	\begin{split}
		\underline{r}(\epsilon) \geq r^{\downarrow}(\epsilon)  & = \sup \left\{r \geq 0 : h_r\left( r + \underline{t} \right) \leq \frac{1- \epsilon_{\max}/2 }{1- \epsilon_{\max} }  \right\} \\
		& \geq \sup \left\{r \geq 0 : h_r\left( r + \underline{t} \right) \leq \frac{1- \epsilon_{\max}/2 }{1- \epsilon_{\max} }, r \leq \frac{(1 - \underline{t}^2)}{6(\beta \vee 1)}  \right\} \\
		& \overset{ \eqref{ineq:mollifier-sup-lower-bound} }\geq   \sup \left\{r \geq 0 : \exp \left\{ \frac{3c' r \beta}{(1 - \underline{t}^2)^{\beta + 1}} \right\} \leq \frac{1- \epsilon_{\max}/2 }{1- \epsilon_{\max} }, r \leq \frac{(1 - \underline{t}^2)}{6(\beta \vee 1)}  \right\} \\
		& = \sup \left\{r \geq 0 :  r\leq \frac{(1 - \underline{t}^2)^{\beta + 1}}{3 c' \beta} \log \left(\frac{1- \epsilon_{\max}/2 }{1- \epsilon_{\max} } \right) , r \leq \frac{(1 - \underline{t}^2)}{6(\beta \vee 1)}  \right\}\\
		& \overset{(a)}= \frac{(1 - \underline{t}^2)^{\beta + 1}}{3 c' \beta} \log \left(\frac{1- \epsilon_{\max}/2 }{1- \epsilon_{\max} } \right)\\
		& \overset{\textnormal{Lemma } \ref{lm:mollifier-quantile} }\gtrsim  \left( 1/\log(1/\underline{q}(\epsilon)) \right)^{(\beta + 1)/\beta} \gtrsim \left( \frac{1}{\log n} + \frac{1}{\log(1/ \epsilon )} \right)^{ \frac{\beta + 1}{\beta} },
	\end{split}
\end{equation*} where (a) is because when $n$ is sufficiently large and $\epsilon_{\max}$ is sufficiently small, then $1 - \underline{t}^2$ is smaller than a sufficiently small constant so that when we take $r = \frac{(1 - \underline{t}^2)^{\beta + 1}}{3 c \beta} \log \left(\frac{1- \epsilon_{\max}/2 }{1- \epsilon_{\max} } \right)$, it satisfies $r \leq \frac{(1 - \underline{t}^2)}{6(\beta \vee 1)}$. 

Next, let us move on to the proof of the upper bound.  We have already shown that $h_r(x)$ is a nondecreasing function of $x$ when $x \geq r$, then by Lemma \ref{lem:r-arrow}, it is enough to show $r^{\uparrow}(\epsilon) \lesssim \left( \frac{1}{\log n} + \frac{1}{\log(1/ \epsilon )} \right)^{ \frac{\beta + 1}{\beta} }$. Now if we assume $ 1 - \bar{t}^2 \leq 3/4$ and $\frac{r}{1 - \bar{t}^2} \leq \frac{1}{2} \wedge \frac{1}{\beta}$, then % $x = r + F^{-1} ( 1 - \bar{q}(\epsilon) )$ and let
\begin{equation}\label{ineq:gx-bound2-mollifier}
	\begin{split}
		h_r(r+\bar{t}) 
		& \overset{ \eqref{ineq:mollifier-sup-lower-bound} }=  \exp \left\{  \frac{1}{(1 - \bar{t}^2 )^\beta} \left[  \frac{1}{(1 - \frac{r^2 + 2\bar{t}r}{1 - \bar{t}^2} )^\beta}  -1 \right]  \right\} \\
		& \overset{(a)}\geq \exp \left\{  \frac{1}{(1 - \bar{t}^2 )^\beta} \left[  \frac{1}{(1 - \frac{r}{1 - \bar{t}^2} )^\beta}  -1 \right]  \right\} \\
		& \overset{(b)}\geq \exp \left\{  \frac{1}{(1 - \bar{t}^2 )^\beta} \left[  \frac{1}{1 - \frac{c' r \beta}{1 - \bar{t}^2}}  -1 \right]  \right\} \\
		& =  \exp \left\{  \frac{c' r \beta}{(1 - \bar{t}^2 )^{\beta + 1}}\ \frac{1}{1 - \frac{c' r \beta}{1 - \bar{t}^2}}  \right\} \geq \exp \left\{  \frac{c'' r \beta}{(1 - \bar{t}^2 )^{\beta + 1}}  \right\}
	\end{split}
\end{equation} where (a) is because $\bar{t} \geq 1/2$ given $ 1 - \bar{t}^2 \leq 3/4$; in (b), we use the property $(1 - \frac{r}{1 - \bar{t}^2} )^\beta \leq 1 - c'\frac{r \beta}{1 - \bar{t}^2}$ for some constant $c' < 1$ depending on $\beta$ only given $\frac{r}{1 - \bar{t}^2} \leq 1/2$ and note that $c'\frac{r \beta}{1 - \bar{t}^2} \leq c'$ by our constraint on $r$. Then by Lemma \ref{lem:r-arrow}, we have
\begin{equation*}
	\begin{split}
		 r^{\uparrow}(\epsilon)= &    \inf \Big\{r \geq 0 :  h_r(r + \bar{t}) \geq \frac{6}{1 - \epsilon_{\max}}   \Big\} \\
		 \leq &  \inf \Big\{r \geq 0 :  h_r(r + \bar{t}) \geq \frac{6}{1 - \epsilon_{\max}}, 1 - \bar{t}^2 \leq 3/4, \frac{r}{1 - \bar{t}^2} \leq \frac{1}{2} \wedge \frac{1}{\beta}  \Big\}\\
		 \overset{ \eqref{ineq:gx-bound2-mollifier} }\leq & \inf \Big\{r \geq 0 :  \exp \left\{  \frac{c' r \beta}{(1 - \bar{t}^2 )^{\beta + 1}}  \right\} \geq \frac{6}{1 - \epsilon_{\max}}, 1 - \bar{t}^2 \leq 3/4, \frac{r}{1 - \bar{t}^2} \leq \frac{1}{2} \wedge \frac{1}{\beta}  \Big\} \\
		 = & \inf \Big\{r \geq 0 :  r \geq \frac{\log(6/(1 - \epsilon_{\max}) ) (1 - \bar{t}^2)^{\beta + 1 }}{c' \beta}, 1 - \bar{t}^2 \leq 3/4, \frac{r}{1 - \bar{t}^2} \leq \frac{1}{2} \wedge \frac{1}{\beta}  \Big\} \\
		 \overset{(a)}= &  \frac{\log(6/(1 - \epsilon_{\max}) ) (1 - \bar{t}^2)^{\beta + 1 }}{c' \beta} \\
		 \overset{ \textnormal{Lemma } \ref{lm:mollifier-quantile} }\leq & \frac{\log(6/(1 - \epsilon_{\max}) ) \left( 2/\log(1/\bar{q}(\epsilon)) \right)^{(\beta + 1)/\beta}}{c' \beta} \\
		 \lesssim& \left( \frac{1}{\log n} + \frac{1}{\log(1/ \epsilon )} \right)^{ \frac{\beta + 1}{\beta} },
	\end{split}
\end{equation*} where (a) is because when $n/\log(32/\alpha)$ is sufficiently large and $\epsilon_{\max}$ is sufficiently small, then $1 - \bar{t}^2$ is smaller than a sufficiently small constant so that $1 - \bar{t}^2 \leq 3/4, \frac{r}{1 - \bar{t}^2} \leq \frac{1}{2} \wedge \frac{1}{\beta}  $ are both satisfied when we take $r = \frac{\log(6/(1 - \epsilon_{\max}) ) (1 - \bar{t}^2)^{\beta + 1 }}{c' \beta}$.

%%%%%%%%%%%%%%%%%%%%%%%%%%%%%%
\subsubsection{Bates distribution}
%%%%%%%%%%%%%%%%%%%%%%%%%%%%%

	Let us begin with the proof of the lower bound. First, it is enough to show that when $\epsilon \in [0, \epsilon_{\max}/2]$, 
	\begin{equation*}
		r_{\alpha}(\epsilon, [0,\epsilon_{\max}]) \gtrsim \left( \frac{1}{n} + \epsilon \right)^{ 1/k }.
	\end{equation*}  

We will compute a lower bound of $\underline{r}(\epsilon)$  in Corollary \ref{coro:length-lower-bound} for the Bates distribution. For convenience, we will divide the proof based on $k = 1$ and $k \geq 2$.

{\bf (Case 1: $k = 1$)}  Let
	\begin{equation} \label{eq:uniform-hr}
		h_r(x) := \frac{f(x-r)}{f(x)} = \left\{\begin{array}{l l}
			0, & -1/2 \leq x <  -1/2 + r; \\
			1, & -1/2 + r \leq x \leq 1/2; \\
			\infty, & 1/2 < x \leq 1/2 + r.
		\end{array} \right.
	\end{equation}
	Then
	\begin{equation*}
		\begin{split}
			\underline{r}(\epsilon)  & = \sup \left\{r \geq 0: \sup_{- 1/2 \leq x \leq r + \underline{t} } h_r\left( x \right) \leq \frac{1- (\epsilon \vee \frac{\alpha}{n}) }{1- \epsilon_{\max} }  \right\} \\
			& \overset{ \eqref{eq:uniform-hr} } = \sup \left\{r \geq 0: r + \underline{t} \leq 1/2 \right\} \\
			&  = \underline{q}(\epsilon) \gtrsim  \epsilon + \frac{1}{n}.
		\end{split}
	\end{equation*}

{\bf (Case 2: $k \geq 2$)} Notice that the density of Bates distribution is unimodal, this is because Bates distribution is the convolution of uniforms, thus its density is essentially a weighted average of the density of the uniform distribution with maximum weight in the center and decaying weight when moving beyond the center. Then $\sup_{-1/2 \leq x \leq r+\underline{t}} h_r(x) = \sup_{r \leq x \leq r+\underline{t}} h_r(x)$ and 
\begin{equation} \label{ineq:Bates-lower-bound-first-step}
	\begin{split}
		\underline{r}(\epsilon)  &= \sup \left\{r \geq 0 : \sup_{-1/2 \leq x \leq r + \underline{t} } h_r( x) \leq \frac{1- (\epsilon \vee \frac{\alpha}{n}) }{1- \epsilon_{\max} }  \right\}  \\
		 &= \sup \left\{r \geq 0 : \sup_{r \leq x \leq r + \underline{t} } h_r( x) \leq \frac{1- (\epsilon \vee \frac{\alpha}{n}) }{1- \epsilon_{\max} }  \right\} \\
		 & \overset{(a)}\geq \sup \left\{r \geq 0 : \sup_{r \leq x \leq r + \underline{t} } h_r( x) \leq \frac{1- \epsilon_{\max}/2 }{1- \epsilon_{\max} }  \right\},
	\end{split}
\end{equation} where (a) is because when $n$ is large, we have $\epsilon \vee \frac{\alpha}{n} \leq \epsilon_{\max}/2$.

Next, let us study the value $\sup_{r \leq x \leq r+\underline{t}} h_r(x)$. First, we can let $\epsilon_{\max}$ to be small enough and $n$ to be large enough so that $ \underline{q}(\epsilon) \leq 1/(k(k-1)!)$ and $\underline{t} \in (1/2 - 1/(2k), 1/2]$. We will divide the regime $[r, r+\underline{t}]$ into two parts: (i) $[r, 1/2 - 1/(2k)]$ and (ii) $[1/2 - 1/(2k), r + \underline{t}]$. Since the density $f(x)$ is continuous, $h_r(x)$ as a function of $(r,x)$ is continuous and compact on $[0, c] \times [0,1/2 - 1/k]$ for any small constant $c$. So $h_r(x)$ as a function of $(r,x)$ is uniformly continuous, i.e., for any $\gamma > 0$, $\exists \delta > 0$ such that when $\sqrt{(x_1 - x_2)^2 + (r_1 - r_2)^2} \leq \delta$, $|h_{r_1}(x_1) - h_{r_2}(x_2)| \leq \gamma$. Take $\gamma = \frac{\epsilon_{\max}/2}{1- \epsilon_{\max}}$, then there exists a universal constant $\delta^* > 0$ so that when $r \in [0,\delta^*]$, we have
\begin{equation} \label{ineq:bates-uniform-con}
	\sup_{0 \leq x \leq 1/2-1/k } |h_r(x) - h_0(x)| \leq \frac{\epsilon_{\max}/2}{1- \epsilon_{\max}}.
\end{equation}
In addition, $h_0(x) = 1$ for all $x$. Thus \eqref{ineq:bates-uniform-con} implies that 
\begin{equation} \label{ineq:uniform-con-result}
	\sup_{0 \leq x \leq 1/2-1/k } h_r(x) \leq \frac{1-\epsilon_{\max}/2 }{1- \epsilon_{\max}} \textnormal{ for all } r \in [0, \delta^*]. 
\end{equation} Thus,
\begin{equation}\label{ineq:uniform-control-consequence}
	\begin{split}
		& \left\{r \geq 0 : \sup_{r \leq x \leq r + \underline{t} } h_r( x) \leq \frac{1- \epsilon_{\max}/2 }{1- \epsilon_{\max} }  \right\} \\
		 & \supseteq \left\{r \geq 0 : \sup_{r \leq x \leq r + \underline{t} } h_r( x) \leq \frac{1- \epsilon_{\max}/2 }{1- \epsilon_{\max} }, r \leq \delta^*  \right\} \\
		& \overset{ \eqref{ineq:uniform-con-result} }= \left\{r \geq 0 : \sup_{1/2-1/k \leq x \leq r + \underline{t} } h_r( x) \leq \frac{1- \epsilon_{\max}/2 }{1- \epsilon_{\max} }, r \leq \delta^*  \right\}.
	\end{split}
\end{equation}

A combination of \eqref{ineq:Bates-lower-bound-first-step} and \eqref{ineq:uniform-control-consequence} yields,
\begin{equation*}
	\begin{split}
		\underline{r}(\epsilon)  & \overset{\eqref{ineq:Bates-lower-bound-first-step} }\geq  \sup \left\{r \geq 0 : \sup_{r \leq x \leq r + \underline{t} } h_r( x) \leq \frac{1- \epsilon_{\max}/2 }{1- \epsilon_{\max} }  \right\}  \\
		& \overset{ \eqref{ineq:uniform-control-consequence} }\geq \sup \left\{r \geq 0 : \sup_{1/2-1/k \leq x \leq r + \underline{t} } h_r\left( x \right) \leq \frac{1- \epsilon_{\max}/2 }{1- \epsilon_{\max} }, r \leq \delta^*  \right\} \\
		& \overset{(a)}= \sup \left\{r \geq 0 : \sup_{1/2-1/k \leq x \leq r + \underline{t} } \left(\frac{1/2 + r - x}{1/2 - x} \right)^{k-1}  \leq \frac{1- \epsilon_{\max}/2 }{1- \epsilon_{\max} }, r \leq \delta^*  \right\} \\
		& \overset{(b)} =  \sup \left\{r \geq 0 : \left(\frac{1/2 - \underline{t}}{1/2 - \underline{t} - r} \right)^{k-1}  \leq \frac{1- \epsilon_{\max}/2 }{1- \epsilon_{\max} }, r \leq \delta^*  \right\} \\
		& = \sup \left\{r \geq 0 : r \leq (1/2 - \underline{t}) \left(1 - \left(\frac{1- \epsilon_{\max}}{1- \epsilon_{\max}/2} \right)^{1/(k-1)}\right), r \leq \delta^*  \right\} \\
		& \overset{(c)}= (1/2 - \underline{t}) \left(1 - \left(\frac{1- \epsilon_{\max}}{1- \epsilon_{\max}/2} \right)^{1/(k-1)}\right)\\
		& \overset{ \textnormal{Lemma } \ref{lm:bates-quantile} }\gtrsim \underline{q}(\epsilon)^{1/k} \gtrsim  \left( \frac{1}{n} + \epsilon \right)^{ 1/k },
	\end{split}
\end{equation*} where (a) is because when $1/2 - 1/k \leq x \leq 1/2$, $f(x) = \frac{k^k}{(k-1)!} (1/2 - x)^{k-1}$ by Lemma \ref{lm:bates-quantile}; (b) is because $\frac{1/2 + r - x}{1/2 - x}$ is an increasing function of $x$; (c) is because we can take $\epsilon_{\max}$ to be small enough and $n$ to be large enough so that $\underline{t}$ is very close to $1/2$ so that when $r = (1/2 - \underline{t}) \left(1 - \left(\frac{1- \epsilon_{\max}}{1- \epsilon_{\max}/2} \right)^{1/(k-1)}\right)$, it satisfies $r \leq \delta^*$.

Next, let us show the upper bound. Again, we divide the rest of the proof into two cases: $k = 1$ and $k \geq 2$.

{\bf (Case 1: $k = 1$)} By Lemma \ref{lem:r-arrow}, it is enough to show $ r^{\uparrow}(\epsilon) \lesssim  \frac{1}{n} + \epsilon $, where $r^{\uparrow}(\epsilon) $ is defined in Lemma \ref{lem:r-arrow}.
	\begin{equation*}
		\begin{split}
			r^{\uparrow}(\epsilon) = &  \inf \Big\{r \geq 0 :  h_r(r + \bar{t}) \geq \frac{6}{1 - \epsilon_{\max}} \Big\} \\
		\overset{ \eqref{eq:uniform-hr} }= & \inf \Big\{r \geq 0 : r + \bar{t} > 1/2  \Big\} \\
		= & \bar{q}(\epsilon)  \lesssim \frac{1}{n} + \epsilon.
		\end{split}
	\end{equation*}

{\bf (Case 2: $k \geq 2$)}
Let $\epsilon_{\max}$ to be small enough and $n/\log(32/\alpha)$ to large enough so that the condition $\bar{t} > -1/k + 1/2$ is satisfied. Thus, $h_r(x) = \left(\frac{1/2 + r - x}{1/2 - x} \right)^{k-1}$ by Lemma \ref{lm:bates-quantile}, this shows that $h_r(x)$ is a nondecreasing function of $x$ when $x \geq r + F^{-1} ( 1 - \bar{q}(\epsilon) )$. Again by Lemma \ref{lem:r-arrow}, it is enough to show $r^{\uparrow}(\epsilon) \lesssim \left( \frac{1}{n} + \epsilon \right)^{ 1/k }$, and we have 
\begin{equation*}
	\begin{split}
		r^{\uparrow}(\epsilon)  &=   \inf \Big\{r \geq 0 :  h_r(r + \bar{t}) \geq \frac{6}{1 - \epsilon_{\max}} \Big\} \\
		= & \inf \Big\{r \geq 0 :  \frac{(1/2 - \bar{t})^{k-1}}{(1/2 - (\bar{t} + r))^{k-1}} \geq \frac{6}{1 - \epsilon_{\max}}\Big\} \\
		= & (1/2 - \bar{t}) \left( 1 - \left( \frac{1 - \epsilon_{\max} }{6} \right)^{1/(k-1)} \right)\\
		\overset{ \textnormal{Lemma } \ref{lm:bates-quantile} }= &\left(\frac{\bar{q}(\epsilon)  (k-1)!}{k^{k-1}}\right)^{1/k} \left( 1 - \left( \frac{1 - \epsilon_{\max} }{6} \right)^{1/(k-1)} \right) \\
		\lesssim & \left( \frac{1}{n} + \epsilon \right)^{ 1/k }.
	\end{split}
\end{equation*}  This finishes the proof.

%%%%%%%%%%%%%%%%%%%%%%%%%%%%%
\section{Proofs in Section \ref{sec:extension}}
%%%%%%%%%%%%%%%%%%%%%%%%%%%%%%

%%%%%%%%%%%%%%%%%%%%%%%%%%%%%%%%%%%%%%%%%%%%%%%%%
\subsection{Proof of Theorem \ref{th:optimal-eps-max}} \label{proof:opt-eps-max}
%%%%%%%%%%%%%%%%%%%%%%%%%%%%%%%%%%%%%%%%%%%%

%%%%%%%%%%%%%%%%%%%%%%%%%%%%%%%%%%%%%%%%%%%%
\subsubsection{Optimal Procedure and Length Upper Bound}
%%%%%%%%%%%%%%%%%%%%%%%%%%%%%%%%%%%%%%%%%%%%
The optimal ARCI procedure will be derived based on inverting robust testing problems. Based on Proposition \ref{prop:test-to-CI}, we know that to construct ARCI, we just need to design testing procedures to solve $\cH(\theta, \theta- r_{\epsilon}, \epsilon)$ and $\cH(\theta, \theta+ r_{\epsilon}, \epsilon)$. The testing procedures we design are
\begin{eqnarray}
\label{eq:test-eps-max} \quad ~~~~~ \phi_{\theta, \theta- r_{\epsilon}, \epsilon} &=& \indi \left\{ \frac{1}{n}\sum_{i=1}^n \indi \left \{ X_i - (\theta - r_{\epsilon} ) \geq t_{\epsilon} \right\} \leq 1-\Phi(t_{\epsilon})+ \epsilon  + \sqrt{\frac{\log(2/\alpha)}{2n}} \right\}, \\
\label{eq:test+eps-max}\quad ~~~~~  \phi_{\theta, \theta+ r_{\epsilon}, \epsilon} &=& \indi \left\{ \frac{1}{n}\sum_{i=1}^n \indi \left \{ X_i - (\theta + r_{\epsilon} ) \leq -t_{\epsilon} \right\} < 1-\Phi(t_{\epsilon})+ \epsilon  + \sqrt{\frac{\log(2/\alpha)}{2n}} \right\},
\end{eqnarray} where 
\begin{equation} \label{eq:t-eps-max}
	t_{\epsilon} = \epsilon_{\max} \sqrt{\frac{n}{2 \log(2/\alpha)} } \wedge \Phi^{-1} \left( 1 - \frac{1}{ \frac{\epsilon_{\max}}{ \sqrt{\frac{2 \log(2/\alpha)}{n}} + \epsilon } + 10 e }  \right)
\end{equation} The tests in \eqref{eq:test-eps-max} and \eqref{eq:test+eps-max} are slightly different from the ones in \eqref{eq:test-gaussian} and \eqref{eq:test+gaussian} and they can achieve better performance when $\epsilon_{\max} \to 0$. The guarantee of these tests is given as follows.
\begin{Lemma}\label{prop:upper-bound-Gaussian-eps-max}
There exists some universal constant $C>0$, such that for any $\alpha\in(0,1)$, $\epsilon_{\max} \in [0,0.05]$ and $n\geq C\log(1/\alpha)$, the testing functions $\{\phi_{\theta, \theta \pm r_{ \epsilon }, \epsilon}: \epsilon \in [0, \epsilon_{\max}],\theta\in\mathbb{R}\}$ defined by \eqref{eq:test-eps-max} and \eqref{eq:test+eps-max} with parameters $r_{\epsilon} = \frac{6 \epsilon_{\max} }{t_{\epsilon}}$ and $t_{\epsilon}$ set by \eqref{eq:t-eps-max} satisfy
\begin{equation*}
		\begin{split}
			\textnormal{(simultaneous Type-1 error)} &\quad \sup_Q P_{ \epsilon_{\max} ,\theta,Q}\left(\sup_{\epsilon\in[0,\epsilon_{\max}]}\phi_{\theta,\theta\pm r_{\epsilon},\epsilon}=1\right)\leq \alpha,\\
			\textnormal{(Type-2 error)} &\quad \sup_QP_{\epsilon,\theta\pm r_{\epsilon},Q}\left(\phi_{\theta,\theta\pm r_{\epsilon},\epsilon}=0\right)\leq \alpha,
		\end{split}
	\end{equation*}
for all $\theta\in\mathbb{R}$ and $\epsilon\in[0,\epsilon_{\max}]$.
\end{Lemma}	
%%%%%%%%%%%%%%%%%%%%%%%%%%%%%%%%%%
The proof of Lemma \ref{prop:upper-bound-Gaussian-eps-max} is deferred to the end of this section. Following the same derivation as in \eqref{eq:interval-step1}-\eqref{eq:interval-step2}, the ARCI based on inverting \eqref{eq:test-eps-max} and \eqref{eq:test+eps-max} via \eqref{def:CI-based-on-test} is 
 \begin{equation}\label{eq:adaptive-CI-eps-max}
\begin{split}
	\widehat{\CI} = \Big[ &\max_{0 \leq \epsilon \leq  \epsilon_{\max}  } \left(F_n^{-1}\left( 1-\Phi(t_{\epsilon})+ \epsilon  + \sqrt{\frac{\log(2/\alpha)}{2n}}  \right)  + t_{\epsilon} -  \frac{6 \epsilon_{\max} }{t_{\epsilon}} \right), \\
	& \min_{0 \leq \epsilon \leq  \epsilon_{\max}  } \left(F_n^{-1}\left( \Phi(t_{\epsilon})- \epsilon  - \sqrt{\frac{\log(2/\alpha)}{2n}}  \right)  - t_{\epsilon} +  \frac{6 \epsilon_{\max} }{t_{\epsilon}} \right) \Big],
\end{split}
\end{equation} where $t_{\epsilon}$ is provided in \eqref{eq:t-eps-max}. Therefore, the combination of Proposition \ref{prop:test-to-CI} and Lemma \ref{prop:upper-bound-Gaussian-eps-max} leads to the coverage and length guarantee of the ARCI in \eqref{eq:adaptive-CI-eps-max}. In particular, it implies $r_{\alpha}(\epsilon,[0,\epsilon_{\max}]) \leq 2r_{\epsilon} = \frac{12 \epsilon_{\max} }{t_{\epsilon}}$ and its order is given in the following lemma. % the following lemma shows the order of $r_{\epsilon} = \frac{6 \epsilon_{\max} }{t_{\epsilon}}$.
\begin{Lemma}\label{lm:order-eps-max}
		For any $\epsilon_{\max} \in (0,1]$, positive integer $n$ and $\alpha \in (0,1)$, then
	\begin{equation*}
		\epsilon_{\max}/t_{\epsilon} \asymp \sqrt{\frac{\log(1/\alpha)}{n}} + \frac{\epsilon_{\max}}{\sqrt{\log(n\epsilon_{\max}^2/\log(1/\alpha)+e)}} + \frac{\epsilon_{\max}}{\sqrt{\log\left(\frac{\epsilon_{\max}}{\epsilon}+e\right)}}, \forall \epsilon \in [0, \epsilon_{\max}]
	\end{equation*} where $t_{\epsilon}$ is given in \eqref{eq:t-eps-max}.
\end{Lemma} % \frac{1}{\sqrt{\log(1/\epsilon)}} + \frac{1}{\sqrt{\log( n/(\log(1/\alpha)) )}}

%%%%%%%%%%%%%%%%%%%%%%%%%%%%%%%%%%%%%%%%%%%%
\subsubsection{Length Lower Bound}
%%%%%%%%%%%%%%%%%%%%%%%%%%%%%%%%%%%%%%%%%%%%
To show the ARCI length lower bound, we will still resort to establishing the testing lower bound. The sharper testing lower bound we can establish for $\cH(\theta, \theta - r_{ \epsilon }, \epsilon)$ and $\cH(\theta, \theta + r_{ \epsilon }, \epsilon)$ with respect to $\epsilon_{\max}$ is given as follows.
\begin{Lemma} \label{prop:test-lower-bound-gaussian-eps-tozero}
	For any $n$ greater than a sufficiently large positive constant, $\alpha \in (0,1)$ and $\epsilon \in [0, \epsilon_{\max}]$, there exists some constant $c > 0$ only depending on $\alpha$, such that as long as
	$$r \leq c \left( \frac{1}{\sqrt{n}} + \frac{\epsilon_{\max}}{\sqrt{\log(n\epsilon_{\max}^2+e)}} + \frac{\epsilon_{\max}}{\sqrt{\log\left(\frac{\epsilon_{\max}}{\epsilon}+e\right)}} \right),$$
	we have
		\begin{equation}
			\inf_{Q_0, Q_1} \TV(P^{\otimes n}_{\epsilon_{\max}, \theta, Q_0}, P^{\otimes n}_{\epsilon, \theta \pm r, Q_1}) \leq 0.6, \label{eq:gau-tv-match-eps-tozero}
		\end{equation}
for all $\theta\in\mathbb{R}$.
\end{Lemma}

%%%%%%%%%%%%%%%%%%%%%%%%%%%%%%%%%%%%%%%%%%%%%%
We defer the proof of Lemma \ref{prop:test-lower-bound-gaussian-eps-tozero} to the end of this section. With this lemma, the establishment of the length lower bound in Theorem \ref{th:optimal-eps-max} is similar to the proof of Theorem \ref{th:lower-bound-Gaussian}. Concretely, we prove the claim by a contradiction argument. Take $c$ to be the small constant in Lemma \ref{prop:test-lower-bound-gaussian-eps-tozero}. Suppose $r_{\alpha}(\epsilon, [0,\epsilon_{\max}]) > c \left( \frac{1}{\sqrt{n}} + \frac{\epsilon_{\max}}{\sqrt{\log(n\epsilon_{\max}^2+e)}} + \frac{\epsilon_{\max}}{\sqrt{\log\left(\frac{\epsilon_{\max}}{\epsilon}+e\right)}} \right)$ for all $\epsilon \in [0, \epsilon_{\max}]$ does not hold, i.e., there exists a $\epsilon \in [0,\epsilon_{\max}]$ and an ARCI $\widehat{\CI}$ such that the following conditions hold
\begin{equation*} 
	\inf_{ \epsilon \in [0, \epsilon_{\max}], \theta, Q} P_{\epsilon, \theta, Q}\left( \theta \in\widehat{\CI} \right) \geq 1-\alpha \quad \textnormal{ and } \quad \sup_{\theta, Q} P_{\epsilon, \theta, Q} \left( |\widehat{\CI}| \geq r \right) \leq \alpha,
\end{equation*} with $r = c \left( \frac{1}{\sqrt{n}} + \frac{\epsilon_{\max}}{\sqrt{\log(n\epsilon_{\max}^2+e)}} + \frac{\epsilon_{\max}}{\sqrt{\log\left(\frac{\epsilon_{\max}}{\epsilon}+e\right)}} \right)$. Then given any $\theta \in \bbR$, following the same analysis as in \eqref{ineq:typeI+II}, we can design a test using $\widehat{\CI}$ with the following guarantee
\begin{equation} \label{ineq:typeI+II-to-zero}
\begin{split}
	 \sup_{Q } P_{\epsilon_{\max}, \theta , Q } \left( \phi_{\theta, \theta- r, \epsilon} = 1  \right) + \sup_{Q } P_{ \epsilon, \theta - r, Q  } \left( \phi_{\theta, \theta- r, \epsilon} = 0  \right)  \leq  3 \alpha.
\end{split}
	\end{equation}

On the other hand, by Lemma \ref{prop:test-lower-bound-gaussian-eps-tozero} and the Neyman-Person Lemma, we have the following lower bound for testing $\cH(\theta, \theta - r, \epsilon)$, 
	\begin{equation}
		\begin{split}
		&\inf_{ \phi \in \{0,1 \} } \left\{ \sup_{Q } P_{\epsilon_{\max}, \theta, Q } \left( \phi = 1  \right) + \sup_{Q } P_{ \epsilon , \theta - r , Q  } \left( \phi = 0  \right) \right\}\\
			 &\geq 1 - \inf_{Q_0, Q_1} \TV\left(P^{\otimes n}_{\epsilon_{\max},\theta, Q_0}, P^{\otimes n}_{\epsilon, \theta - r , Q_1} \right) \\
			& \geq 1 - 0.6.
		\end{split}
	\end{equation}
	This contradicts with \eqref{ineq:typeI+II-to-zero} since $\alpha< 0.1$. Thus, the assumption does not hold, and this finishes the proof of the length lower bound part of this theorem.

%%%%%%%%%%%%%%%%%%%%%%%%%%%%%%%%%%%%%%%%%%%%%%%%%
\subsubsection{Proof of Lemma \ref{prop:upper-bound-Gaussian-eps-max}}
%%%%%%%%%%%%%%%%%%%%%%%%%%%%%%%%%%%%%%%%%%%%%%
Due to symmetry, here we only present the proof for the error control of $\phi_{\theta, \theta- r_{\epsilon}, \epsilon}$ for solving $\cH(\theta, \theta- r_{\epsilon}, \epsilon)$. 	We first bound the Type-2 error. For any $\theta, Q$,
		\begin{equation*}
			\begin{split}
				&P_{ \epsilon, \theta - r_{\epsilon}, Q }\left( \phi_{\theta, \theta- r_{\epsilon}, \epsilon} = 0 \right) \\
				=& P_{ \epsilon, \theta - r_{\epsilon}, Q }\left( \frac{1}{n}\sum_{i=1}^n \indi \left \{ X_i - (\theta - r_{\epsilon} ) \geq t_{\epsilon} \right\} > 1-\Phi(t_{\epsilon})+ \epsilon  + \sqrt{\frac{\log(2/\alpha)}{2n}}   \right) \\
				= & P_{ \epsilon, 0, Q }\left( \frac{1}{n}\sum_{i=1}^n \indi \left \{ X_i  \geq t_{\epsilon} \right\} > 1-\Phi(t_{\epsilon})+ \epsilon  + \sqrt{\frac{\log(2/\alpha)}{2n}} \right) \\
				= & P_{ \epsilon, \theta, Q }\left( \frac{1}{n}\sum_{i=1}^n \indi \left \{ X_i  \geq t_{\epsilon} \right\} - P_{ \epsilon, 0, Q }(X \geq t_{\epsilon}) > P_0(X \geq t_{\epsilon})+ \epsilon - P_{ \epsilon, 0, Q }(X \geq t_{\epsilon})  + \sqrt{\frac{\log(2/\alpha)}{2n}} \right) \\
				\overset{(a)}\leq &  P_{ \epsilon, \theta, Q }\left( \frac{1}{n}\sum_{i=1}^n \indi \left \{ X_i  \geq t_{\epsilon} \right\} - P_{ \epsilon, 0, Q }(X \geq t_{\epsilon}) > \sqrt{\frac{\log(2/\alpha)}{2n}} \right) \\
				\overset{(b)}\leq & \alpha
			\end{split}
		\end{equation*} where (a) is because $P_{ \epsilon, 0, Q }(X \geq t_{\epsilon}) \leq \epsilon + P_0(X \geq t_{\epsilon})$ by the model assumption and (b) is by Hoeffding's inequality.

Next, we move on to the simultaneous Type-1 error control, and we will show an equivalent statement is true: for any $\theta \in \bbR$,
	\begin{equation*}
		\inf_{Q}P_{ \epsilon_{\max}, \theta, Q }\left( \phi_{\theta, \theta- r_{\epsilon}, \epsilon} = 0 \text{ for all } \epsilon \in [0, \epsilon_{\max} ] \right) \geq 1- \alpha.
		\end{equation*} First, by the DKW inequality (see Lemma \ref{lm:DKW}), we have with probability at least $1 - \alpha$, the following event holds:
	\begin{equation*}
		(A) = \{ \textnormal{for all } x \in \bbR, |F_n(x) - F_{\epsilon_{\max}, \theta, Q}(x)| \leq \sqrt{ \log (2/\alpha)/(2n) } \},
	\end{equation*}where $F_{\epsilon_{\max}, \theta, Q }(\cdot)$ denotes the CDF of $P_{\epsilon_{\max}, \theta, Q}$. Then under $P_{ \epsilon_{\max}, \theta, Q }$, given $(A)$ happens, for any $\epsilon \in [0,\epsilon_{\max}]$,
	\begin{equation*}
		\begin{split}
			\phi_{\theta, \theta- r_{\epsilon}, \epsilon} = 0 &\Longleftrightarrow \frac{1}{n} \sum_{j =1}^n \indi(X_j - (\theta - r_{\epsilon}) \geq t_{\epsilon}) > 1-\Phi(t_{\epsilon})+ \epsilon  + \sqrt{\frac{\log(2/\alpha)}{2n}} \\
			& \Longleftarrow \frac{1}{n}\sum_{j =1}^n \indi(X_j - (\theta - r_{\epsilon}) > t_{\epsilon}) > 1-\Phi(t_{\epsilon})+ \epsilon  + \sqrt{\frac{\log(2/\alpha)}{2n}} \\
			& \Longleftrightarrow F_n(\theta - r_{\epsilon} + t_{\epsilon}) < \Phi(t_{\epsilon}) - \epsilon  - \sqrt{\frac{\log(2/\alpha)}{2n}}\\
			& \overset{(A)}\Longleftarrow F_{\epsilon_{\max}, \theta, Q }(\theta - r_{\epsilon} + t_{\epsilon}) + \sqrt{\frac{\log(2/\alpha)}{2n}} < \Phi(t_{\epsilon}) - \epsilon  - \sqrt{\frac{\log(2/\alpha)}{2n}} \\
			& \Longleftrightarrow P_{\epsilon_{\max}, \theta, Q }(X > \theta - r_{\epsilon} + t_{\epsilon}) > P_0(X \geq t_{\epsilon} ) + \epsilon + \sqrt{\frac{2\log(2/\alpha)}{n}} \\
			& \overset{(a)}\Longleftarrow (1 - \epsilon_{\max}) P_{0}(X \geq t_{\epsilon} - r_{\epsilon}) > P_0(X \geq t_{\epsilon} ) + \epsilon + \sqrt{\frac{2\log(2/\alpha)}{n}}\\
			& \Longleftrightarrow (1 - \epsilon_{\max}) P_{0}( t_{\epsilon} - r_{\epsilon}\leq X \leq t_{\epsilon}) >  \epsilon_{\max} P_0(X \geq t_{\epsilon} ) + \epsilon + \sqrt{\frac{2\log(2/\alpha)}{n}},
		\end{split}
	\end{equation*} where (a) is by the model assumption. Thus, to show the simultaneous Type-1 error control, it is enough to show 
	\begin{equation} \label{ineq:type-1-error-condition}
		(1 - \epsilon_{\max}) P_{0}( t_{\epsilon} - r_{\epsilon}\leq X \leq t_{\epsilon}) >  \epsilon_{\max} P_0(X \geq t_{\epsilon} ) + \epsilon + \sqrt{\frac{2\log(2/\alpha)}{n}}
	\end{equation} holds for all $\epsilon \in [0, \epsilon_{\max}]$. We will divide the rest of the proof into two cases.
	\begin{itemize}[leftmargin=*]
		\item (Case 1: $\frac{\epsilon_{\max}}{ \sqrt{\frac{2 \log(2/\alpha)}{n}} + \epsilon } \leq 10 e$) By the construction of $t_{\epsilon}$, we know
		\begin{equation} \label{ineq:case1-range-t}
			\Phi^{-1}\left(1 - \frac{1}{10e} \right) \wedge \epsilon_{\max} \sqrt{\frac{n}{2 \log(2/\alpha)} } \leq t_{\epsilon} \leq \Phi^{-1}\left(1 - \frac{1}{20e} \right), \quad \forall \epsilon \in [0, \epsilon_{\max}].
		\end{equation} Next, we can proceed based on the value of $\epsilon_{\max}$.
		
		{\noindent \bf (I: $ \epsilon_{\max} \sqrt{\frac{n}{2 \log(2/\alpha)} }  \leq \Phi^{-1}\left(1 - \frac{1}{10e} \right)$)} In this case, we also have $ \epsilon_{\max} \sqrt{\frac{n}{2 \log(2/\alpha)} } \leq \Phi^{-1} \left( 1 - \frac{1}{ \frac{\epsilon_{\max}}{ \sqrt{\frac{2 \log(2/\alpha)}{n}} + \epsilon } + 10 e }  \right)$, so 
		\begin{equation*}
			t_{\epsilon} = \epsilon_{\max} \sqrt{\frac{n}{2 \log(2/\alpha)} }  \quad \textnormal{ and } \quad r_{\epsilon} =  6 \sqrt{\frac{2 \log(2/\alpha)}{n}}.
		\end{equation*} Thus, for any $\epsilon \in [0, \epsilon_{\max}]$,
		\begin{equation*}
			\begin{split}
				&(1 - \epsilon_{\max}) P_{0}( t_{\epsilon} - r_{\epsilon}\leq X \leq t_{\epsilon}) -  \epsilon_{\max} P_0(X \geq t_{\epsilon} ) - \epsilon - \sqrt{\frac{2\log(2/\alpha)}{n}}\\
				&\geq  (1 - \epsilon_{\max}) P_{0}( t_{\epsilon} - r_{\epsilon}\leq X \leq t_{\epsilon}) -  2\epsilon_{\max} - \sqrt{\frac{2\log(2/\alpha)}{n}}\\
				& \overset{(a)}\geq (1 - \epsilon_{\max}) r_{\epsilon} \phi\left(\Phi^{-1}\left(1 - \frac{1}{10e} \right) \right) - \left(2 \Phi^{-1}\left(1 - \frac{1}{10e} \right) +1 \right)\sqrt{\frac{2\log(2/\alpha)}{n}} \\
				& \overset{(b)}\geq \left( 6\cdot0.95\cdot\phi\left(\Phi^{-1}\left(1 - \frac{1}{10e} \right) \right) - \left(2 \Phi^{-1}\left(1 - \frac{1}{10e} \right) +1 \right) \right) \sqrt{\frac{2\log(2/\alpha)}{n}}\\
				& > 0
			\end{split}
		\end{equation*} where (a) is because $ \epsilon_{\max}  \leq \sqrt{\frac{2 \log(2/\alpha)}{n}} \Phi^{-1}\left(1 - \frac{1}{10e} \right)$ and $\max\{|t_{\epsilon } - r_{\epsilon}|, t_{\epsilon } \} \leq \max\{t_{\epsilon },  r_{\epsilon}\} \leq \Phi^{-1}\left(1 - \frac{1}{10e} \right)$ where the last inequality is because $t_{\epsilon} \leq \Phi^{-1}\left(1 - \frac{1}{10e} \right)$ in the (I) case and $r_{\epsilon} \leq \Phi^{-1}\left(1 - \frac{1}{10e} \right)$ as long as $n/\log(1/\alpha)$ is large enough; (b) is because $\epsilon_{\max} \leq 0.05$. Thus, \eqref{ineq:type-1-error-condition} is checked.
		
		{\noindent \bf (II: $ \epsilon_{\max} \sqrt{\frac{n}{2 \log(2/\alpha)} }  > \Phi^{-1}\left(1 - \frac{1}{10e} \right)$)} In this case, we have $t_{\epsilon} \geq \Phi^{-1}\left(1 - \frac{1}{10e} \right)$ by \eqref{ineq:case1-range-t}. In addition,
		\begin{equation} \label{ineq:rlesst}
			r_{\epsilon} \leq t_{\epsilon} \quad \textnormal{ since } \quad \frac{6 \epsilon_{\max}}{t_{\epsilon}} \leq t_{\epsilon} \Longleftarrow 6 \epsilon_{\max} \leq 6 \cdot 0.05 \leq  \left(\Phi^{-1}\left(1 - \frac{1}{10e} \right) \right)^2.
		\end{equation} Then for any $\epsilon \in [0, \epsilon_{\max}]$,
		\begin{equation*}
			\begin{split}
				&(1 - \epsilon_{\max}) P_{0}( t_{\epsilon} - r_{\epsilon}\leq X \leq t_{\epsilon}) -  \epsilon_{\max} P_0(X \geq t_{\epsilon} ) - \epsilon - \sqrt{\frac{2\log(2/\alpha)}{n}}\\
				&\geq  (1 - \epsilon_{\max}) P_{0}( t_{\epsilon} - r_{\epsilon}\leq X \leq t_{\epsilon}) -  2\epsilon_{\max} - \sqrt{\frac{2\log(2/\alpha)}{n}}\\
				& \overset{(a)}\geq (1 - \epsilon_{\max}) r_{\epsilon} \phi\left( \Phi^{-1}\left(1 - \frac{1}{20e} \right)\right) - 2\epsilon_{\max} - \frac{\epsilon_{\max}}{\Phi^{-1}\left(1 - \frac{1}{10e} \right) }\\
				& \overset{(b)}\geq \frac{6(1 - \epsilon_{\max})}{\Phi^{-1}\left(1 - \frac{1}{20e} \right)} \epsilon_{\max}  \phi\left( \Phi^{-1}\left(1 - \frac{1}{20e} \right)\right) - 2\epsilon_{\max} - \frac{\epsilon_{\max}}{\Phi^{-1}\left(1 - \frac{1}{10e} \right) } \\
				& \geq \epsilon_{\max}\left( 6 \cdot 0.95 \frac{ \phi\left( \Phi^{-1}\left(1 - \frac{1}{20e} \right)\right)}{\Phi^{-1}\left(1 - \frac{1}{20e} \right)} - 2 - \frac{1}{\Phi^{-1}\left(1 - \frac{1}{10e} \right) } \right)\\
				& > 0
			\end{split}
		\end{equation*} where (a) is because $\sqrt{\frac{2\log(2/\alpha)}{n}} \leq \frac{\epsilon_{\max}}{\Phi^{-1}\left(1 - \frac{1}{10e} \right) }$ under the (II) setting and $r_{\epsilon} \leq t_{\epsilon}$, $$P_{0}( t_{\epsilon} - r_{\epsilon}\leq X \leq t_{\epsilon}) \geq r_{\epsilon} \phi(t_{\epsilon}) \geq r_{\epsilon} \phi\left( \Phi^{-1}\left(1 - \frac{1}{20e} \right)\right)$$ where the last inequality is because $t_{\epsilon} \leq \Phi^{-1}\left(1 - \frac{1}{20e} \right)$ by \eqref{ineq:case1-range-t}; (b) is by the choice of $r_{\epsilon}$ and $t_{\epsilon} \leq \Phi^{-1}\left(1 - \frac{1}{20e} \right)$. Thus, \eqref{ineq:type-1-error-condition} is checked.

		\item (Case 2: $\frac{\epsilon_{\max}}{ \sqrt{\frac{2 \log(2/\alpha)}{n}} + \epsilon }> 10 e$) In this case, we have 
		\begin{equation*}
			\begin{split}
				\Phi^{-1} \left( 1 - \frac{1}{ \frac{\epsilon_{\max}}{ \sqrt{\frac{2 \log(2/\alpha)}{n}} + \epsilon } + 10 e }  \right) \geq \Phi^{-1} \left( 1 - \frac{1}{ 20 e }  \right) \geq 2.
			\end{split}
		\end{equation*} Thus, by Lemma \ref{lm:normal-quantile}, we have
		\begin{equation*}
			\begin{split}
				\Phi^{-1} \left( 1 - \frac{1}{ \frac{\epsilon_{\max}}{ \sqrt{\frac{2 \log(2/\alpha)}{n}} + \epsilon } + 10 e }  \right) &\leq \sqrt{2\log \left( \frac{\epsilon_{\max}}{ \sqrt{\frac{2 \log(2/\alpha)}{n}} + \epsilon } + 10 e \right) } \\
				& \leq \sqrt{2\log \left( \frac{2\epsilon_{\max}}{ \sqrt{\frac{2 \log(2/\alpha)}{n}} + \epsilon }\right) }\\
				& \overset{(a)}\leq 2\log \left( \frac{2\epsilon_{\max}}{ \sqrt{\frac{2 \log(2/\alpha)}{n}} + \epsilon } \right)\\
				& \overset{(a)}\leq \frac{\epsilon_{\max}}{ \sqrt{\frac{2 \log(2/\alpha)}{n}} + \epsilon }  \leq \epsilon_{\max} \sqrt{\frac{n}{2 \log(2/\alpha)} }
			\end{split}
		\end{equation*} where in (a) we use the fact that $\frac{\epsilon_{\max}}{ \sqrt{\frac{2 \log(2/\alpha)}{n}} + \epsilon }> 10 e$. Thus, in this regime, we have
		\begin{equation*}
			t_{\epsilon} = \Phi^{-1} \left( 1 - \frac{1}{ \frac{\epsilon_{\max}}{ \sqrt{\frac{2 \log(2/\alpha)}{n}} + \epsilon } + 10 e }  \right).
		\end{equation*} This implies that 
		\begin{equation} \label{ineq:tail-bound-eps-max}
			\begin{split}
				P_0(X \geq t_{\epsilon}) = \frac{1}{ \frac{\epsilon_{\max}}{ \sqrt{\frac{2 \log(2/\alpha)}{n}} + \epsilon } + 10 e } =\frac{1}{\sqrt{2\pi}} \int^{\infty}_{t_{\epsilon}} \exp(-x^2/2) dx \leq \frac{1}{\sqrt{2\pi}} \frac{1}{t_{\epsilon}} \exp(-t_{\epsilon}^2/2),
			\end{split}
		\end{equation} where the last inequality is by Theorem 1.2.3 in \cite{durrett2019probability}. Then
		\begin{equation*}
			\begin{split}
				& (1 - \epsilon_{\max}) P_{0}( t_{\epsilon} - r_{\epsilon}\leq X \leq t_{\epsilon}) >  \epsilon_{\max} P_0(X \geq t_{\epsilon} ) + \epsilon + \sqrt{\frac{2\log(2/\alpha)}{n}}, \forall \epsilon \in [0, \epsilon_{\max}]\\
				\overset{(a)}\Longleftarrow & (1 - \epsilon_{\max}) r_{\epsilon} \phi(t_{\epsilon})  >  \epsilon_{\max} \frac{1}{\sqrt{2\pi}} \frac{1}{t_{\epsilon}} \exp(-t_{\epsilon}^2/2)+ \epsilon + \sqrt{\frac{2\log(2/\alpha)}{n}}, \forall \epsilon \in [0, \epsilon_{\max}] \\
				\Longleftrightarrow & (1 - \epsilon_{\max}) \frac{6 \epsilon_{\max} }{\sqrt{2\pi}t_{\epsilon}} \exp(-t_{\epsilon}^2/2)  >  \epsilon_{\max} \frac{1}{\sqrt{2\pi}} \frac{1}{t_{\epsilon}} \exp(-t_{\epsilon}^2/2)+ \epsilon + \sqrt{\frac{2\log(2/\alpha)}{n}}, \forall \epsilon \in [0, \epsilon_{\max}] \\
				\Longleftrightarrow & (1 - \epsilon_{\max}) 6 \epsilon_{\max} >  \epsilon_{\max} + \sqrt{2\pi}t_{\epsilon} \exp(t_{\epsilon}^2/2) \left( \epsilon + \sqrt{\frac{2\log(2/\alpha)}{n}}\right), \forall \epsilon \in [0, \epsilon_{\max}] \\
				\overset{\eqref{ineq:tail-bound-eps-max}}\Longleftarrow & (1 - \epsilon_{\max}) 6 \epsilon_{\max}  >  \epsilon_{\max} + \left(\frac{\epsilon_{\max}}{ \sqrt{\frac{2 \log(2/\alpha)}{n}} + \epsilon } + 10 e \right) \left( \epsilon + \sqrt{\frac{2\log(2/\alpha)}{n}}\right), \forall \epsilon \in [0, \epsilon_{\max}] \\
				\overset{(b)}\Longleftarrow & (1 - \epsilon_{\max}) 6 \epsilon_{\max}  >  \epsilon_{\max} + 2\epsilon_{\max}\\
				\Longleftarrow & 0.95 \cdot 6 >  3,
			\end{split}
		\end{equation*} where (a) is because \eqref{ineq:tail-bound-eps-max} and $t_{\epsilon} \geq r_{\epsilon}$ by the same argument as in \eqref{ineq:rlesst} and (b) is because $\frac{\epsilon_{\max}}{ \sqrt{\frac{2 \log(2/\alpha)}{n}} + \epsilon }> 10 e$ in Case 2. Thus, \eqref{ineq:type-1-error-condition} is checked.
	\end{itemize}
	This finishes the check of \eqref{ineq:type-1-error-condition} in two cases and finishes the proof of this lemma.

%%%%%%%%%%%%%%%%%%%%%%%%%%%%%%%%%%%%%%%%%%%%%%%%%
\subsubsection{Proof of Lemma \ref{lm:order-eps-max}} We divide the proof into two cases. First, when $\frac{\epsilon_{\max}}{ \sqrt{\frac{2 \log(2/\alpha)}{n}} + \epsilon } \leq 10 e$, then $\Phi^{-1} \left( 1 - \frac{1}{ \frac{\epsilon_{\max}}{ \sqrt{\frac{2 \log(2/\alpha)}{n}} + \epsilon } + 10 e }  \right) \asymp 1$, so $t_{\epsilon}  \asymp 1 \wedge \epsilon_{\max} \sqrt{\frac{n}{2 \log(2/\alpha)} }$. Thus,
\begin{equation*}
	\begin{split}
		\epsilon_{\max}/t_{\epsilon} \asymp \sqrt{\frac{\log(1/\alpha)}{n}}  + \epsilon_{\max}.  
	\end{split}
\end{equation*} At the same time, when $\frac{\epsilon_{\max}}{ \sqrt{\frac{2 \log(2/\alpha)}{n}} + \epsilon } \leq 10 e$,
\begin{equation*}
	\frac{\epsilon_{\max}}{\sqrt{\log(n\epsilon_{\max}^2/\log(1/\alpha)+e)}} + \frac{\epsilon_{\max}}{\sqrt{\log\left(\frac{\epsilon_{\max}}{\epsilon}+e\right)}} \asymp  \epsilon_{\max}.
\end{equation*} Thus,
\begin{equation*}
	\begin{split}
		\epsilon_{\max}/t_{\epsilon} \asymp \sqrt{\frac{\log(1/\alpha)}{n}}  + \frac{\epsilon_{\max}}{\sqrt{\log(n\epsilon_{\max}^2/\log(1/\alpha)+e)}} + \frac{\epsilon_{\max}}{\sqrt{\log\left(\frac{\epsilon_{\max}}{\epsilon}+e\right)}}.  
	\end{split}
\end{equation*}

Second, when $\frac{\epsilon_{\max}}{ \sqrt{\frac{2 \log(2/\alpha)}{n}} + \epsilon } \geq 10 e$. Since $\Phi^{-1} \left( 1 - \frac{1}{ \frac{\epsilon_{\max}}{ \sqrt{\frac{2 \log(2/\alpha)}{n}} + \epsilon } + 10 e }  \right) \geq \Phi^{-1} \left( 1 - \frac{1}{ 20 e }  \right) \geq 2$, by Lemma \ref{lm:normal-quantile}, we have
\begin{equation*}
	\begin{split}
		\Phi^{-1} \left( 1 - \frac{1}{ \frac{\epsilon_{\max}}{ \sqrt{\frac{2 \log(2/\alpha)}{n}} + \epsilon } + 10 e }  \right) & \asymp \sqrt{ \log \left(  \frac{\epsilon_{\max}}{ \sqrt{\frac{2 \log(2/\alpha)}{n}} + \epsilon } + 10 e \right) } \\
		&\asymp \sqrt{ \log \left(  \frac{\epsilon_{\max}}{ \sqrt{\frac{2 \log(2/\alpha)}{n}} + \epsilon } \right) } \\
		& \asymp \sqrt{\log\left(\epsilon^2_{\max} n/\log(1/\alpha) \right)} \wedge \sqrt{\log\left(\epsilon_{\max}/\epsilon \right)}.
	\end{split}
\end{equation*} Thus,
\begin{equation*}
	\begin{split}
		t_{\epsilon} \asymp \epsilon_{\max} \sqrt{\frac{n}{2 \log(2/\alpha)} } \wedge \sqrt{\log\left(\epsilon^2_{\max} n/\log(1/\alpha) \right)} \wedge \sqrt{\log\left(\epsilon_{\max}/\epsilon \right)},
	\end{split}
\end{equation*} and 
\begin{equation*}
	\begin{split}
		\epsilon_{\max}/t_{\epsilon} \asymp \sqrt{\frac{\log(1/\alpha)}{n}}  + \frac{\epsilon_{\max}}{\sqrt{\log(n\epsilon_{\max}^2/\log(1/\alpha))}} + \frac{\epsilon_{\max}}{\sqrt{\log\left(\frac{\epsilon_{\max}}{\epsilon}\right)}}.  
	\end{split}
\end{equation*} At the same time, when $\frac{\epsilon_{\max}}{ \sqrt{\frac{2 \log(2/\alpha)}{n}} + \epsilon } \geq 10 e$, 
\begin{equation*}
	\frac{\epsilon_{\max}}{\sqrt{\log(n\epsilon_{\max}^2/\log(1/\alpha))}} + \frac{\epsilon_{\max}}{\sqrt{\log\left(\frac{\epsilon_{\max}}{\epsilon}\right)}} \asymp \frac{\epsilon_{\max}}{\sqrt{\log(n\epsilon_{\max}^2/\log(1/\alpha)+e)}} + \frac{\epsilon_{\max}}{\sqrt{\log\left(\frac{\epsilon_{\max}}{\epsilon}+e\right)}}.
\end{equation*} Thus,
\begin{equation*}
	\epsilon_{\max}/t_{\epsilon} \asymp \sqrt{\frac{\log(1/\alpha)}{n}} + \frac{\epsilon_{\max}}{\sqrt{\log(n\epsilon_{\max}^2/\log(1/\alpha)+e)}} + \frac{\epsilon_{\max}}{\sqrt{\log\left(\frac{\epsilon_{\max}}{\epsilon}+e\right)}}.
\end{equation*} This finishes the proof.

%%%%%%%%%%%%%%%%%%%%%%%%%%%%%%%%%%%%%%%%%%%%%%%%%
\subsubsection{Proof of Lemma \ref{prop:test-lower-bound-gaussian-eps-tozero}}
%%%%%%%%%%%%%%%%%%%%%%%%%%%%%%%%%%%%%%%%%%%%%%
For simplicity, in this proof, we will use the notation $a_n \lesssim b_n$ to mean that $a_n \leq C b_n$ holds for some constant $C > 0$ independent of $n$, but $C$ could depend on $\alpha$. Next, we show the lower bound term by term.
\begin{itemize}[leftmargin=*]
	\item (Case 1: $r \lesssim  \frac{1}{\sqrt{n}}$) This is the lower bound we can establish even when there is no contamination, and it can be proved via standard Le Cam's two-point testing method.
	\item (Case 2: $r \lesssim  \frac{\epsilon_{\max}}{\sqrt{\log(n\epsilon_{\max}^2+e)}}$) First, if $\epsilon_{\max} \lesssim \frac{1}{\sqrt{n}}$, then $\frac{\epsilon_{\max}}{\sqrt{\log(n\epsilon_{\max}^2+e)}} \lesssim \frac{1}{\sqrt{n}}$ and the result follows from Case 1. Now, let us focus on the setting $\epsilon_{\max} \gtrsim \frac{1}{\sqrt{n}}$ and in this case, the assumption on $r$ is $r \lesssim  \frac{\epsilon_{\max}}{\sqrt{\log(n\epsilon_{\max}^2)}} \asymp  \frac{\epsilon_{\max}}{\sqrt{\log(n\epsilon_{\max}^2+e)}}$. Due to the location invariant of the Gaussian distribution, to show the statement, it is enough to show $\inf_{Q_0, Q_1} \TV(P^{\otimes n}_{\epsilon_{\max}, r, Q_0}, P^{\otimes n}_{\epsilon, 0, Q_1}) \leq 0.6$. To show that it is enough to show $H^2(P_{\epsilon_{\max}, r, Q_0}, P_{\epsilon, 0, Q_1}) \lesssim 1/n $ where $H^2(P, Q)$ denotes the Hellinger distance between $P$ and $Q$. This is because by \eqref{ineq:testing-to-hellinger}, we have
	\begin{equation*}
		1 - \TV\left( P^{\otimes n}_{\epsilon_{\max}, r, Q_0}, P^{\otimes n}_{\epsilon, 0, Q_1}\right) \geq \frac{1}{2} \left( 1- \frac{1}{2} H^2(P_{\epsilon_{\max}, r, Q_0}, P_{\epsilon, 0, Q_1}) \right)^{2n},
	\end{equation*} and when $$H^2(P_{\epsilon_{\max}, r, Q_0}, P_{\epsilon, 0, Q_1}) \leq c/n$$ for a small enough constant and $n$ is large enough, we can have $\frac{1}{2} \left( 1- \frac{1}{2} H^2(P_{\epsilon_{\max}, r, Q_0}, P_{\epsilon, 0, Q_1}) \right)^{2n} \geq 0.4$ and it implies that $ \TV\left( P^{\otimes n}_{\epsilon_{\max}, r, Q_0}, P^{\otimes n}_{\epsilon, 0, Q_1}\right) \leq 0.6$. 
	
	To show $H^2(P_{\epsilon_{\max}, r, Q_0}, P_{\epsilon, 0, Q_1}) \lesssim 1/n $, we will actually show 
	$\chi^2(P_{\epsilon_{\max}, r, Q_0}, P_{\epsilon, 0, Q_1}) \lesssim 1/n$ as $H^2(P,Q) \leq \chi^2(P,Q)$ for any $P,Q$ by \cite[Eq. (2.27)]{tsybakov2009introduction}. We consider the following construction of $Q_0$ and $Q_1$. Take $Q_1 = N(0, 1)$ and define the density function of $Q_0$ by 
\begin{equation} \label{eq:gaussian-q0-eps-tozero}
q_0(x) = \frac{\delta( \phi(x) - (1-\epsilon_{\max}) \phi(x-r) ) \indi(x \leq t)}{\epsilon_{\max}},
\end{equation} where $\delta = \frac{\epsilon_{\max}}{P_0(X \leq t) - (1-\epsilon_{\max}) P_r(X \leq t)}$ so that we can check $\int q_0(x) dx = 1$. Note that this construction strategy is slightly different from the one in the proof of Lemma \ref{prop:test-lower-bound-gaussian} and the new construction is better when $\epsilon_{\max} \to 0$. Here we truncate the difference $\phi(x) - (1-\epsilon_{\max}) \phi(x-r)$ at level $x \leq t$ to guarantee that $q_0(x)$ is a valid density, specifically, we need to choose $t$ such that $\phi(x) - (1-\epsilon_{\max}) \phi(x-r) \geq 0 $ for all $ x \leq t$ and this can be guaranteed as
\begin{equation*}
	\begin{split}
		& \phi(x) - (1-\epsilon_{\max}) \phi(x-r) \geq 0, \quad \forall x \leq t \\
		\Longleftrightarrow & \phi(x)/\phi(x-r) \geq 1 - \epsilon_{\max}, \quad \forall x \leq t \\ 
		\Longleftrightarrow & \exp(xr - r^2/2) \leq \frac{1}{1 - \epsilon_{\max} }, \quad \forall x \leq t \\
		\Longleftrightarrow & tr - r^2/2 \leq \log\left(\frac{1}{1 - \epsilon_{\max} } \right) \\
		\Longleftrightarrow & t \leq \frac{\log\left(\frac{1}{1 - \epsilon_{\max} } \right)}{r} + \frac{r}{2} \\
		\overset{(a)}\Longleftarrow & t \leq \frac{2 \epsilon_{\max} }{3r}
	\end{split}
\end{equation*} where (a) is because $\log\left(\frac{1}{1 - \epsilon_{\max} } \right) \geq \frac{2}{3} \epsilon_{\max}$ for any $\epsilon_{\max} \in [0, 1/3]$. So we will take $t = \frac{2 \epsilon_{\max} }{3r}$, and now $q_0(x)$ is a valid density. In the following, we will show that when $r \lesssim  \frac{\epsilon_{\max}}{\sqrt{\log(n\epsilon_{\max}^2)}}$, then $\chi^2(P_{\epsilon_{\max}, r, Q_0}, P_{\epsilon, 0, Q_1}) \lesssim 1/n$. 

For the specific construction of $Q_0$ and $Q_1$ above, we have 
\begin{equation} \label{ineq:chi-square-bound-eps-tozero}
	\begin{split}
		&\chi^2(P_{\epsilon_{\max}, r, Q_0}, P_{\epsilon, 0, Q_1}) = \int \frac{\left((1 - \epsilon_{\max}) \phi(x - r) + \epsilon_{\max} q_0(x) - \phi(x) \right)^2}{\phi(x)} dx \\
		&\overset{\eqref{eq:gaussian-q0-eps-tozero}}= \int \frac{\left((1 - \epsilon_{\max}) \phi(x - r) + \delta( \phi(x) - (1-\epsilon_{\max}) \phi(x-r) ) \indi(x \leq t) - \phi(x) \right)^2}{\phi(x)} dx \\
		& = (1 - \delta)^2 \int_{-\infty}^t \frac{\left((1 - \epsilon_{\max}) \phi(x - r)  - \phi(x) \right)^2}{\phi(x)} dx + \int_t^{\infty} \frac{\left((1 - \epsilon_{\max}) \phi(x - r)  - \phi(x) \right)^2}{\phi(x)} dx.
	\end{split}
\end{equation} Next, we bound the two integrals at the end of \eqref{ineq:chi-square-bound-eps-tozero} separately. First, notice that when $\epsilon_{\max} \gtrsim \frac{1}{\sqrt{n}}$, we have $n \epsilon_{\max}^2 \gtrsim 1$ and
\begin{equation} \label{ineq:range-r-t}
	r \lesssim  \frac{\epsilon_{\max}}{\sqrt{\log(n\epsilon_{\max}^2)}} \lesssim \epsilon_{\max} \lesssim 1 \quad \textnormal{and}\quad  t = \frac{2 \epsilon_{\max} }{3r} \gtrsim \sqrt{\log(n\epsilon_{\max}^2)} \gtrsim 1 \gtrsim r.
\end{equation} We begin by bounding the second integration term at the end of \eqref{ineq:chi-square-bound-eps-tozero}. 
\begin{equation} \label{ineq:chi-square-term2}
	\begin{split}
		&\int_t^{\infty} \frac{\left((1 - \epsilon_{\max}) \phi(x - r)  - \phi(x) \right)^2}{\phi(x)}  dx \\
		&= (1 - \epsilon_{\max})^2 \exp(r^2) P_{0}(X \geq t - 2r) + P_0(X \geq t) - 2(1 - \epsilon_{\max}) P_0(X \geq t - r) \\
		& \overset{(a)}\leq (1 - 2 \epsilon_{\max})(1 + Cr^2) P_{0}(X \geq t - 2r) + P_0(X \geq t) - 2(1 - \epsilon_{\max}) P_0(X \geq t - r) \\
		& \quad \quad + \epsilon^2_{\max} P_{0}(X \geq t - 2r) \\
		& \leq (1 - 2 \epsilon_{\max}) P_{0}(X \geq t - 2r) + P_0(X \geq t) - 2(1 - \epsilon_{\max}) P_0(X \geq t - r) \\
		& \quad \quad  + C(\epsilon^2_{\max}+r^2) P_{0}(X \geq t - 2r),
	\end{split}
\end{equation} where (a) is by Taylor expansion. If we denote $f(r) = (1 - 2 \epsilon_{\max}) P_{0}(X \geq t - 2r) + P_0(X \geq t) - 2(1 - \epsilon_{\max}) P_0(X \geq t - r)$, by Taylor expansion, we have
\begin{equation*}
	\begin{split}
		f(r) = f(0) + rf'(0) + r^2 f''(\xi)/2 
	\end{split}
\end{equation*} for some $\xi \in [0, r]$. It is easy to check $f(0) = 0$, $f'(0) = -2 \epsilon_{\max} \phi(t)$ and $|f''(\xi)| \lesssim \exp(-t^2/8)t$ when $2r \leq t/2$, which is implied by \eqref{ineq:range-r-t}. Thus,
\begin{equation} \label{ineq:term1-1}
	|f(r)| \lesssim \epsilon_{\max} r \phi(t) + r^2 \exp(-t^2/8)t.
\end{equation} At the same time, 
\begin{equation} \label{ineq:term1-2}
	P_{0}(X \geq t - 2r) \leq P_{0}(X \geq t/2) \lesssim \exp(-t^2/8)/t
\end{equation} when $2r \leq t/2$. So plug \eqref{ineq:term1-1} and \eqref{ineq:term1-2} into \eqref{ineq:chi-square-term2}, we have 
\begin{equation*}
	\begin{split}
		&\int_t^{\infty} \frac{\left((1 - \epsilon_{\max}) \phi(x - r)  - \phi(x) \right)^2}{\phi(x)}  dx \\
		& \lesssim  \epsilon_{\max} r \phi(t) + r^2 \exp(-t^2/8)t + (\epsilon^2_{\max}+r^2)\exp(-t^2/8)/t \\
		& \overset{(a)}\lesssim  \epsilon_{\max} r \phi(t) \\
		& \overset{(b)}\lesssim  \epsilon_{\max} \frac{\epsilon_{\max}}{\sqrt{\log(n\epsilon_{\max}^2)}} \exp(-t^2/2) \\
		& \overset{(c)}\lesssim  \epsilon_{\max} \frac{\epsilon_{\max}}{\sqrt{\log(n\epsilon_{\max}^2)}} \frac{1}{n \epsilon^2_{\max} } \lesssim  1/n
	\end{split}
\end{equation*} where in (a) we use the fact $t =  \frac{2 \epsilon_{\max} }{3r}$ and $t \gtrsim 1$; (b) is by the assumption on $r$; (c) is because $t \gtrsim \sqrt{\log(n\epsilon_{\max}^2)}$.

	Next, we bound the first integration term in \eqref{ineq:chi-square-bound-eps-tozero}. First,
	\begin{equation*}
		\begin{split}
			\int_{-\infty}^t \frac{\left((1 - \epsilon_{\max}) \phi(x - r)  - \phi(x) \right)^2}{\phi(x)} dx &\leq \int_{-\infty}^{\infty} \frac{\left((1 - \epsilon_{\max}) \phi(x - r)  - \phi(x) \right)^2}{\phi(x)} dx \\
			& = (1 - \epsilon_{\max})^2 \exp(r^2) + 1 - 2(1 - \epsilon_{\max})\\
			& \lesssim (1 - \epsilon_{\max})^2(1 + Cr^2) -1 + 2 \epsilon_{\max} \\
			& \lesssim r^2 + \epsilon^2_{\max}.
		\end{split}
	\end{equation*} At the same time, by the definition of $\delta$, 
	\begin{equation*}
		\begin{split}
			|1- \delta| &= \left| 1- \frac{\epsilon_{\max}}{P_0(X \leq t) - (1-\epsilon_{\max}) P_r(X \leq t)} \right| \\
			& = \left| \frac{P_0(X \leq t) - (1-\epsilon_{\max}) P_r(X \leq t) - \epsilon_{\max}}{P_0(X \leq t) - (1-\epsilon_{\max}) P_r(X \leq t)} \right| \\
			& = \left| \frac{P_0(X \leq t) - (1-\epsilon_{\max}) P_0(X \leq t-r) - \epsilon_{\max}}{P_0(X \leq t) - (1-\epsilon_{\max}) P_0(X \leq t-r)} \right| \\
			& =  \left| \frac{- \epsilon_{\max} P_0(X \geq t) + (1-\epsilon_{\max}) P_0(t-r \leq X \leq t) }{ \epsilon_{\max}P_0(X \leq t) + (1-\epsilon_{\max}) P_0(t-r \leq X \leq t)} \right| \\
			& \leq \frac{P_0(X \geq t)}{P_0(X \leq t)} + \frac{(1-\epsilon_{\max}) P_0(t-r \leq X \leq t)}{\epsilon_{\max}P_0(X \leq t)}  \\
			& \overset{\eqref{ineq:range-r-t}}\leq 2P_0(X \geq t) + 2(1-\epsilon_{\max}) r \phi(t-r)/\epsilon_{\max} \\
			& \overset{\eqref{ineq:range-r-t}  }\lesssim \phi(t)/t + \frac{r}{\epsilon_{\max}} \phi(t/2)\\
			& \lesssim \frac{r}{\epsilon_{\max}} \phi(t/2),
		\end{split}
	\end{equation*}  where in the last inequality, we use the fact $t =  \frac{2 \epsilon_{\max} }{3r}$. Thus, for the first integration term in \eqref{ineq:chi-square-bound-eps-tozero}, 
	\begin{equation*}
		\begin{split}
			&(1 - \delta)^2 \int_{-\infty}^t \frac{\left((1 - \epsilon_{\max}) \phi(x - r)  - \phi(x) \right)^2}{\phi(x)} dx \\
			& \lesssim (r^2 + \epsilon^2_{\max})\frac{r^2}{\epsilon^2_{\max}} \exp(-t^2/4) \\
			& \overset{\eqref{ineq:range-r-t}  }\lesssim r^2 \exp(-t^2/4) \lesssim \frac{\epsilon^2_{\max}}{\log(n\epsilon_{\max}^2)} \frac{1}{n \epsilon^2_{\max} } \lesssim  1/n. 
		\end{split}
	\end{equation*} In summary, we have proved $\chi^2(P_{\epsilon_{\max}, r, Q_0}, P_{\epsilon, 0, Q_1}) \lesssim 1/n$ in Case 2.
	
	\item (Case 3: $r \lesssim  \frac{\epsilon_{\max}}{\sqrt{\log\left(\frac{\epsilon_{\max}}{\epsilon}+e\right)}}$) First, if $\epsilon_{\max} \leq C \epsilon$ for some large constant $C \geq 1$, then $r \lesssim  \frac{\epsilon_{\max}}{\sqrt{\log\left(\frac{\epsilon_{\max}}{\epsilon}+e\right)}}$ is the same as $r \lesssim  \epsilon_{\max}$. In addition, 
	\begin{equation*}
		\begin{split}
			\inf_{Q_0, Q_1} \TV(P^{\otimes n}_{\epsilon_{\max}, r, Q_0}, P^{\otimes n}_{\epsilon, 0, Q_1}) & \leq \inf_{Q_0, Q_1} \TV(P^{\otimes n}_{\epsilon_{\max}, r, Q_0}, P^{\otimes n}_{\epsilon_{\max}/C, 0, Q_1}) \\
			& \leq \inf_{Q_0, Q_1} \TV(P^{\otimes n}_{\epsilon_{\max}/C, r, Q_0}, P^{\otimes n}_{\epsilon_{\max}/C, 0, Q_1}).
		\end{split}
	\end{equation*} By the proof of \cite[Theorem 2.2]{chen2018robust}, we know that when $r \lesssim \epsilon_{\max}/C$, then $\inf_{Q_0, Q_1} \TV(P^{\otimes n}_{\epsilon_{\max}/C, r, Q_0}, P^{\otimes n}_{\epsilon_{\max}/C, 0, Q_1}) = 0$. So the claim is proved.  

Now, let us focus on the setting $\epsilon_{\max} \geq C \epsilon$. In this case, we have 
$\epsilon_{\max}/\epsilon \gtrsim 1$ and
\begin{equation} \label{ineq:range-r-case2}
	r \lesssim  \frac{\epsilon_{\max}}{\sqrt{\log(\epsilon_{\max}/ \epsilon + e)}} \lesssim \frac{\epsilon_{\max}}{\sqrt{\log(\epsilon_{\max}/ \epsilon)}}  \lesssim  \epsilon_{\max} \lesssim 1. 
\end{equation} In this case, we want to construct $Q_0$ and $Q_1$ to exactly match the density of $P_{\epsilon_{\max}, r, Q_0}$ and $P_{\epsilon, 0, Q_1}$. Define the density of $Q_1$ as 
\begin{equation} \label{eq:q1-density-eps-tozero}
	\begin{split}
		q_1(x) = \left\{ \begin{array}{ll}
			\zeta\frac{(1 - \epsilon_{\max}) \phi(x-r) - (1 - \epsilon) \phi(x)}{\epsilon}, & \textnormal{ if } (1 - \epsilon_{\max}) \phi(x-r) > (1 - \epsilon) \phi(x),\\
			0, & \textnormal{ otherwise },
		\end{array} \right.
	\end{split}
\end{equation} where $\zeta$ satisfies that $\zeta \int_{\{x: (1 - \epsilon_{\max}) \phi(x-r) > (1 - \epsilon) \phi(x)\}} \frac{(1 - \epsilon_{\max}) \phi(x-r) - (1 - \epsilon) \phi(x)}{\epsilon} dx= 1$. We note that 
\begin{equation} \label{eq:integration-range}
	\begin{split}
		(1 - \epsilon_{\max}) \phi(x-r) > (1 - \epsilon) \phi(x) &\Longleftrightarrow \exp(xr - r^2/2 ) > \frac{1 - \epsilon}{1 - \epsilon_{\max}} \\
		& \Longleftrightarrow x > \frac{\log \left( \frac{1 - \epsilon}{1 - \epsilon_{\max}}\right) }{r} + \frac{r}{2}.
	\end{split}
\end{equation} So by construction, $Q_1$ is a valid distribution. Note again that this construction of $Q_1$ is slightly different from the one in the proof of Lemma \ref{prop:test-lower-bound-gaussian} and this new construction is better when $\epsilon_{\max} \to 0$. In addition, we define the density of $Q_0$ as
\begin{equation*}
	q_0(x) = \frac{(1-\epsilon) \phi(x) + \epsilon q_1(x) - (1 - \epsilon_{\max} )\phi(x - r) }{ \epsilon_{\max} } \, \textnormal{ for all } x \in \bbR.
\end{equation*} The intuition is that given $Q_0$ is a valid distribution, the choice of $Q_0$ above can exactly match $P_{\epsilon_{\max}, r, Q_0 }$ and $P_{\epsilon, 0, Q_1 }$. By construction, we have $\int q_0(x) dx = 1$ and if $\zeta \geq 1$, then $q_0(x) \geq 0$ for all $x \in \bbR$. Thus, if $\zeta \geq 1$, then $Q_0$ is a valid distribution.

Next, we are going to show that when $r \lesssim \frac{\epsilon_{\max}}{\sqrt{\log(\epsilon_{\max}/ \epsilon)}}$, then $\zeta \geq 1$ and the above constructed $Q_0$ and $Q_1$ are valid distributions and $\TV(P^{\otimes n}_{\epsilon_{\max}, r, Q_0}, P^{\otimes n}_{\epsilon, 0, Q_1}) = 0$, which yields the desired result. Since $\zeta \int_{\{x: (1 - \epsilon_{\max}) \phi(x-r) > (1 - \epsilon) \phi(x)\}} \frac{(1 - \epsilon_{\max}) \phi(x-r) - (1 - \epsilon) \phi(x)}{\epsilon} dx= 1$, to show $\zeta \geq 1$, it is equivalent to show that
\begin{equation*}
	 \int_{\{x: (1 - \epsilon_{\max}) \phi(x-r) > (1 - \epsilon) \phi(x)\}} \frac{(1 - \epsilon_{\max}) \phi(x-r) - (1 - \epsilon) \phi(x)}{\epsilon} dx \leq 1
\end{equation*} when $r \lesssim \frac{\epsilon_{\max}}{\sqrt{\log(\epsilon_{\max}/ \epsilon)}}$. Note
\begin{equation} \label{ineq:exact-match-verification}
	\begin{split}
		 &\int_{\{x: (1 - \epsilon_{\max}) \phi(x-r) > (1 - \epsilon) \phi(x)\}} \frac{(1 - \epsilon_{\max}) \phi(x-r) - (1 - \epsilon) \phi(x)}{\epsilon} dx \\
		 &\leq   \int_{\{x: (1 - \epsilon_{\max}) \phi(x-r) > (1 - \epsilon) \phi(x)\}} \frac{(1 - \epsilon_{\max}) (\phi(x-r) - \phi(x))}{\epsilon} dx \\
		 & \overset{\eqref{eq:integration-range}}= \int_{x \geq \frac{\log \left( \frac{1 - \epsilon}{1 - \epsilon_{\max}}\right) }{r} + \frac{r}{2}} \frac{(1 - \epsilon_{\max}) (\phi(x-r) - \phi(x))}{\epsilon} dx\\
		 & \leq \int_{x \geq \frac{\log \left( \frac{1 - \epsilon}{1 - \epsilon_{\max}}\right) }{r} } \frac{(1 - \epsilon_{\max}) (\phi(x-r) - \phi(x))}{\epsilon} dx.
	\end{split}
\end{equation} Notice that when for any $x \geq \frac{\log \left( \frac{1 - \epsilon}{1 - \epsilon_{\max}}\right) }{r} $, since $\epsilon_{\max} \geq C \epsilon$, 
\begin{equation} \label{ineq:x-range-bound}
	x \geq \frac{\log \left( \frac{1 - \epsilon_{\max} /C}{1 - \epsilon_{\max}}\right) }{r} \gtrsim \frac{\epsilon_{\max}}{r} \overset{ \eqref{ineq:range-r-case2} } \geq  \sqrt{8\log(\epsilon_{\max}/ \epsilon)} \gtrsim r.
\end{equation} So by Taylor expansion, we have $\phi(x-r) - \phi(x) \lesssim r x \phi(x/2)$. Plug it into \eqref{ineq:exact-match-verification}, we have 
\begin{equation*}
	\begin{split}
&\int_{x \geq \frac{\log \left( \frac{1 - \epsilon}{1 - \epsilon_{\max}}\right) }{r} } \frac{(1 - \epsilon_{\max}) (\phi(x-r) - \phi(x))}{\epsilon} dx \\
& \lesssim \int_{x \geq \sqrt{8\log(\epsilon_{\max}/ \epsilon)}} \frac{(1 - \epsilon_{\max}) rx \phi(x/2)}{\epsilon} dx \\
& \lesssim \int_{x \geq \sqrt{8\log(\epsilon_{\max}/ \epsilon)}} \frac{(1 - \epsilon_{\max}) rx \exp(-x^2/8) }{\epsilon} dx \\
& = \frac{(1 - \epsilon_{\max}) r }{\epsilon} \left(-4 \exp(-x^2/8) \right)\Big|^{\infty}_{\sqrt{8\log(\epsilon_{\max}/ \epsilon)} } \\
& =\frac{4(1 - \epsilon_{\max}) r }{\epsilon} \frac{\epsilon}{\epsilon_{\max}} \\
& \leq 1
	\end{split}
\end{equation*} where the last inequality holds as $r \lesssim \frac{\epsilon_{\max}}{\sqrt{\log(\epsilon_{\max}/ \epsilon)}}$ and $\frac{\epsilon_{\max}}{\sqrt{\log(\epsilon_{\max}/ \epsilon)}} \gtrsim 1$.
\end{itemize}
This finishes the proof for both cases.

%%%%%%%%%%%%%%%%%%%%%%%%%%%%%%%%%%%%%%%%%%%%%%%%%
\subsection{Proof of Lemma \ref{th:estimation-guarantee-mod}}
The proof of this lemma is almost the same as that of Lemma \ref{th:estimation-guarantee}. We will briefly sketch the proof here. By the DKW inequality (see Lemma \ref{lm:DKW}), we have with probability at least $1 - \alpha$, the following event holds:
	\begin{equation*}
		(A) = \{ \textnormal{for all } x \in \bbR, |F_n(x) - F_{\epsilon, \theta, Q}(x)| \leq \sqrt{ \log (2/\alpha)/(2n) } \},
	\end{equation*} where $F_{\epsilon, \theta, Q }(\cdot)$ denotes the CDF of $P_{\epsilon, \theta, Q}$. Then following the same analysis as in \eqref{eq:right-estimator-reduction} and \eqref{eq:left-estimator-reduction}, to show $\htheta_{R}(t)  - \theta \geq -8/t$ and $\htheta_{L}(t)  - \theta \leq 8/t$ hold simultaneously for all $t\in\left[4, \Phi^{-1}\left(1 - \sqrt{ \log (2/\alpha)/(2n) } \right)\right]$ and uniformly over all $\epsilon\in[0,0.99]$, all $\theta\in\mathbb{R}$, and all $Q$ given $(A)$ happens, it is enough to show: for $t_{\min} = 4$ and $t_{\max} = \Phi^{-1} ( 1- \sqrt{ \log (2/\alpha)/(2n) } )$, the following condition is satisfied: 
\begin{equation*}
	\begin{split}
			& 2( 1 - \Phi(t) ) < (1 - \epsilon_{\max} )P_{\theta}(X \geq \theta +t -8/t) - \sqrt{ \log (2/\alpha)/(2n) }, \forall t \in [t_{\min}, t_{\max}] \\
			\Longleftrightarrow & 2 P_{0}(X \geq t) < (1 - \epsilon_{\max} )P_{0}(X \geq t -8/t) - \sqrt{ \log (2/\alpha)/(2n) }, \forall t \in [t_{\min}, t_{\max}] \\
			\Longleftarrow & \left\{ \begin{array}{l}
				P_{0}(X \geq t) \geq \sqrt{ \log (2/\alpha)/(2n) }, \forall t \in [t_{\min}, t_{\max}] \\
				(1 - \epsilon_{\max} )P_{0}(X \geq t -8/t) > 3 P_{0}(X \geq t),\forall t \in [t_{\min}, t_{\max}]
			\end{array} \right. \\
			\overset{ \eqref{ineq:density-ratio-to-tail-prob} }\Longleftarrow & \left\{ \begin{array}{l}
				t_{\max} \leq \Phi^{-1} ( 1- \sqrt{ \log (2/\alpha)/(2n) } ), \\
				(1 - \epsilon_{\max} ) \phi(x - 8/t) > 3 \phi(x),\forall x \geq t, \forall t \in [t_{\min}, t_{\max}]
			\end{array} \right. \\
			\Longleftrightarrow & \left\{ \begin{array}{l}
				t_{\max} \leq \Phi^{-1} ( 1- \sqrt{ \log (2/\alpha)/(2n) } ), \\
				\exp( 8 - 32/t^2_{\min} ) > \frac{3}{1- \epsilon_{\max} }.
			\end{array} \right.
	\end{split}
\end{equation*} It is easy to check that the above two conditions are satisfied.

%%%%%%%%%%%%%%%%%%%%%%%%%%%%%%%
\subsection{Proof of Proposition \ref{th:ARCI-large-eps1}}
%%%%%%%%%%%%%%%%%%%%%%%%%%%%%%%%%
We begin by showing the guarantee of $\htheta_{\textnormal{median}}$ under the following event 
 	\begin{equation*}
		(A) = \{ \textnormal{for all } x \in \bbR, |F_n(x) - F_{\epsilon, \theta, Q}(x)| \leq \sqrt{ \log (2/\alpha)/(2n) } \},
	\end{equation*} where $F_{\epsilon, \theta, Q }(\cdot)$ denotes the CDF of $P_{\epsilon, \theta, Q}$. Notice that (A) happens with probability at least $1-\alpha$ by the DKW inequality (see Lemma \ref{lm:DKW}), and it is the same event that guarantees the coverage of \eqref{eq:just-coverage} as we have shown in Lemma \ref{th:estimation-guarantee-mod}. 
	
	Notice that $\htheta_{\textnormal{median}} = F_n^{-1}(1/2) = X_{( \lceil n/2 \rceil )}$ and $F_\theta(\theta) = 1/2$, where $F_\theta$ is the CDF of $P_\theta$. Thus
	\begin{equation} \label{eq:median-guarantee}
		\begin{split}
			 & |F_\theta(\theta) - F_\theta(\htheta_{\textnormal{median}}) | \\
			= &  |F_\theta(\theta) - F_n(\htheta_{\textnormal{median}}) + F_n(\htheta_{\textnormal{median}})  - F_{\epsilon, \theta, Q}(\htheta_{\textnormal{median}}) + F_{\epsilon, \theta, Q}(\htheta_{\textnormal{median}}) - F_{\theta}(\htheta_{\textnormal{median}}) | \\
			\leq & |F_\theta(\theta) - F_n(\htheta_{\textnormal{median}}) | + |F_n(\htheta_{\textnormal{median}})  - F_{\epsilon, \theta, Q}(\htheta_{\textnormal{median}})| + |F_{\epsilon, \theta, Q}(\htheta_{\textnormal{median}}) - F_{\theta}(\htheta_{\textnormal{median}})| \\
			\overset{(a)}\leq & \frac{1}{n} + \sqrt{ \log (2/\alpha)/(2n) } + \epsilon,
		\end{split}
	\end{equation} where in (a) we use event (A) property and the model assumption $P_{\epsilon, \theta, Q} = (1 - \epsilon )P_\theta + \epsilon Q$.
	
	Then when $n/\log(1/\alpha)$ is large enough, we have  $\frac{1}{n} + \sqrt{ \log (2/\alpha)/(2n) } + \epsilon \leq 0.491$, so by Taylor expansion
	\begin{equation} \label{ineq:taylor-expansion}
		|F_\theta(\theta) - F_\theta(\htheta_{\textnormal{median}}) | = |\htheta_{\textnormal{median}} - \theta| \phi(\xi),
	\end{equation} where $\xi \in [0, \Phi^{-1}(0.991)]$. A combination of \eqref{eq:median-guarantee} and \eqref{ineq:taylor-expansion} yields that 
	\begin{equation*}
		|\htheta_{\textnormal{median}} - \theta| \leq \frac{1}{\phi( \Phi^{-1}(0.991) )} \left( \frac{1}{n} + \sqrt{ \log (2/\alpha)/(2n) } + \epsilon \right) \leq R
	\end{equation*} where $R$ is chosen so that the last inequality holds even when $\epsilon = 0.49$. By combining with Lemma \ref{th:estimation-guarantee-mod}, this shows the coverage of $\widehat{\CI}$ given (A) happens as both intervals
 	\begin{equation*}
 		\bigcap_{4 \leq t \leq  \Phi^{-1}\left(1 - \sqrt{ \log (2/\alpha)/(2n) } \right) }\left[\wh{\theta}_L(t)-\frac{8}{t},\wh{\theta}_R(t)+\frac{8}{t}\right]\quad \textnormal{and}\quad \left[\wh{\theta}_{\rm median}-R,\wh{\theta}_{\rm median}+R\right]
 	\end{equation*} will cover $\theta$.

	Next, we move on to the length guarantee. Suppose $\epsilon_0$ is a sufficiently small constant so that for any $\epsilon \leq \epsilon_0$, we have $t_{\epsilon}$ in \eqref{def:t-constant} is greater than $4$. We divide the rest of the proof into two cases.
	\vskip.2cm
	{\bf (Case 1: $\epsilon_0\leq \epsilon \leq \epsilon_{\max}$)} In this case,
	\begin{equation*}
		|\widehat{\CI}| \leq \Big|\left[\wh{\theta}_{\rm median}-R,\wh{\theta}_{\rm median}+R\right]\Big| = 2R \asymp \frac{1}{\sqrt{\log n}} + \frac{1}{\sqrt{\log(1/ \epsilon )}},
	\end{equation*}  where the second equivalence is because both sides are of constant order.
	
	{\bf (Case 2: $0\leq \epsilon \leq \epsilon_0$)} In this case, given (A) happens, by a similar calculation as in \eqref{ineq:length-guarantee}, we have $|\widehat{\CI}| \leq 16/t_{\epsilon}$ by Lemmas \ref{th:estimation-length-guarantee} and \ref{th:estimation-guarantee-mod}. Finally, we have $1/t_{\epsilon} \lesssim  \frac{1}{\sqrt{\log(1/\epsilon)}} + \frac{1}{\sqrt{\log( n/(\log(1/\alpha)) )}}$ for all $\epsilon \in [0, \epsilon_0]$ by Lemma \ref{lm:t-epsilon-bound}. This finishes the proof.

 %%%%%%%%%%%%%%%%%%%%%%%%%%%%%%%%%%%%%
\subsection{Proof of Proposition \ref{prop:list-d-m-e-c} } 
Let us introduce a good event $(B)$,
	\begin{equation*}
		\begin{split}
			(B) = \left\{  \frac{1}{n}\sum_{i=1}^n\indi\left(|X_i-\theta| \leq 2\right)\geq 0.005 \right\}. 
		\end{split}
	\end{equation*}
	It is clear that given $(B)$ happens, we have $\theta \in \cH$. Thus, one of the points in $\mathcal{L}$ must be of at most $4$ distance to $\theta$ as otherwise it contradicts the maximum size property of $\mathcal{L}$. In addition, the size of $\mathcal{L}$ can also be controlled deterministically by the construction of $\mathcal{L}$. In particular, for each $\htheta_i \in \mathcal{L}$, by definition 
	\begin{equation*}
		 \sum_{j =1}^n \indi(|X_j - \htheta_i| \leq 2) \geq 0.005 n.
	\end{equation*}
	Thus
	\begin{equation*}
		n \geq \sum_{\htheta_i \in \mathcal{L}} \sum_{j =1}^n \indi(|X_j - \htheta_i| \leq 2) \geq |\mathcal{L}| \frac{1}{200} n.
	\end{equation*} This yields an upper bound $|\mathcal{L}| \leq 200$ as desired. 
		
	Thus, to finish the proof, we just need to show $(B)$ happens with probability at least $1-\alpha$. We define the following event 
 	\begin{equation*}
		(A) = \{ \textnormal{for all } x \in \bbR, |F_n(x) - F_{\epsilon, \theta, Q}(x)| \leq \sqrt{ \log (2/\alpha)/(2n) } \},
	\end{equation*} where $F_{\epsilon, \theta, Q }(\cdot)$ denotes the CDF of $P_{\epsilon, \theta, Q}$. Notice that (A) happens with probability at least $1-\alpha$ by the DKW inequality (see Lemma \ref{lm:DKW}).
	
	Then, given (A) happens,
	\begin{equation*}
	\begin{split}
		(B) &\Longleftrightarrow \left\{ \frac{1}{n} \sum_{j =1}^n \indi(X_j - \theta \leq 2) - \frac{1}{n} \sum_{j =1}^n \indi(X_j - \theta < -2) \geq 0.005  \right\} \\
		& \Longleftarrow \left\{ F_n(\theta + 2) - F_n(\theta - 2) \geq 0.005  \right\} \\
		& \overset{(A)}\Longleftarrow \left\{ F_{\epsilon, \theta, Q }(\theta + 2) - F_{\epsilon, \theta, Q }(\theta - 2) - \sqrt{ 2\log (2/\alpha)/n } \geq 0.005  \right\} \\
		& \Longleftarrow \left\{ (1-0.99) \bbP(N(0,1) \in [-2,2]) - \sqrt{ 2\log (2/\alpha)/n } \geq 0.005 \right\},
	\end{split}
	\end{equation*} notice that the last condition holds when $n/\log(1/\alpha)$ is large enough. This shows that the event (B) happens whenever (A) happens, thus it happens with probability at least $1-\alpha$. This finishes the proof of this proposition.
	
%%%%%%%%%%%%%%%%%%%%%%%%%%%%%%%%
\subsection{Proof of Theorem \ref{th:ARCI-list-decodable} }
%%%%%%%%%%%%%%%%%%%%%%%%%%%%%%%
First, it is clear that the confidence set in \eqref{eq:arcs-final} can be written as the union of intervals, and the number of unions is at most $C_1$ by Proposition \ref{prop:list-d-m-e-c}.

Now, let us move to the coverage guarantee. Recall that the event (A) is denoted as 
 \begin{equation*}
		(A) = \{ \textnormal{for all } x \in \bbR, |F_n(x) - F_{\epsilon, \theta, Q}(x)| \leq \sqrt{ \log (2/\alpha)/(2n) } \},
	\end{equation*}  where $F_{\epsilon, \theta, Q }(\cdot)$ denotes the CDF of $P_{\epsilon, \theta, Q}$ and (A) happens with probability at least $1-\alpha$ by the DKW inequality (see Lemma \ref{lm:DKW}).
 
 Lemma \ref{th:estimation-guarantee-mod} and Proposition \ref{prop:list-d-m-e-c} show that when (A) happens, both intervals
 	\begin{equation*}
 		\bigcap_{4 \leq t \leq  \Phi^{-1}\left(1 - \sqrt{ \log (2/\alpha)/(2n) } \right) }\left[\wh{\theta}_L(t)-\frac{8}{t},\wh{\theta}_R(t)+\frac{8}{t}\right]\quad \textnormal{and}\quad \bigcup_{\wh{\theta}\in\mathcal{L}}\left[\wh{\theta}-4,\wh{\theta}+4\right]
 	\end{equation*} will cover $\theta$. This shows the coverage property of the proposed $\widehat{\CS}$.
 	
	Next, we move on to the length guarantee. Suppose $\epsilon_0$ is a sufficiently small constant so that for any $\epsilon \leq \epsilon_0$, we have $t_{\epsilon}$ in \eqref{def:t-constant} is greater than $4$. We divide the rest of the proof into two cases.
	\vskip.2cm
	{\bf (Case 1: $\epsilon_0\leq \epsilon \leq \epsilon_{\max}$)} In this case,
	\begin{equation*}
		|\widehat{\CI}| \leq \Big|\bigcup_{\wh{\theta}\in\mathcal{L}}\left[\wh{\theta}-4,\wh{\theta}+4\right]\Big| \leq 8\#\mathcal{L}  \asymp \frac{1}{\sqrt{\log n}} + \frac{1}{\sqrt{\log(1/ \epsilon )}},
	\end{equation*}  where the second equivalence is because both sides are of constant order.
	
	{\bf (Case 2: $0\leq \epsilon \leq \epsilon_0$)} In this case, given (A) happens, by a similar calculation as in \eqref{ineq:length-guarantee}, we have $|\widehat{\CI}| \leq 16/t_{\epsilon}$ by Lemmas \ref{th:estimation-length-guarantee} and \ref{th:estimation-guarantee-mod}. Finally, we have $1/t_{\epsilon} \lesssim  \frac{1}{\sqrt{\log(1/\epsilon)}} + \frac{1}{\sqrt{\log( n/(\log(1/\alpha)) )}}$ for all $\epsilon \in [0,\epsilon_0 ]$ by Lemma \ref{lm:t-epsilon-bound}. This finishes the proof.

%%%%%%%%%%%%%%%%%%%%%%%%%%%%%%%%
\subsection{Proof of Theorem \ref{th:lower-bound-Gaussian-unknown-variance}}
%%%%%%%%%%%%%%%%%%%%%%%%%%%%%%%
Let us first present a useful lemma in identifying a hard testing instance, and its proof will be given in the following subsection.
\begin{Lemma} \label{prop:test-lower-bound-gaussian-unknown-variance}
For any $\sigma > 0$ and any $\epsilon \in [0, \epsilon_{\max}]$ with $\epsilon_{\max} = 0.05$, if
	\begin{equation} \label{eq:Gaussian-lower-bound-condition-unknown-variance}
		r \leq  \sigma \sqrt{ 2(1 - 0.99^2) \log\left( \frac{0.99}{1- \epsilon_{\max} } \right) },
	\end{equation} then there exists $Q_0$ and $Q_1$ such that $P_{\epsilon_{\max}, \theta, Q_0, 0.99 \sigma}   = P_{\epsilon, \theta - r, Q_1, \sigma}$.
\end{Lemma}	
	Now, we prove Theorem \ref{th:lower-bound-Gaussian-unknown-variance} using Lemma \ref{prop:test-lower-bound-gaussian-unknown-variance} with a contradiction argument. Let $c \leq \sqrt{ 2(1 - 0.99^2) \log\left( \frac{0.99}{1- \epsilon_{\max} } \right) }$ and suppose $r_{\alpha}\left(\epsilon,\sigma,[0,0.05],\left[\frac{\sigma}{2},\sigma\right]\right)> c\sigma$ does not hold. Then there exists a $\epsilon \in [0, \epsilon_{\max}]$ and an ARCI $\widehat{\CI}$ such that the following conditions hold
\begin{equation} \label{ineq:contradiction-condition-unknown-variance}
	\widehat{\CI} \in \cI_{\alpha}([0,0.05],[\sigma/2, \sigma ]) \quad \textnormal{ and } \quad \sup_{\theta, Q} \bbP _{\epsilon, \theta, Q, \sigma} \left( |\widehat{\CI}| \geq r \right) \leq \alpha,
\end{equation} where $r = c \sigma$. Then let us consider the testing problem 
\begin{equation}\label{eq:test-def-unknown-variance}
\begin{split}
		H_{0}: P \in \left\{ P_{\epsilon_{\max}, r, Q, 0.99\sigma}:  Q \right\} \quad  \textnormal { v.s. } \quad 
		H_{1}: P \in \left\{ P_{\epsilon, 0, Q, \sigma}: Q \right\}.
	\end{split}
\end{equation} 
We can construct a test for solving the above testing problem based on the ARCI $\widehat{\CI}$:
	\begin{equation*}
		\begin{split}
			\phi( \{ X_i \}_{i=1}^n ) := \indi \left(  r  \notin  \widehat{\CI} \right)
		\end{split}
	\end{equation*} with the following Type-1 error guarantee,
		\begin{equation} \label{eq:CI-to-test-error-unknown-variance}
		\begin{split}
		\sup_{Q} P_{\epsilon_{\max}, r , Q, 0.99\sigma} \left( \phi = 1  \right)  = \sup_{Q} P_{\epsilon_{\max}, r, Q, 0.99\sigma } \left( r \notin  \widehat{\CI} \right)  \leq \alpha,
		\end{split}
	\end{equation} where the inequality is because $\widehat{\CI} \in \cI_{\alpha}([0,0.05],[\sigma/2, \sigma ]) $ and the Type-2 error guarantee,
\begin{equation*} %\label{ineq:type-II-error}
		\begin{split}
			& \sup_{Q} P_{ \epsilon, 0, Q, \sigma} \left( \phi = 0  \right) \\
			&= \sup_{Q} P_{ \epsilon, 0, Q, \sigma} \left( r   \in \widehat{\CI}   \right) \\
			& \leq \sup_{Q} P_{ \epsilon, 0, Q, \sigma} \left( r \in  \widehat{\CI}, 0 \in \widehat{\CI} \right) + \sup_{Q} P_{ \epsilon, 0, Q, \sigma} \left( 0 \notin  \widehat{\CI} \right) \\
			& \overset{(a)}\leq \sup_{Q} P_{ \epsilon, 0, Q, \sigma} \left( |\widehat{\CI}| \geq r \right) + \sup_{Q} P_{\epsilon, 0, Q,\sigma} \left( 0 \notin  \widehat{\CI} \right)  \\
			& \overset{(b)}\leq \alpha + \sup_{Q } P_{\epsilon, 0 , Q, \sigma} \left( |\widehat{\CI}| \geq r\right)\\
			& \overset{\eqref{ineq:contradiction-condition-unknown-variance}}\leq 2 \alpha,
		\end{split}
	\end{equation*}
where in (a), we use the fact that $\widehat{\CI}$ is an interval, and (b) is again by the coverage property of $\widehat{\CI}$. This yields
	\begin{equation} \label{ineq:typeI+II-unknown-variance}
		\sup_{ Q } P_{\epsilon_{\max}, r , Q, 0.99 \sigma } \left( \phi = 1  \right)+ \sup_{Q } P_{ \epsilon, 0, Q, \sigma  } \left( \phi = 0  \right) \leq  3 \alpha.
	\end{equation}

On the other hand, by Lemma \ref{prop:test-lower-bound-gaussian-unknown-variance} and the Neyman-Person Lemma, when $r= c \sigma$, we have
	\begin{equation} \label{ineq:neyman-person-unknown-variance}
		\begin{split}
		&\inf_{ \phi \in \{0,1 \} } \left\{ \sup_{Q } P_{\epsilon_{\max}, r, Q, 0.99 \sigma } \left( \phi = 1  \right) + \sup_{Q } P_{ \epsilon , 0, Q, \sigma} \left( \phi = 0  \right) \right\}\\
			&\geq 1 - \inf_{Q_0, Q_1} \TV\left(P^{\otimes n}_{\epsilon_{\max},r, Q_0, 0.99 \sigma }, P^{\otimes n}_{\epsilon, 0 , Q_1, \sigma} \right) \\
			& =1,
		\end{split}
	\end{equation} which contradicts with \eqref{ineq:typeI+II-unknown-variance} since $\alpha< 1/4$. This finishes the proof of this theorem.
	 
%%%%%%%%%%%%%%%%%%%%%%%%%%%%%%
\subsubsection{Proof of Lemma \ref{prop:test-lower-bound-gaussian-unknown-variance}}
%%%%%%%%%%%%%%%%%%%%%%%%%%%%%%	
	Without loss of generality, we take $\theta = r$, so we just need to find $Q_0$ and $Q_1$ such that $P_{\epsilon_{\max}, r, Q_0, 0.99 \sigma}   = P_{\epsilon, 0, Q_1, \sigma}$. We let $Q_1 =  N(0, \sigma^2)$ and take $Q_0$ having density given as follows
	\begin{equation*}
		q_0(x) := \frac{f_{0, \sigma}(x) - (1 - \epsilon_{\max}) f_{r, 0.99 \sigma}(x) }{\epsilon_{\max}}, \forall x \in \bbR,
	\end{equation*} where $f_{\theta, \sigma}(x)$ is the density of $N(\theta, \sigma^2)$ at $x$. It is clear that $\int q_0(x) dx = 1$. As long as we show $q_0(x) \geq 0$ for all $x \in \bbR$ under the condition \eqref{eq:Gaussian-lower-bound-condition-unknown-variance} regarding $r$, we are done.
	\begin{equation} \label{eq:unknown-variance-exact-match}
		\begin{split}
			q_0(x) \geq 0, \forall x \in \bbR &\Longleftrightarrow f_{0, \sigma}(x) \geq ( 1- \epsilon_{\max} ) f_{r, 0.99 \sigma}(x), \forall x \in \bbR \\
			& \Longleftrightarrow \frac{x^2}{2} - \frac{(x - r)^2}{2 \times 0.99^2} \leq \sigma^2\log\left( \frac{0.99}{ 1- \epsilon_{\max} } \right), \forall x \in \bbR\\
			& \Longleftrightarrow \left(\frac{1}{2} - \frac{1}{2 \times 0.99^2} \right)x^2 + \frac{rx}{0.99^2} - \frac{r^2}{2\times 0.99^2} \leq \sigma^2\log\left( \frac{0.99}{ 1- \epsilon_{\max} } \right), \forall x \in \bbR \\
			& \overset{(a)}\Longleftrightarrow \frac{r^2}{2(1 - 0.99^2) 0.99^2} -  \frac{r^2}{2 \times 0.99^2} \leq \sigma^2 \log\left( \frac{0.99}{ 1- \epsilon_{\max} } \right) \\
			& \Longleftrightarrow \frac{r^2}{2(1 - 0.99^2) } \leq \sigma^2 \log\left( \frac{0.99}{ 1- \epsilon_{\max} } \right)
		\end{split}
	\end{equation} where in (a), we directly compute the maximum of the quadratic function $\left(\frac{1}{2} - \frac{1}{2 \times 0.99^2} \right)x^2 + \frac{rx}{0.99^2} - \frac{r^2}{2\times 0.99^2}$ in $x$. We note that the right-hand side of \eqref{eq:unknown-variance-exact-match} is satisfied by the condition on $r$ posed in \eqref{eq:Gaussian-lower-bound-condition-unknown-variance}.

%%%%%%%%%%%%%%%%%%%%%%%%%%%%%%%%%%%%%%%%%%
\section{Optimal Estimation Error Over Different Classes of Distributions} \label{sec:optimal-estimation-error}
%%%%%%%%%%%%%%%%%%%%%%%%%%%%%%%%%%%%%%%%%%
In this section, we provide the minimax optimal estimation error rates for estimating $\theta$ under the four classes of distributions we considered in Section \ref{sec:examples-add}. First, we have the following upper bound.
\begin{Proposition} \label{prop:minimax-rate-others-upper-bound}
	Suppose $P_\theta$ is either the $t$ distribution, generalized Gaussian, mollifier, or Bates distribution with location parameter $\theta$. Given $n$ i.i.d. samples drawn from model \eqref{def:model} or the TV contamination model in Section \ref{sec:TV-contamination}. If $n/\log(1/\alpha)$ is sufficiently large and $\epsilon$ is smaller than a sufficiently small constant, then with probability at least $1 - \alpha$, the median achieves error $|\htheta_{\textnormal{median}} - \theta| \leq C( \epsilon +  \sqrt{ \log(1/\alpha)/n })$ for some universal constant $C > 0$ depending on the distribution class only. 
\end{Proposition}

\begin{proof}
	This can be derived from Proposition 1.14 in \cite{diakonikolas2023algorithmic}. From their result, we have for any location family $P_\theta$, with probability $1-\alpha$, $\htheta_{\textnormal{median}}$ is between $(1/2 - \epsilon - C \sqrt{ \log(1/\alpha)/n } )$ quantile of $P_\theta$ and $(1/2 + \epsilon + C \sqrt{ \log(1/\alpha)/n } )$ quantile of $P_\theta$ for some $C > 0$. Since for all four classes of distributions, the density in $[\theta - c_1, \theta + c_1]$ for some small constant $c_1 > 0$ lies between $[c_2, C_2]$ for some constants $0 < c_2 < C_2$. Thus, $(1/2 + \epsilon + C \sqrt{ \log(1/\alpha)/n } )$ quantile of $P_\theta$ is of at most $\theta + O(\epsilon + \sqrt{ \log(1/\alpha)/n } )$ and similar for $(1/2 - \epsilon - C \sqrt{ \log(1/\alpha)/n } )$ quantile of $P_\theta$ due to symmetry.  
\end{proof}

Now, we move on to the lower bound.
\begin{Proposition} \label{prop:minimax-rate-others-lower-bound}
	Suppose $P_\theta$ is either the $t$ distribution, generalized Gaussian with shape parameter $\beta > 0.5$, mollifier, or Bates distribution with positive integer $k \geq 3$ with location parameter $\theta$. Given $n$ i.i.d. samples drawn from \eqref{def:model} with $\epsilon < 1/2$. Then there exists $c, C > 0$ such that $\inf_{\htheta} \sup_{\theta \in \bbR, Q} P_{\epsilon, \theta, Q }\left( |\htheta - \theta| \geq C\left( \frac{1}{\sqrt{n}} + \epsilon\right) \right) \geq c$. 
\end{Proposition}
\begin{proof}	First, we note that it is enough to prove the lower bound separately for $1/\sqrt{n}$ and $\epsilon$. We start with the $1/\sqrt{n}$ lower bound, and it can be proved in the setting when there is no contamination. The main tool we use is Le Cam's two-point method \citep{tsybakov2009introduction}. To show the $1/\sqrt{n}$ rate lower bound in the setting $\epsilon = 0$, it is enough to show
	\begin{equation} \label{ineq:hellinger-general-bound}
		\H^2(P_r, P_0 ) = O(r^2)
	\end{equation} for $0\leq r \leq  c$ with a sufficiently small $c > 0$, where $\H(P_1, P_2) = \sqrt{ \int (\sqrt{p_1(x)} - \sqrt{p_2(x)} )^2 d x } $ denotes the Hellinger distance between distributions $P_1$ and $P_2$. This is because, in general 
	\begin{equation} \label{ineq:testing-to-hellinger}
		\begin{split}
			\inf_{\htheta} \sup_{\theta } P_{\theta }\left( |\htheta - \theta| \geq r/2 \right) &\geq \inf_{\htheta} \sup_{\theta \in \{0,r\}} P_{\theta }\left( |\htheta - \theta| \geq r/2 \right) \geq 1 - \TV(P_r^{\otimes n}, P_0^{\otimes n} ) \\
			& =  \int \left(\prod_{i=1}^n f_r(x_i) \wedge \prod_{i=1}^n f_0(x_i) \right) dx_1 \ldots dx_n \\
			& \geq \frac{1}{2} \left(  \prod_{i=1}^n \int\sqrt{f_r(x_i) f_0(x_i)}   dx_i\right)^2 = \frac{1}{2} \left( 1- \frac{1}{2} H^2(P_r, P_0) \right)^{2n},
		\end{split}
	\end{equation} where $f_\theta(\cdot)$ denotes the density function of $P_\theta$, the second inequality is by the Neyman-Person Lemma and the third inequality is by Le Cam's inequality (see Lemma 2.3 in \cite{tsybakov2009introduction}). So if \eqref{ineq:hellinger-general-bound} holds, we can take $r$ to be $c/\sqrt{n}$ for some small constant, and this would yield the lower bound. Next, we will consider four distribution classes separately.

	{\bf (Generalized Gaussian with $\beta  > 0.5$)} First
	\begin{equation} \label{ineq:generalized-gaussian-hellinger}
		\begin{split}
			 & \H^2(P_r, P_0 ) \\
			& = \int_r^\infty (\sqrt{f(x-r)} - \sqrt{f(x)} )^2 d x  + \int_{-\infty}^0 (\sqrt{f(x-r)} - \sqrt{f(x)} )^2 d x \\
			& \quad + \int_0^r (\sqrt{f(x-r)} - \sqrt{f(x)} )^2 d x \\
			& = 2 \int_r^\infty (\sqrt{f(x-r)} - \sqrt{f(x)} )^2 d x + \int_0^r (\sqrt{f(x-r)} - \sqrt{f(x)} )^2 d x.
		\end{split}
	\end{equation} Next, we show both terms on the right-hand side are of \eqref{ineq:generalized-gaussian-hellinger} of order $O(r^2)$ when $\beta > 0.5$. 
	
	By Taylor expansion of $g(r) = \sqrt{f(x-r)}$, for the first term on the right-hand side of \eqref{ineq:generalized-gaussian-hellinger}, we have
	\begin{equation*}
		\begin{split}
			\int_r^\infty (\sqrt{f(x-r)} - \sqrt{f(x)} )^2 d x & \lesssim r^2 \int_{r}^\infty \sup_{\xi \in [0, r]} \exp(- |x - \xi|^\beta) |x - \xi|^{2(\beta - 1)}d x \\
			& \lesssim \left\{ \begin{array}{ll}
				 r^2 \int_r^\infty \exp( - (x - r)^\beta ) x^{2(\beta - 1)}d x, &  \textnormal{ if } \beta \geq 1,\\
				 r^2 \int_r^\infty \exp( - (x - r)^\beta ) (x - r)^{2(\beta - 1)}d x, & \textnormal{ if } 0.5<\beta < 1,
			\end{array} \right.  \\
			& \overset{(a)}\lesssim \left\{ \begin{array}{ll}
				 r^2, &  \textnormal{ if } \beta \geq 1,\\
				 r^2 \int_0^\infty \exp( - x^\beta ) x^{2(\beta - 1)} d x, & \textnormal{ if } 0.5<\beta < 1,
			\end{array} \right. \\
			& \lesssim \left\{ \begin{array}{ll}
				 r^2, &  \textnormal{ if } \beta \geq 1,\\
				 r^2 \left( \int_0^1 x^{2(\beta - 1)}d x + \int_1^\infty \exp( - x^\beta ) d x \right), & \textnormal{ if } 0.5<\beta < 1,
			\end{array} \right. \\
			& \overset{(b)}\lesssim r^2,
		\end{split}
	\end{equation*} where in (a) we use the fact that all the moments of generalized Gaussian distribution exist and are a polynomial function with nonnegative exponents of the location, scale, and shape parameters \citep{nadarajah2005generalized}; (b) is because $\beta > 0.5$.
	
	Now, let us bound the second term on the right-hand side of \eqref{ineq:generalized-gaussian-hellinger}.
	\begin{equation*}
		\begin{split}
			 & \int_0^r (\sqrt{f(x-r)} - \sqrt{f(x)} )^2 d x \\
			&= \int_0^r f(x)  \left(\frac{\sqrt{f(x-r)}}{\sqrt{f(x)}} - 1 \right)^2 d x \\
			& \overset{(a)}\lesssim  \int_0^r \left(\exp\left( - (r - x)^\beta/2 + x^\beta/2 \right) - 1 \right)^2 d x \\
			& \lesssim \int_0^{r/2} \left(\exp\left( - (r - x)^\beta/2 \right) - 1 \right)^2 d x + \int_{r/2}^r \left(\exp\left( x^\beta/2 \right) - 1 \right)^2 d x \\
			& \lesssim \int_0^{r/2} \left(\exp\left( - r^\beta/2 \right) - 1 \right)^2 d x + \int_{r/2}^r \left(\exp\left( r^\beta/2 \right) - 1 \right)^2 d x \\
			& \lesssim r \left(\exp\left( - r^\beta/2 \right) - 1 \right)^2 + r \left(\exp\left( r^\beta/2 \right) - 1 \right)^2 \\
			& \overset{(b)}\lesssim r^{1 + 2 \beta} \overset{(c)}\lesssim r^2
		\end{split}
	\end{equation*} where (a) is because $f(x)$ is of constant order for $x \in [0,r]$ and (b) is because when $|\delta| < c$ for a sufficiently small $c$, $1 + \delta \leq \exp(\delta) \leq 1 + C_1 \delta $ for some universal constant $C_1 > 0$; (c) is because $\beta \geq 0.5$. This finishes the proof for the generalized Gaussian with $\beta > 0.5$. 
	
	{\bf (Student's $t$ distribution)}  Fix any degree of freedom $\nu > 0$,
	\begin{equation*}
		\begin{split}
			\H^2(P_r, P_0) & = \int_{-\infty}^\infty (\sqrt{f(x-r)} - \sqrt{f(x)} )^2 d x \\
			& \overset{(a)}\lesssim r^2 \int_{-\infty}^{\infty} \sup_{\xi \in [0, r]} \frac{( 1 + (x - \xi)^2/\nu )^{-(\nu+3)}}{( 1 + (x - \xi)^2/\nu )^{-(\nu+1)/2}} (x - \xi)^2 dx\\
			& \lesssim r^2 \int_{-\infty}^{\infty} \sup_{\xi \in [0, r]} ( 1 + (x - \xi)^2/\nu )^{-(\nu+5)/2} (x - \xi)^2 dx
		\end{split}
	\end{equation*} where (a) is by the Taylor expansion of $g(r) = \sqrt{f(x-r)}$. Next,
	\begin{equation*}
		\begin{split}
		& \int_{-\infty}^{\infty} \sup_{\xi \in [0, r]} ( 1 + (x - \xi)^2/\nu )^{-(\nu+5)/2} (x - \xi)^2 dx \\
		\leq & 	\int_{-\infty}^{0}( 1 + x^2/\nu )^{-(\nu+5)/2} (x - r)^2 dx + \int_{0}^{r} r^2 dx + \int_{r}^{\infty} ( 1 + (x - r)^2/\nu )^{-(\nu+5)/2} x^2 dx\\
		=  & r^3 + 2\int_{r}^{\infty} ( 1 + (x - r)^2/\nu )^{-(\nu+5)/2} x^2 dx \\
		= & r^3 + 2\int_{0}^{\infty} ( 1 + x^2/\nu )^{-(\nu+5)/2} (x+r)^2 dx \\
		=& r^3 + 2\int_{0}^{\infty} ( 1 + x^2/\nu )^{-(\nu+1)/2} \frac{(x+r)^2}{ (1 + x^2/\nu)^2} dx \\
		\overset{(a)}\lesssim & r^3 + 2\int_{0}^{\infty} ( 1 + x^2/\nu )^{-(\nu+1)/2}  dx\\
		\lesssim & 1,
		\end{split}
	\end{equation*} where (a) is because $ \frac{(x+r)^2}{ (1 + x^2/\nu)^2} \lesssim 1$ for any $x \geq 0$ when $\nu$ is fixed and $r \lesssim 1$. In summary, we have shown that $\H^2(P_r, P_0)  \lesssim r^2$, and this finishes the proof for the t-distribution case.
	
	{\bf (Mollifier distribution)} Fix any $\beta > 0$,
	\begin{equation} \label{ineq:hellinger-mollifier}
		\begin{split}
			&\H^2(P_r, P_0) \\
			& = \int_{-1}^{1+r} (\sqrt{f(x-r)} - \sqrt{f(x)} )^2 d x \\
			& \overset{(a)}\lesssim r^2 \int_{-1}^{1+r} \sup_{\xi \in [0,r]} (x - \xi)^2  \exp\left( - \frac{1}{( 1- (x-\xi)^2 )^\beta} \right) ( 1- (x - \xi)^2 )^{-2 (\beta + 1)} \indi(|x - \xi| < 1) dx.
		\end{split}
	\end{equation} where (a) is by the Taylor expansion of $g(r) = \sqrt{f(x-r)}$. Now let us focus on the function $$h(x) = \exp\left( - \frac{1}{( 1- x^2)^\beta} \right) ( 1- x^2 )^{-2 (\beta + 1)} = \exp\left( - \frac{1}{( 1- x^2)^\beta} + 2(\beta + 1)\log\left( \frac{1}{1-x^2} \right) \right)$$ when $x \in (-1,1)$. Since $h(x)$ is symmetric around zero, to show its monotonicity, we just need to focus on $x \in (0,1)$. 
	If we denote $g(x) = - \frac{1}{( 1- x^2)^\beta} + 2(\beta + 1)\log\left( \frac{1}{1-x^2} \right)$, then 
	\begin{equation*}
		g'(x) = -2x\beta (1 - x^2)^{-(\beta + 1)} + 4x(\beta + 1) \frac{1}{1 - x^2}  = \frac{2x}{1-x^2} \left( 2(\beta + 1) - \frac{\beta}{(1-x^2)^\beta}  \right). 
	\end{equation*} Thus when $x \in (0,1)$, $g'(x) \geq 0 \Longleftrightarrow x^2 \leq 1 - (\beta /(2(\beta+1)) )^{1/\beta}$. Thus, $\sup_{x \in (-1,1)} h(x) \leq h\left( \sqrt{1 - (\beta /(2(\beta+1)) )^{1/\beta} } \right) \lesssim 1$. As a result of that, we can simplify \eqref{ineq:hellinger-mollifier}, and it yields:
	\begin{equation*}
		\begin{split}
			\H^2(P_r, P_0)  \lesssim r^2 \int_{-1}^{1+r} \sup_{\xi \in [0,r]} (x - \xi)^2  \indi(|x - \xi| < 1) dx \lesssim r^2.
		\end{split}
	\end{equation*} This finishes the proof for the mollifier case.

	{\bf (Bates distribution with $k \geq 3$)} Let us use notation $P_{\theta}^k$ to denote the Bates distribution with parameter $k$ and location $\theta$. We claim that if we can show for $k = 3$, $\H(P^3_r, P^3_0) \lesssim r^2$ when $r \leq c$ for a sufficiently small $c > 0$, then we are done. This is because for any fixed $k \geq 4$, if we let $X \sim P_{r}^k, X' \sim P_{\frac{k}{3} r}^3$, $Y \sim P_0^k$ and $Y' \sim P_0^3$, which are all independent to each other, and let $U_1, U_2, \ldots$ be other independent Uniform[$-1/2,1/2$] random variables, then $X \overset{d}= \frac{3X' + \sum_{i=1}^{k-3} U_i}{k}$ and $Y \overset{d}= \frac{3Y' + \sum_{i=1}^{k-3} U_i}{k}$, by the data processing inequality for general $f$-divergence (see Chapter 7 in \cite{polyanskiy2024information}) , we have
	\begin{equation*}
		\begin{split}
			\H^2(P^k_r, P^k_0) = \H^2(X, Y) \leq \H^2(X', Y') = \H^2(P^3_{ \frac{k}{3}r }, P^3_0) \overset{(a)}\lesssim \frac{k^2}{3^2} r^2 \lesssim r^2,
		\end{split}
	\end{equation*} where (a) is by the assumption. So next, we will focus on $k = 3$, and we also omit the superscript $3$ in $P_\theta^3$ for simplicity. From Lemma \ref{lm:bates-quantile}, the density of Bates distribution with $k = 3$ is given as follows:
	\begin{equation*}
		\begin{split}
			f(x) = \left\{\begin{array}{ll}
				\frac{3}{2} ( 3 x + 1.5)^2, & \textnormal{ if } x \in [-1/2, -1/6);\\
				\frac{3}{2} \left( (3x + 1.5)^2 - 3(3x + 0.5)^2 \right), & \textnormal{ if } x \in [-1/6, 1/6);\\
				\frac{3}{2} (3x - 1.5)^2, & \textnormal{ if } x \in [1/6, 1/2].
			\end{array} \right.
		\end{split}
	\end{equation*}  Thus, for $r \leq c$ with a sufficiently small $c$,
	\begin{equation} \label{ineq:Bates-hellinger-bound}
		\begin{split}
			&\H^2(P_r, P_0) \\
			& = \int_{-1/2}^{1/2 + r} (\sqrt{f(x-r)} - \sqrt{f(x)} )^2 d x \\
			& = 2 \int_{-1/2}^{r/2} (\sqrt{f(x-r)} - \sqrt{f(x)} )^2 d x \\
			& \lesssim \int_{-1/2}^{-1/2 + r} f(x)  d x + \int_{-1/2+r}^{-1/6} (3(x-r) + 1.5 - 3x - 1.5 )^2 dx \\
			& \quad + \int_{-1/6}^{-1/6 + r} \left( 3(x - r) + 1.5 - \sqrt{ (3x + 1.5)^2 - 3(3x + 0.5)^2} \right)^2   dx \\
			& \quad + \int_{-1/6+r}^{r/2} \left( \sqrt{ (3(x-r) + 1.5)^2 - 3(3(x-r) + 0.5)^2} - \sqrt{(3x + 1.5)^2 - 3(3x + 0.5)^2}  \right)^2 dx. 
		\end{split}
	\end{equation} It is easy to check the sum of the first two terms in \eqref{ineq:Bates-hellinger-bound} is bounded by $C r^2$ for some $C > 0$. Let us bound the last two terms in \eqref{ineq:Bates-hellinger-bound}.
	\begin{equation*}
		\begin{split}
			& \int_{-1/6}^{-1/6 + r} \left( 3(x - r) + 1.5 - \sqrt{ (3x + 1.5)^2 - 3(3x + 0.5)^2} \right)^2   dx \\
			& =  \int_{-1/6}^{-1/6 + r} \left( \frac{(3(x - r) + 1.5)^2 - (3x + 1.5)^2 + 3(3x + 0.5)^2}{3(x - r) + 1.5 + \sqrt{ (3x + 1.5)^2 - 3(3x + 0.5)^2}} \right)^2  dx \\
			& = \int_{-1/6}^{-1/6 + r} \left( \frac{ (-3r)(6x + 3 - 3 r) + 3(3x + 0.5)^2}{3(x - r) + 1.5 + \sqrt{ (3x + 1.5)^2 - 3(3x + 0.5)^2}} \right)^2  dx \\
			& \overset{(a)}\lesssim r^3,
		\end{split}
	\end{equation*} (a) is because, inside the integral, the numerator has order $r$ while the denominator has order $1$. Next, let us bound the final term:
	\begin{equation*}
		\begin{split}
			& \int_{-1/6+r}^{r/2} \left( \sqrt{ (3(x-r) + 1.5)^2 - 3(3(x-r) + 0.5)^2} - \sqrt{(3x + 1.5)^2 - 3(3x + 0.5)^2}  \right)^2 dx \\
			= & \int_{-1/6+r}^{r/2} \left( \frac{(3(x-r) + 1.5)^2 - 3(3(x-r) + 0.5)^2 - (3x + 1.5)^2 + 3(3x + 0.5)^2}{\sqrt{ (3(x-r) + 1.5)^2 - 3(3(x-r) + 0.5)^2} + \sqrt{(3x + 1.5)^2 - 3(3x + 0.5)^2}}  \right)^2 dx\\
			= & \int_{-1/6+r}^{r/2} \left( \frac{ (-3r)(6x + 3 - 3r) + 9r(6x + 1 - 3r) }{\sqrt{ (3(x-r) + 1.5)^2 - 3(3(x-r) + 0.5)^2} + \sqrt{(3x + 1.5)^2 - 3(3x + 0.5)^2}}  \right)^2 dx \\
			\overset{(a)}\lesssim & r^2,
		\end{split}
	\end{equation*} where (a) is because inside the integral, the numerator has order bounded by $r$ and the denominator has constant order.  In summary, we have shown $\H^2(P_r, P_0) \lesssim r^2$, and this finishes the proof for the Bates distribution.

	Now, we have finished the proof for the $1/\sqrt{n}$ part in the lower bound. To show the $\epsilon$ part in the lower bound, we can use the general result of Theorem 5.1 in \cite{chen2018robust}. In particular, to characterize the effect of contamination on the minimax rate, we need to compute the modulus of continuity:
	\begin{equation} \label{eq:modulus-of-contunuity}
		\omega(\epsilon, \bbR) = \sup \{ |\theta_1 - \theta_2|: \TV(P_{\theta_1}, P_{\theta_2}) \leq \epsilon/(1 - \epsilon); \theta_1, \theta_2 \in \bbR\}. 
	\end{equation} If we can show $\TV(P_{r}, P_{0}) \lesssim r$, then we have $\omega(\epsilon, \bbR) \gtrsim r$.  Notice that two times $\TV(P_{r}, P_{0})$ is equal to the non-overlapping area of the density of $P_0$ and its parallel shift $P_{r}$. Then it is clear that $\TV(P_{r}, P_{0}) \leq r f_0(0)$. Since in all of our examples, $f_0(0) \lesssim 1$, then $\TV(P_{r}, P_{0}) \lesssim r$. Finally, by combining Theorem 5.1 in \cite{chen2018robust} and the result we have proved in the first part, we get the minimax rate lower bound is of order $ \frac{1}{\sqrt{n}} + \epsilon$.  
\end{proof}

%%%%%%%%%%%%%%%%%%%%%%%%%%%%%%%
\section{Additional Lemmas}  \label{sec:proof-additional-lemmas}
%%%%%%%%%%%%%%%%%%%%%%%%%%%%%%
The following is the Dvoretzky-Kiefer-Wolfowitz-Massart inequality (DKW inequality) from \cite{massart1990tight}.
\begin{Lemma}\label{lm:DKW}
	Suppose $n$ is a positive integer. Let $X_1, \ldots, X_n$ be i.i.d. real-valued random variables drawn from a distribution with CDF $F(\cdot)$. Then for any $\alpha \in (0,1)$,
	\begin{equation*}
		\bbP\left( \sup_{t \in \bbR} |F_n(t) - F(t)| > \sqrt{ \frac{\log(2/\alpha)}{2n} } \right) \leq \alpha,
	\end{equation*}where $F_n(\cdot)$ is the empirical CDF.
\end{Lemma}

The following result is a version of the DKW inequality, which bounds the difference of the empirical CDF and population CDF by the variance term at each point. A version of this result has appeared in \citep[Lemma 3.2]{bartl2023variance} without an explicit constant in the failure probability term; here, we provide explicit constants.

\begin{Lemma} \label{lm:ratio-DKW} Let $F(\cdot)$ be the distribution function and $F_n(\cdot )$ be the empirical distribution function. Suppose $n \geq 400$. Then, for every $\frac{100}{n} \leq \Delta \leq 1/4$, with probability at least $1-8\exp(-\Delta n/100)$, for every $t\in\mathbb{R}$ that satisfies $F(t)(1-F(t)) \geq \Delta$ we have
\begin{equation*}
    |F_n(t)-F(t)|\leq\frac{F(t)(1-F(t))}{2}.
\end{equation*}
\end{Lemma}
\begin{proof} First, it is easy to check that the result is true when $\Delta = 1/4$ by using the ordinary DKW inequality \citep{massart1990tight}. Next, we focus on the case $\Delta < 1/4$ and we will prove this lemma by leveraging Lemma \ref{lm:reeve-corollary} and a union bound argument.

Let us first divide two sets $\{p\in [0,\frac{1}{2}]: \Delta\leq p(1-p) \leq 1/4 \}$ and $\{p\in[ \frac{1}{2},1]: \Delta\leq p(1-p) \leq 1/4\}$ into different subsets. 
\begin{equation*}
	\begin{split}
		&A_1 = \{p\in [0,1/2]: \frac{1}{8} \leq p(1-p) \leq \frac{1}{4} \}, \quad B_1 = \{p \in[ 1/2,1]: \frac{1}{8} \leq p(1-p) \leq \frac{1}{4} \}, \\
		&\cdots \\
		&A_i = \{p \in [0,1/2]: \frac{1}{2^{i+2}} \leq p(1-p) \leq \frac{1}{2^{i+1}} \}, \quad B_i = \{p \in[ 1/2,1]: \frac{1}{2^{i+2}} \leq p(1-p) \leq \frac{1}{2^{i+1}} \}, \\
		& \cdots \\
		&A_{ \lceil\log_2( \frac{1}{\Delta} ) \rceil - 2 } = \{p\in [0,1/2]: \frac{1}{2^{\lceil\log_2( \frac{1}{\Delta} ) \rceil}} \leq p(1-p) \leq \frac{1}{2^{\lceil\log_2( \frac{1}{\Delta} ) \rceil-1}} \}, \\
		& B_{\lceil\log_2( \frac{1}{\Delta} ) \rceil - 2} = \{p \in[ 1/2,1]: \frac{1}{2^{\lceil\log_2( \frac{1}{\Delta} ) \rceil}} \leq p(1-p) \leq \frac{1}{2^{\lceil\log_2( \frac{1}{\Delta} ) \rceil-1}} \}.
	\end{split}
\end{equation*} It is clear that $ \bigcup_{i=1}^{\lceil\log_2( \frac{1}{\Delta} ) \rceil - 2 }  A_i \supseteq \{p \in [0,1/2]:  \Delta\leq p(1-p) \leq 1/4 \}$ and $ \bigcup_{i=1}^{\lceil\log_2( \frac{1}{\Delta} ) \rceil - 2 }  B_i \supseteq \{p\in [1/2,1]:  \Delta\leq p(1-p) \leq 1/4 \}$. For each $A_i$ and $B_i$, we also define
\begin{equation*}
	\begin{split}
		\cI_{A_i} = \{t: F(t) \in A_i\} \quad \textnormal{ and }\quad \cI_{B_i} = \{t: F(t) \in B_i\}.
	\end{split}
\end{equation*} By definition, we have $\bigcup_{i=1}^{\lceil\log_2( \frac{1}{\Delta} ) \rceil - 2 }  (\cI_{A_i} \cup \cI_{B_i} ) \supseteq \{t: F(t) (1 -F(t)) \geq \Delta \} $. Thus,
\begin{equation} \label{ineq:DKW2-union-bound}
	\begin{split}
		&\bbP \left( F_n(t) - F(t) > \frac{1}{2} F(t)(1 - F(t)) \textnormal{ for some } t \in \bbR \textnormal{ such that } F(t)(1 - F(t)) \geq \Delta  \right) \\
		\leq & \bbP \left( F_n(t) - F(t) > \frac{1}{2} F(t)(1 - F(t)) \textnormal{ for some } t \in\bigcup_{i=1}^{\lceil\log_2( \frac{1}{\Delta} ) \rceil - 2 }  (\cI_{A_i} \cup \cI_{B_i} )  \right) \\
		\leq & \sum_{i=1}^{\lceil\log_2( \frac{1}{\Delta} ) \rceil - 2} \Big( \bbP \Big( F_n(t) - F(t) > \frac{1}{2} F(t)(1 - F(t)) \textnormal{ for some } t \in \cI_{A_i} \Big) \\
		& \quad \quad + \bbP \Big( F_n(t) - F(t) > \frac{1}{2} F(t)(1 - F(t)) \textnormal{ for some } t \in \cI_{B_i} \Big)  \Big).
	\end{split}
\end{equation} Next, we bound $\bbP \Big( F_n(t) - F(t) > \frac{1}{2} F(t)(1 - F(t)) \textnormal{ for some } t \in \cI_{A_i} \Big)$ via Lemma \ref{lm:reeve-corollary}. 
\begin{equation*}
	\begin{split}
		 & \bbP \Big( F_n(t) - F(t) > \frac{1}{2} F(t)(1 - F(t)) \textnormal{ for some } t \in \cI_{A_i} \Big) \\
		\leq & \bbP \Big( \sup_{t \in \cI_{A_i}} \{F_n(t) - F(t)\} > \frac{1}{2} \inf_{t \in \cI_{A_i} } {F(t)(1 - F(t))} \Big) \\
		\leq & \bbP \Big( \sup_{t \in \cI_{A_i}} \{F_n(t) - F(t)\} > \frac{1}{4} \sup_{t \in \cI_{A_i} } {F(t)(1 - F(t))} \Big) \\
		\overset{\textnormal{Lemma } \ref{lm:reeve-corollary} }\leq & \exp \left( - \frac{n \inf_{t \in \cI_{A_i}} \{F(t) (1 - F(t)) \} }{50} \right) \\
		\leq &  \exp( - \frac{n \cdot 2^{-(i+2)}}{50} ), 
	\end{split}
\end{equation*} where the last inequality is by the definition of $\cI_{A_i}$ and $A_i$. By symmetry, we also have 
\begin{equation*}
	\bbP \Big( F_n(t) - F(t) > \frac{1}{2} F(t)(1 - F(t)) \textnormal{ for some } t \in \cI_{B_i} \Big) \leq \exp( - \frac{n \cdot 2^{-(i+2)}}{50} ).
\end{equation*} Plug in the above two inequalities into \eqref{ineq:DKW2-union-bound}, we get 
\begin{equation*}
	\begin{split}
		& \bbP \left( F_n(t) - F(t) > \frac{1}{2} F(t)(1 - F(t)) \textnormal{ for some } t \in \bbR \textnormal{ such that } F(t)(1 - F(t)) \geq \Delta  \right) \\
		\leq & 2 \sum_{i=1}^{\lceil\log_2( \frac{1}{\Delta} ) \rceil - 2} \exp( - \frac{n \cdot 2^{-(i+2)}}{50} ) \\
		= & 2 \left( \exp\left(- \frac{n \cdot 2^{-\lceil\log_2( \frac{1}{\Delta} ) \rceil}}{50} \right) + \exp\left(- \frac{n \cdot 2^{-\lceil\log_2( \frac{1}{\Delta} ) \rceil + 1} }{50} \right) +\cdots + \exp(-n/400)   \right) \\
		\overset{(a)}\leq & 2 \left( \exp\left(- \frac{n \Delta}{100} \right) + \exp\left(- \frac{n \Delta \times 2}{100} \right) +\cdots  \right) \\
		\overset{(b)}\leq & 2 \left( \exp( - \frac{n\Delta}{100} ) +  \exp( - \frac{n\Delta}{100} - 1 ) + \cdots \right) \\
		\leq & 2 \frac{\exp(-n\Delta/100)}{1 - e^{-1}} \leq 4 \exp(-n\Delta/100),
	\end{split}
\end{equation*} where (a) is because $2^{-\lceil\log_2( \frac{1}{\Delta} ) \rceil} \geq  \frac{1}{2} \Delta $; in (b) we use the fact that $n\Delta \geq 100$. By symmetry, we also have 
\begin{equation*}
\begin{split}
	&\bbP \left( F(t) - F_n(t) > \frac{1}{2} F(t)(1 - F(t)) \textnormal{ for some } t \in \bbR \textnormal{ such that } F(t)(1 - F(t)) \geq \Delta  \right) \\
	\leq & 4 \exp(-n\Delta/100),
\end{split}
\end{equation*} and the conclusion follows by a union bound.
\end{proof}

The following lemma could be viewed as a corollary of Theorem 1 in \cite{reeve2024short}. 
\begin{Lemma}\label{lm:reeve-corollary}
	Let $F(\cdot)$ be the distribution function and $F_n(\cdot )$ be the empirical distribution function of $n$ samples. Given any interval $\cI \subseteq \bbR$, then 
	\begin{equation*}
		\bbP \left( \sup_{t \in \cI} \{ F_n(t) - F(t) \} > \frac{1}{4} \sup_{t \in \cI } \{F(t)(1 - F(t))\}   \right) \leq \exp\left( - \frac{n \inf_{t \in \cI } \{F(t)(1 - F(t))\} }{50} \right).
	\end{equation*}
\end{Lemma}
\begin{proof}
Let $\sigma^2_{\min} = \inf_{t \in \cI} \{F(t)(1 - F(t)) \}$ and $\sigma^2_{\max} = \sup_{t \in \cI} \{F(t)(1 - F(t)) \}$. Take $\varepsilon = \frac{\log(1/\delta)}{n}$ with $\delta = \exp(-n\sigma^2_{\min}/50)$. We also introduce the notation $\omega(p, \varepsilon)$, which was introduced in \cite{reeve2024short} and will be used later. For each $p \in (0,e^{-\varepsilon}]$, there is a unique value $\eta \in [0,1-p]$ with $\textrm{KL}(p+\eta\|p) = \epsilon$ and we denote this unique value by $\omega(p, \varepsilon)$. We extend $\omega$ to $[0,1] \times [0,\infty)$ by letting $\omega(p, \varepsilon):=1-p$ if $p \in (e^{-\varepsilon },1]$; we also let $\omega(p, \varepsilon ):=0$ whenever $p = 0$, so that $p \mapsto \omega(p, \varepsilon)$ is continuous.

Given any $t \in \cI$, let $p = F(t)$ and $\sigma^2(p) = p(1-p) \in [\sigma^2_{\min}, \sigma^2_{\max}]$, we have $\textrm{KL}(p + \eta\|p) < \varepsilon $ for any $\eta \in (0, \omega(p, \varepsilon ) )$. Following the same proof as in Corollary 1 of \cite{reeve2024short}, we have
\begin{equation*}
\begin{split}
	\eta &< \left( \sqrt{4\sigma^2(p) \rho^2_{\varepsilon} + \sigma^2(\rho_{\varepsilon}) } + (1-2p) \sigma(\rho_{\varepsilon }) \right) \sqrt{\frac{\varepsilon}{2}} \\
	& \leq 2(\sigma(p) + \sigma(\rho_{\varepsilon } )) \sqrt{\frac{\varepsilon}{2}} \\
	& \leq 2(\sigma(p) +  \frac{3\sqrt{2} \sqrt{\varepsilon } }{9} ) \sqrt{\frac{\varepsilon}{2}} \\
	& \overset{(a)}\leq \frac{16}{75} \sigma^2(p) \leq \frac{1}{4} \sigma^2(p),
\end{split}
\end{equation*} where $\rho_{\varepsilon } = 9/(9+ 2 \varepsilon )$ and in (a), we use the fact $\varepsilon =  \frac{\sigma^2_{\min}}{50} \leq \frac{\sigma^2(p)}{50} $.

Since this result holds for all $t\in \cI$ and $\eta \in (0, \omega(p, \varepsilon ) )$, we have
\begin{equation*}
	\begin{split}
		\sup_{t \in \cI } \omega(F(t), \varepsilon ) \leq \frac{1}{4}\sup_{ \sigma^2 \in [\sigma^2_{\min}, \sigma^2_{\max}] } \sigma^2 = \frac{1}{4} \sup_{t \in \cI } \{F(t)(1 - F(t))\}.
	\end{split}
\end{equation*} Then by Theorem 1 in \cite{reeve2024short}, we have
	\begin{equation*}
		\bbP \left( \sup_{t \in \cI} \{ F_n(t) - F(t) \} > \frac{1}{4} \sup_{t \in \cI } \{F(t)(1 - F(t))\}   \right) \leq \delta = \exp\left( - \frac{n \inf_{t \in \cI } \{F(t)(1 - F(t))\} }{50} \right).
	\end{equation*}
\end{proof}

\begin{Lemma}\label{lm:normal-quantile}
	Let $\Phi(\cdot)$ be the CDF of the standard Gaussian distribution. Given some $q < 1/2$, $x = \Phi^{-1}(1 - q)$. Suppose $x \geq 2$, then
	\begin{equation*}
	\begin{split}
		\frac{1}{2} \sqrt{\log( 1/q ) } 	\leq x \leq  \sqrt{2 \log( 1/q ) }\quad  \text{ and }\quad \frac{1}{\sqrt{2\pi}} \exp(-x^2) \leq q \leq \frac{1}{2\sqrt{2\pi}} \exp(-x^2/2).
	\end{split}
	\end{equation*} 
\end{Lemma}
\begin{proof}
	By the definition of $x$, we have $\int_x^\infty \frac{1}{\sqrt{2\pi}} \exp(-t^2/2) dt = q$. In addition, for any $x > 0$ (see Theorem 1.2.3 in \cite{durrett2019probability}),
	\begin{equation} \label{ineq:gaussian-integration-ineq}
		(\frac{1}{x} - \frac{1}{x^3}) \exp(-x^2/2) \leq \int_{x}^{\infty} \exp( - t^2/2 ) dt \leq \frac{1}{x} \exp(-x^2/2).
	\end{equation}
	Since $x \geq 2$, $\frac{1}{x} - \frac{1}{x^3} \geq \frac{1}{x}( 1- \frac{1}{4} ) = \frac{3}{4x}$ and $\frac{1}{x}\leq \frac{1}{2}$. At the same time, it is easy to check $\frac{3}{4x} \exp(-x^2/2) \geq \exp(-x^2)$ for all $x \geq 2$. Thus, we have
	\begin{equation*}
		\frac{1}{\sqrt{2\pi}} \exp(-x^2) \leq q \leq \frac{1}{2\sqrt{2\pi}} \exp(-x^2/2)
	\end{equation*} holds for any $x \geq 2$. Thus
	\begin{equation*}
 	\frac{1}{2} \sqrt{\log( 1/q ) } \overset{(a)}\leq \sqrt{\log(1/ (\sqrt{2\pi}q))}	\leq x \leq \sqrt{2 \log( 1/(2 \sqrt{2\pi} q) ) } \leq \sqrt{2 \log( 1/q ) },
	\end{equation*} where in (a) we use the fact $q \leq 0.023$ when $x \geq 2$, so the inequality always holds when $x \geq 2$. 
\end{proof}

\begin{Lemma} \label{lm:generalized-gaussian-quantile} Let $X$ follows generalized Gaussian with shape parameter $\beta > 1$, i.e., the PDF is $f(x) = I_\beta \exp(-|x|^\beta)$. Then for any $x > 0$,
\begin{equation*}
	\left( 1 - \frac{\beta-1}{\beta} x^{-\beta} \right)  \frac{I_\beta}{\beta x^{\beta - 1} }\exp(-x^\beta) \leq \bbP(X \geq x) \leq \frac{I_\beta}{\beta x^{\beta - 1} }\exp(-x^\beta).
\end{equation*} Let $F^{-1}(\cdot)$ denote the inverse CDF of $X$. Then there exists a sufficiently small $c> 0$ depending on $\beta$ only such that for any $ q \in [0, c]$, we have 
\begin{equation*}
	\left( \log(1/q)/2 \right)^{1/\beta} \leq F^{-1}(1 - q) \leq \left( \log(1/q) \right)^{1/\beta}. 
\end{equation*}
\end{Lemma}
\begin{proof}
	First,
	\begin{equation} \label{ineq:generalized-gaussian-tail-upper-bound}
		\begin{split}
			\bbP(X \geq x) & = \int_{x}^\infty I_\beta \exp( - t^\beta ) dt = \left( \frac{I_\beta}{- \beta t^{\beta -1} } \exp( - t^\beta ) \right)\Big|^\infty_x - \frac{\beta - 1}{\beta} I_\beta \int_x^\infty \exp(- t^\beta) t^{-\beta} dt \\
			 & = \frac{I_\beta}{\beta x^{\beta - 1} } \exp(- x^\beta) - \frac{\beta - 1}{\beta} I_\beta \int_x^\infty \exp(- t^\beta) t^{-\beta} dt  \leq \frac{I_\beta}{\beta x^{\beta - 1} } \exp(- x^\beta),
		\end{split}
	\end{equation} where the second equality is by integration by parts. Next, we show the lower bound $\bbP(X \geq x)$ by applying the result in \eqref{ineq:generalized-gaussian-tail-upper-bound}:
	\begin{equation} \label{ineq:generalized-gaussian-tail-lower-bound}
		\begin{split}
			\bbP(X \geq x) & \overset{ \eqref{ineq:generalized-gaussian-tail-upper-bound} } =\frac{I_\beta}{\beta x^{\beta - 1} } \exp(- x^\beta) - \frac{\beta - 1}{\beta}  \int_x^\infty I_\beta \exp(- t^\beta) t^{-\beta} dt \\
			& \geq  \frac{I_\beta}{\beta x^{\beta - 1} } \exp(- x^\beta) - \frac{\beta- 1}{\beta x^\beta} \int_x^\infty I_\beta \exp(- t^\beta) dt \\
			& \overset{ \eqref{ineq:generalized-gaussian-tail-upper-bound} } \geq \frac{I_\beta}{\beta x^{\beta - 1} } \exp(- x^\beta) - \frac{\beta- 1}{\beta x^\beta}  \frac{I_\beta}{\beta x^{\beta - 1} } \exp(- x^\beta).
		\end{split}
	\end{equation}
	
	When $c$ is a sufficiently small constant, then for any $q \in [0,c]$, $x = F^{-1}(1 - q) \geq C$ for some large constant such that $\frac{\beta- 1}{\beta x^\beta} \leq 1/2$, $\frac{I_\beta}{\beta x^{\beta - 1} } \leq 1$ and $ \exp(- x^\beta) \leq \frac{I_\beta}{2\beta x^{\beta - 1} }$. Thus
	\begin{equation*}
		 \exp(- 2 x^\beta)  \leq \frac{I_\beta}{2\beta x^{\beta - 1} } \exp(- x^\beta) \overset{ \eqref{ineq:generalized-gaussian-tail-lower-bound} }\leq q = \int_{x}^\infty I_\beta \exp( - t^\beta ) dt  \overset{ \eqref{ineq:generalized-gaussian-tail-upper-bound} }\leq \frac{I_\beta}{\beta x^{\beta - 1} } \exp(- x^\beta) \leq  \exp( - x^\beta ).
	\end{equation*} Thus,
	\begin{equation*}
		\left( \log(1/q)/2 \right)^{1/\beta} \leq F^{-1}(1 - q) \leq \left( \log(1/q) \right)^{1/\beta}.
	\end{equation*}
\end{proof}

\begin{Lemma} \label{lm:mollifier-quantile} Let $X$ follows Mollifier with shape parameter $\beta > 0$, i.e., the PDF is $f(x) = I_\beta \exp\left( -\frac{1}{(1-x^2)^\beta}\right) \indi(|x| < 1)$. Then for any $x \in (0,1)$,
\begin{equation*}
\begin{split}
	\bbP(X \geq x) &\leq I_\beta \frac{(1 - x^2)^{\beta + 1}}{2x \beta}   \exp\left( -\frac{1}{(1-x^2)^\beta}\right) \\
	\bbP(X \geq x) &\geq \left( 1 -  (1 - x^2)^{\beta} \left( \frac{1 - x^2}{2 \beta x^2} + \frac{\beta + 1}{\beta} \right) \right)  I_\beta \frac{(1 - x^2)^{\beta + 1}}{2x \beta}   \exp\left( -\frac{1}{(1-x^2)^\beta}\right).
\end{split}
\end{equation*} Let $F^{-1}(\cdot)$ denote the inverse CDF of $X$. There exists a sufficiently small constant $c > 0$ depending on $\beta$ only such that for any $ q \in [0, c]$, we have
\begin{equation*}
	\left( 1/\log(1/q) \right)^{1/\beta} \leq 1 - \left( F^{-1}(1 - q)\right)^2 \leq \left( 2/\log(1/q) \right)^{1/\beta}. 
\end{equation*}
\end{Lemma}
\begin{proof}
	First,
	\begin{equation} \label{ineq:mollifier-tail-upper-bound}
		\begin{split}
			 & \bbP(X \geq x) \\
			& = I_\beta \int_{x}^1  \exp\left( -\frac{1}{(1-t^2)^\beta}\right) dt \\
			&   = \left( I_\beta \frac{(1 - t^2)^{\beta + 1}}{-2t \beta}   \exp\left( -\frac{1}{(1-t^2)^\beta}\right) \right)\Big|_{x}^1 - \int_{x}^1 I_\beta (1 - t^2)^\beta \left( \frac{1 - t^2}{2 \beta t^2} + \frac{\beta + 1}{\beta} \right)   \exp\left( -\frac{1}{(1-t^2)^\beta}\right) dt \\
			& = I_\beta \frac{(1 - x^2)^{\beta + 1}}{2x \beta}   \exp\left( -\frac{1}{(1-x^2)^\beta}\right) - \int_{x}^1  I_\beta  (1 - t^2)^\beta \left( \frac{1 - t^2}{2 \beta t^2} + \frac{\beta + 1}{\beta} \right)   \exp\left( -\frac{1}{(1-t^2)^\beta}\right) dt  \\
			& \leq I_\beta \frac{(1 - x^2)^{\beta + 1}}{2x \beta}   \exp\left( -\frac{1}{(1-x^2)^\beta}\right).
		\end{split}
	\end{equation} Next,
	\begin{equation}\label{ineq:mollifier-tail-lower-bound}
		\begin{split}
			 & \bbP(X \geq x) \\
			\overset{ \eqref{ineq:mollifier-tail-upper-bound} }= & I_\beta \frac{(1 - x^2)^{\beta + 1}}{2x \beta}   \exp\left( -\frac{1}{(1-x^2)^\beta}\right) - \int_{x}^1  I_\beta  (1 - t^2)^\beta \left( \frac{1 - t^2}{2 \beta t^2} + \frac{\beta + 1}{\beta} \right)   \exp\left( -\frac{1}{(1-t^2)^\beta}\right) dt \\
			\geq & I_\beta \frac{(1 - x^2)^{\beta + 1}}{2x \beta}   \exp\left( -\frac{1}{(1-x^2)^\beta}\right) - (1 - x^2)^{\beta} \left( \frac{1 - x^2}{2 \beta x^2} + \frac{\beta + 1}{\beta} \right) \int_{x}^1 I_\beta \exp\left( -\frac{1}{(1-t^2)^\beta}\right) dt \\
			\overset{ \eqref{ineq:mollifier-tail-upper-bound} } \geq & I_\beta \frac{(1 - x^2)^{\beta + 1}}{2x \beta}   \exp\left( -\frac{1}{(1-x^2)^\beta}\right) \left( 1 -  (1 - x^2)^{\beta} \left( \frac{1 - x^2}{2 \beta x^2} + \frac{\beta + 1}{\beta} \right) \right).
		\end{split}
	\end{equation}
	
		When $c$ is a sufficiently small constant depending on $\beta$ only, then for any $q \in [0,c]$, $1 -x = 1-  F^{-1}(1 - q) \leq c'$ for a small constant $c'$ such that $(1 - x^2)^{\beta} \left( \frac{1 - x^2}{2 \beta x^2} + \frac{\beta + 1}{\beta} \right) \leq 1/2$, $I_\beta \frac{(1 - x^2)^{\beta + 1}}{2x \beta} \leq 1$ and $\exp\left( -\frac{1}{(1-x^2)^\beta}\right) \leq \frac{1}{2} I_\beta \frac{(1 - x^2)^{\beta + 1}}{2x \beta} $. Thus
	\begin{equation*}
	\begin{split}
		q &= I_\beta \int_{x}^1  \exp\left( -\frac{1}{(1-t^2)^\beta}\right) dt  \overset{ \eqref{ineq:mollifier-tail-upper-bound} }\leq I_\beta \frac{(1 - x^2)^{\beta + 1}}{2x \beta}   \exp\left( -\frac{1}{(1-x^2)^\beta}\right) \leq \exp\left( -\frac{1}{(1-x^2)^\beta}\right),\\
		q & \overset{ \eqref{ineq:mollifier-tail-lower-bound} } \geq \left( 1 -  (1 - x^2)^{\beta} \left( \frac{1 - x^2}{2 \beta x^2} + \frac{\beta + 1}{\beta} \right) \right)  I_\beta \frac{(1 - x^2)^{\beta + 1}}{2x \beta}   \exp\left( -\frac{1}{(1-x^2)^\beta}\right) \\
		& \geq \frac{1}{2} I_\beta \frac{(1 - x^2)^{\beta + 1}}{2x \beta}   \exp\left( -\frac{1}{(1-x^2)^\beta}\right) \\
		& \geq \exp\left( -\frac{2}{(1-x^2)^\beta}\right).
	\end{split}
	\end{equation*}  Thus,
	\begin{equation*}
		\left( 1/\log(1/q) \right)^{1/\beta} \leq 1 -\left(  F^{-1}(1 - q)\right)^2 \leq \left( 2/\log(1/q) \right)^{1/\beta}.
	\end{equation*}
\end{proof}

\begin{Lemma} \label{lm:bates-quantile} Let $X$ follows Bates distribution with positive interger parameter $k \geq 1$, i.e., $X = \sum_{j=1}^k U_j$ where $U_j \overset{iid}\sim \textnormal{Uniform}[-1/2,1/2]$ and its PDF is $$f(x) = \frac{k}{(k-1)!} \sum_{0 \leq j \leq \lfloor k(x +1/2) \rfloor } (-1)^j { k \choose j } ( k (x + 1/2) - j )^{k-1}$$ by \citep[Section 29.6]{johnson1995continuous}. In particular, when $x \in (-1/k + 1/2, 1/2]$, its density is $f(x) = \frac{k^k}{(k-1)!} (1/2 - x)^{k-1}$. Then for any $x \in (-1/k + 1/2, 1/2]$,
\begin{equation*}
\begin{split}
	\bbP(X \geq x) & = \frac{k^{k-1}}{(k-1)!} ( 1/2 - x )^k.
\end{split}
\end{equation*} Let $F^{-1}(\cdot)$ denote the inverse CDF of $X$, then for any $q \leq 1/(k(k-1)!)$,
\begin{equation} \label{ineq:bates-quantile}
	1/2 - F^{-1}(1 - q)= \left(\frac{q (k-1)!}{k^{k-1}}\right)^{1/k}. 
\end{equation}
\end{Lemma}
\begin{proof}
	First, when $x \in [-1/2, -1/2 + 1/k)$, it is easy to check $f(x) = \frac{k^k}{(k-1)!} (x - (-1/2))^{k-1}$. Then by symmetry, we have for $x \in (-1/k + 1/2, 1/2]$, its density is $f(x) = \frac{k^k}{(k-1)!} (1/2 - x)^{k-1}$. In addition,
	\begin{equation}\label{ineq:bates-tail}
		\begin{split}
			\bbP(X \geq x) = \int_x^{1/2} \frac{k^k}{(k-1)!} (1/2 - t)^{k-1} dt = \frac{k^{k-1}}{(k-1)!} ( 1/2 - x )^k.
		\end{split}
	\end{equation} Finally, when $q \leq 1/(k(k-1)!)$, we have $F^{-1}(1 -q) \in (-1/k + 1/2, 1/2]$, then \eqref{ineq:bates-quantile} is obtained by inverting \eqref{ineq:bates-tail}.
\end{proof}

\begin{Lemma} \label{lm:t-epsilon-bound}
	For any $\epsilon_{\max} < 0.99$, positive integer $n$ and $\alpha \in (0,1)$ such that $n/\log(1/\alpha)$ is sufficiently large, then for all $\epsilon \in [0, \epsilon_{\max}]$,
	\begin{equation*}
		1/t_{\epsilon} \asymp  \frac{1}{\sqrt{\log(1/\epsilon)}} + \frac{1}{\sqrt{\log( n/(\log(1/\alpha)) )}} ,
	\end{equation*} where $t_{\epsilon} = \Phi^{-1}\left( 1- \epsilon - \sqrt{ \log (2/\alpha)/(2n) }  \right)$.
\end{Lemma}
\begin{proof} We divide the proof into two cases. First, when $\epsilon \in [0.01, \epsilon_{\max}]$, $1/t_{\epsilon}$ is of constant level, so the result follows. 
	At the same time, when $\epsilon \in [0,  0.01]$, we have $t_{\epsilon} \geq 2$ when $n/(\log(1/\alpha))$ is sufficiently large, by Lemma \ref{lm:normal-quantile}, we have
	\begin{equation*}
	\begin{split}
		t_{\epsilon} & \asymp  \sqrt{ \log\left( \frac{1}{\epsilon + \sqrt{ \log (2/\alpha)/(2n) }}  \right)  } \\
		& \asymp \sqrt{ \log\left( \frac{1}{\epsilon}  \right)} \wedge \sqrt{ \log\left(\frac{n}{ \log \left( \frac{1}{\alpha} \right)} \right) }.
	\end{split}
	\end{equation*} Thus, $1/t_{\epsilon} \asymp \frac{1}{\sqrt{\log(1/\epsilon)}} + \frac{1}{\sqrt{\log( n/(\log(1/\alpha)) )}} $.
\end{proof}

\end{sloppypar}

\end{document}